\documentclass{amsart}
\usepackage{amsthm}
\usepackage{amsmath}
\usepackage{amsfonts}
\usepackage{amssymb}
\usepackage{mathrsfs}
\usepackage{tikz-cd}
\usepackage{caption}
\usepackage[utf8]{inputenc}
\usepackage[left=3cm, right=3cm]{geometry}
\usepackage{hyperref}

\theoremstyle{plain}% default
\newtheorem{thm}{Theorem}[section]
\newtheorem{lem}[thm]{Lemma}
\newtheorem{prop}[thm]{Proposition}
\newtheorem{cor}[thm]{Corollary}

\newtheorem{problem}{Problem}
\newtheorem{question}[problem]{Question}
\newtheorem{conj}[problem]{Conjecture}

\newtheorem{thmintro}{Theorem}

\theoremstyle{definition}
\newtheorem{defn}[thm]{Definition}
\newtheorem{rem}[thm]{Remark}
\newtheorem{rmk}[thm]{Remark}
\newtheorem{cst}[thm]{Construction}

\newcommand{\mc}{\mathcal}
\newcommand{\mf}{\mathfrak}
\newcommand{\mscr}{\mathscr}
\newcommand{\sq}{\subseteq}
\newcommand{\ra}{\rightarrow}
\newcommand{\eps}{\epsilon}
\newcommand{\s}{\sigma}
\newcommand{\R}{\mathbb{R}}

\newcommand{\Z}{\mathbb{Z}}
\renewcommand{\o}{\circ}
\newcommand{\wh}{\widehat}

\newcommand{\cH}{\check{H}}
\newcommand{\x}{\times}
\DeclareMathOperator{\supp}{supp}
\DeclareMathOperator{\diam}{diam}
\DeclareMathOperator{\fin}{fine}
\DeclareMathOperator{\disc}{disc}
\DeclareMathOperator{\cone}{cone}
\renewcommand{\hom}{\mathrm{Hom}}
\DeclareMathOperator{\rk}{rk}
\usepackage{xcolor}

\title{On the \v{C}ech cohomology of Morse boundaries}

\author{Elia Fioravanti}
    \address{Karlsruhe Institute of Technology, Karlsruhe, Germany}
	\email{elia.fioravanti@kit.edu}
	\thanks{Fioravanti thanks the Max Planck Institute for Mathematics in Bonn for their hospitality and financial support while this work was being completed. He was also partially supported by Emmy Noether grant 515507199 of the Deutsche Forschungsgemeinschaft (DFG)}

\author{Annette Karrer}
\address{McGill University, Montreal, Canada}
	\email{annette.karrer@mcgill.ca}
 \thanks{Karrer was partially supported by CIRGET and CRM and the Israel Science Foundation (grant no. 1562/19)}
\author{Alessandro Sisto}
	\address{Maxwell Institute and Department of Mathematics, Heriot-Watt University, Edinburgh, UK}
	\email{a.sisto@hw.ac.uk}

 \author{Stefanie Zbinden}
    \address{Maxwell Institute and Department of Mathematics, Heriot-Watt University, Edinburgh, UK}
	\email{sz2020@hw.ac.uk}

\begin{document}

\begin{abstract}
    We consider cusped hyperbolic $n$--manifolds, and compute \v{C}ech cohomology groups of the Morse boundaries of their fundamental groups. In particular, we show that the reduced \v{C}ech cohomology with real coefficients vanishes in dimension at most $n-3$ and does not vanish in dimension $n-2$. A similar result holds for relatively hyperbolic groups with virtually nilpotent peripherals and Bowditch boundary homeomorphic to a sphere; these include all non-uniform lattices in rank--$1$ simple Lie groups.
\end{abstract}

\maketitle

\tableofcontents

\addtocontents{toc}{\protect\setcounter{tocdepth}{1}}
\section*{Introduction}

The Morse boundary $\partial_* X$ of a metric space $X$, as introduced in \cite{CS:contracting,C:Morse}, is a topological space that encodes the ``hyperbolic-like'' directions of $X$. It extends to all metric spaces Gromov's classical concept of boundary for hyperbolic spaces \cite{Gromov-hyp}: when $X$ is hyperbolic, $\partial_* X$ is naturally identified with the Gromov boundary of $X$. The Morse boundary also enjoys another fundamental property of Gromov boundaries, namely quasi-isometric invariance: quasi-isometric metric spaces have homeomorphic Morse boundaries. As a consequence, the Morse boundary $\partial_*G$ of a finitely generated group $G$ is well-defined up homeomorphism. Quasi-isometry invariance of $\partial_*G$ was in fact used in \cite{CS:contracting} to distinguish certain right-angled Coxeter groups up to quasi-isometry.

In view of this, it is interesting to study topological features of Morse boundaries. Since $\partial_*X$ is not even metrisable when $X$ is non-hyperbolic \cite{CD:stable}, it was thought until rather recently that ``fully understanding'' the topology of Morse boundaries of non-hyperbolic spaces would be an impossible feat. In contrast to this prediction, we show that in some important cases one can compute a rather sophisticated and powerful topological invariant, namely \v{C}ech cohomology:

\begin{thmintro}
\label{thm:main_intro}
    Let $G$ be the fundamental group of a cusped hyperbolic $n$--manifold for $n\geq 3$. Then
\begin{enumerate}
    \item $\cH^i(\partial_*G,\R)=\{0\}$ for $1\leq i\leq n-3$; 
    \item $\cH^{n-2}(\partial_*G,\R)$ is infinite dimensional,
\end{enumerate}
where $\cH^*$ denotes reduced \v{C}ech cohomology.
\end{thmintro}

The theorem follows from Theorem~\ref{thm:general_intro} below, which is more general in two ways. First, it covers all relatively hyperbolic groups with virtually nilpotent peripherals and Bowditch boundary homeomorphic to a sphere; these include all non-uniform lattices in rank--$1$ simple Lie groups. Secondly, it allows for more general coefficients. 

We note that, for $i=0$, our result would be equivalent to $\partial_*G$ being connected, but one can even prove that $\partial_*G$ is path connected in the setting of Theorem~\ref{thm:main_intro}. This follows from results in \cite{Mackay-Sisto}, as we argue in Theorem \ref{thm:path_conn}. We will actually further develop various tools from that paper and use them extensively in our arguments.

Our hope is that the techniques we develop to prove the main theorem will be useful to study Morse boundaries of other groups and to distinguish quasi-isometry classes of groups via their Morse boundary; see below for some concrete candidates.

Theorem~\ref{thm:main_intro} also provides the first examples of relatively hyperbolic groups with high-dimensional Morse boundary not witnessed by an ``obvious'' high-dimensional hyperbolic subgroup. For instance, for $n\geq 4$, it is still an open problem whether the fundamental group of a (cusped) hyperbolic $n$--manifold always admits codimension--$1$ subgroups. In relation to this, we mention that Cordes previously showed the existence of high-dimensional spheres in Morse boundaries of mapping class groups \cite[Theorem~4.2]{C:Morse}, using quasi-convex copies of $\mathbb{H}^n$ in the thick part of Teichm\"uller space constructed by Leininger and Schleimer \cite{Leininger-Schleimer}. It is open whether any of these corresponds to a subgroup of the mapping class group, and some certainly do not.

Our main theorem is in a similar line of research as the papers \cite{charney2019complete, Zbi1, Zbi2}, where Morse boundaries of certain groups, including $3$--manifold groups, are fully described. A similar goal seems out of reach in our case (except when $n=3$, as shown in \cite{charney2019complete}), but we also think of our theorem as a proof of concept: even when the Morse boundary is not ``fully describable'', powerful topological invariants can sometimes still be computed.

As a final note, there is a different topology on Morse boundaries \cite{CM:topology} which has the advantage of being metrisable. However, we do not know whether analogues of our main theorem (or of the other results hinted at above) hold for that topology.

\subsection{Distinguishing groups up to quasi-isometry}

To provide further motivation for the study of \v{C}ech cohomology of Morse boundaries, we now construct a collection of groups that can conjecturally be distinguished up to quasi-isometry using this invariant, but where other usual quasi-isometry invariants do not seem to be helpful. We believe that the results and techniques in this paper go a long way towards proving the conjecture, and we point out below what is missing.

Fix for every integer $n\geq 1$ the fundamental group $G_n$ of an $(n+2)$--dimensional cusped hyperbolic manifold. Let $\bar n$ be a finite collection of integers at least 1. Following \cite[Section 6]{GS:smallcanc}, one can construct a certain finitely generated, infinite group $G(\bar n)$ with several properties, including that $G(\bar n)$ contains copies of all $G_n$ with $n\in\bar n$ as hyperbolically embedded subgroups, but $G(\bar n)$ is not (non-trivially) relatively hyperbolic. 

The following conjecture would allow one to distinguish the quasi-isometry classes of the various $G(\bar n)$:

\begin{conj}
\label{conj:smallcanc}
$\cH^i(\partial_* G(\bar n),\R)$ is non-trivial for $i\in \bar n$ and trivial for $i\notin \bar n$, with $1< i\leq \max(\bar n)$.
\end{conj}

The intuition behind the conjecture is that $\partial_* G(\bar n)$ should contain copies of the $\partial_* G_n$ for $n\in\bar n$, and those should contribute cohomology in the appropriate degrees, but $\partial_* G(\bar n)$ should otherwise be $1$--dimensional. Proving the conjecture requires understanding Morse boundaries of small cancellation free products (as $\bar G(n)$ is of this form). We believe this task to be of independent interest, but also not within the scope of this paper, so we leave it for future work. In fact, more work is still required in the context of classical small cancellation groups (which have non-trivial Morse boundary by \cite{GS:smallcanc} and \cite{S:hypemb}, except in trivial cases).

We could not think of methods that would distinguish the quasi-isometry classes of the $G(\bar n)$ as $\bar n$ varies, other than (variations of) the conjecture. For example, the groups are not relatively hyperbolic, which rules out a lot of tools. Similarly, they are not CAT(0), hierarchically hyperbolic, etc, as they are not finitely presented. Also, their asymptotic dimension presumably only depends on $\max(\bar n)$, as should the ``stable asymptotic dimension'' introduced in \cite{CH:stable_asdim}.

\subsection{Questions}

There are various natural questions that arise at this point. Let us first consider those related to the same groups as in Theorem \ref{thm:main_intro}.

\begin{question}
For $G$ as in Theorem \ref{thm:main_intro}, does $\cH^i(\partial_*G,\R)$ vanish for all $i>n-2$?
\end{question}

We believe that the answer is yes, and we plan future work on covering dimension of Morse boundaries that should imply this. Related to this, we note that it is easy to show that $\cH^i(K,\Z)=\{0\}$ for all $i>n-2$ and all compact subsets $K\sq\partial_*G$. Indeed, any such $K$ is homeomorphic to a closed nowhere-dense subset of $S^{n-1}$, and the latter all have covering dimension at most $n-2$ (for instance, by \cite[Theorem~19]{Schultz}).

As mentioned above, Theorem \ref{thm:main_intro} holds for more general coefficients. However, certain natural coefficients are not covered, especially for the non-vanishing part, so in particular we can ask:

\begin{question}
For $G$ as in Theorem \ref{thm:main_intro}, is $\cH^{n-2}(\partial_*G,\Z)$ non-zero?
\end{question}

We believe that answering this question requires deeper understanding of discrete cycles (see discussion on this notion below), and hence working towards it could lead to interesting discoveries.

Moving on, it is natural to ask what the \v{C}ech cohomology is for Morse boundaries of other groups of interest. For example:

\begin{problem}
Compute \v{C}ech cohomology groups $\cH^{i}(\partial_*G,\R)$ for $G$ a mapping class group of a finite-type surface.
\end{problem}

As mentioned above our hope is that our strategy will lead to the use of Morse boundaries as effective quasi-isometry invariants. More generally than in Conjecture \ref{conj:smallcanc}, we then ask:

\begin{problem}
    Use topological invariants of Morse boundaries such as \v{C}ech cohomology to distinguish quasi-isometry classes of groups.
\end{problem}

Another natural class of groups to look into regarding this problem is that of right-angled Coxeter groups.
%We also think that this could be a fruitful approach to study small cancellation quotients of free products, perhaps looking at the \v{C}ech cohomology of specific subspaces of their Morse boundary.

Finally, looking at \v{C}ech cohomology of Morse boundaries was also partly inspired by the fact that, for hyperbolic groups, the \v{C}ech cohomology of the Gromov boundary has been fruitfully studied and contains important information about the group itself, see e.g.\ \cite{BestvinaMess}. Therefore, in a different direction than the other questions, it is natural to ask:

\begin{problem}
    What kind of information about a group can be gleaned from the \v{C}ech cohomology of its Morse boundary?
\end{problem}

\subsection{Outline}

We deduce Theorem~\ref{thm:main_intro} as a consequence of the following more general result. 

\begin{thmintro}\label{thm:general_intro}
Let $(G,\mc{P})$ be a relatively hyperbolic pair. Suppose that $\mc{P}$ is a nonempty finite collection of finitely generated, virtually nilpotent subgroups, and that the Bowditch boundary is homeomorphic to the $(n-1)$--sphere, for some $n\geq 3$. Then
\begin{enumerate}
    \item $\cH^i(\partial_*G,\mc{R})=\{0\}$ for $1\leq i\leq n-3$ and any principal ideal domain $\mc{R}$; 
    \item $\cH^{n-2}(\partial_*G,\mc{F})$ is infinite dimensional for any field $\mc{F}$ of characteristic $0$.
\end{enumerate}
\end{thmintro}

The assumptions of Theorem~\ref{thm:general_intro} are satisfied by all fundamental groups of non-compact, finite-volume, complete Riemannian manifolds with pinched negative sectional curvature (by the Margulis lemma). These include non-uniform lattices in rank--$1$ simple Lie groups.

Theorem~\ref{thm:general_intro} is proved by combining Theorem \ref{thm:vanishing} and Theorem \ref{thm:non_vanishing} in the main body of the text, which deal with the vanishing and non-vanishing part, respectively.

Our approach to these results relies heavily on the notion of \emph{discrete chains}, as described in Subsection~\ref{subsec:discrete_chains}. Indeed, we will translate vanishing and non-vanishing of \v{C}ech cohomology for ``sufficiently nice'' spaces into statements about discrete cycles being, or not being, boundaries of discrete chains. This is one of the main goals of Section \ref{sec:refine_cech}, see in particular Corollary \ref{cor:refinable-cohomology}. 

The topological spaces to which this strategy can be applied are those that we call \emph{super-refinable}, see Definition \ref{defn:super-refinable}. Roughly, these are spaces where each open cover admits a ``very small'' refinement. Another main goal of Section \ref{sec:refine_cech} is to show that (countable) direct limits of compact metric spaces are super-refinable, see Proposition \ref{prop:limit-refinable}. This is relevant because Morse boundaries are, by definition, direct limits of compact metric spaces (possibly over an uncountable index set, but not in our case).

In Section \ref{sect:Bowditch_filling} we prove Proposition \ref{prop:Bowditch_filling}, which allows us to fill discrete cycles in the Bowditch boundary of the relatively hyperbolic groups that we are interested in. This result is near-obvious in the setting of Theorem~\ref{thm:main_intro}, so Section~\ref{sect:Bowditch_filling} is only required to reach the generality of Theorem~\ref{thm:general_intro}.

In Section \ref{sect:Morse_filling} we relate Morse boundaries to Bowditch boundaries. Roughly, the strata of the Morse boundary are obtained by removing from the Bowditch boundary a collection of balls of specific radii around parabolic points. Hence, in order to fill a discrete cycle in the Morse boundary, the idea is to start with a filling in the Bowditch boundary (as constructed in Section \ref{sect:Bowditch_filling}) and ``detour'' it away from parabolic points. We implement this strategy to prove Proposition \ref{diameter bound}, which is one of the main results of the section. We note that a similar strategy was adopted in \cite{Mackay-Sisto} to construct suitable arcs in Bowditch boundaries, and in fact we further develop the techniques of that paper for our purposes. 

In Section \ref{sect:Morse_filling} we also show Proposition \ref{prop:representing_homology} in support of our non-vanishing result, Theorem \ref{thm:non_vanishing}. Roughly, considering the Bowditch boundary minus a finite collection of parabolic points $F$, this proposition provides, for each homology class, a representative that is supported in a fixed stratum of the Morse boundary (depending on $F$). In other words, these representatives stay away not just from $F$, but from \emph{all} parabolic points, in a suitable sense.

In Section \ref{sect:vanishing} we conclude the proof of our vanishing result, Theorem \ref{thm:vanishing}. The main point here is that the fillings constructed in Section \ref{sect:Morse_filling} need to be improved in order to apply Corollary \ref{cor:refinable-cohomology}, and to do so we have to ``make certain discrete chains finer''. In fact, we will use fillings to do so.

Finally, in Section \ref{sect:non-vanishing} we prove our non-vanishing result, Theorem \ref{thm:non_vanishing}. Besides the above-mentioned Proposition \ref{prop:representing_homology}, this boils down to Lemma \ref{lem:nonvanishing_criterion}, which reduces non-vanishing of \v{C}ech cohomology (over a field of characteristic $0$) to finding discrete cycles that are not boundaries of specific chains. This can be seen as a counterpart to Corollary \ref{cor:refinable-cohomology}.

\medskip
{\bf Acknowledgements.} We thank the referee for their several useful comments.

\addtocontents{toc}{\protect\setcounter{tocdepth}{2}}
\section{Discrete chains and super-refinements}
\label{sec:refine_cech}

In this section we define and study discrete chains and related notions. We could not find any papers in the literature studying the exact versions of these notions that we need. However, there are a few places where similar concepts have been considered. The closest we could find is \cite{Barcelo-Capraro-White}, where the authors define discrete homology groups at a given scale, and then consider a direct limit as the scale goes to infinity. Instead, here we are interested in the behaviour at small scales.

\subsection{\v{C}ech cohomology and discrete chains}\label{subsec:discrete_chains}

Let $X$ be a topological space and let $\mc{O}$ be an open cover of $X$. 

We denote by $N(\mc{O})$ the nerve of $\mc{O}$, and by $C_*(N(\mc{O}))$ its chain complex of simplicial chains with $\Z$--coefficients. In other words, for each $n\geq 0$, the group $C_n(N(\mc{O}))$ is the free $\Z$--module generated by ordered $(n+1)$--tuples $[O_0,\dots,O_n]$, where $O_0,\dots,O_n\in\mc{O}$ and $O_0\cap\dots\cap O_n\neq\emptyset$.

If $\mc{O}'<\mc{O}$ is a refinement, there are well-defined maps in homology and cohomology (see e.g.\ part~(1) of Lemma~\ref{lem:family} below):
\begin{align*}
&H_*(N(\mc{O}'),\Z)\ra H_*(N(\mc{O}),\Z), & &H^*(N(\mc{O}),\Z)\ra H^*(N(\mc{O}'),\Z) .
\end{align*}
The \emph{\v{C}ech homology}\footnote{We caution the reader that \v{C}ech homology fails to satisfy the Eilenberg--Steenrod exactness axiom and, thus, it is not a homology theory as such. Instead, \v{C}ech cohomology does constitute an honest cohomology theory.} and \emph{cohomology} of $X$ are then defined as: 
\begin{align*}
\cH_*(X,\Z)=\varprojlim_{\mc{O}} H_*(N(\mc{O}),\Z), & &\cH^*(X,\Z)=\varinjlim_{\mc{O}} H^*(N(\mc{O}),\Z),
\end{align*}
where the two limits are taken over the directed set of open covers $\mc{O}$ of $X$ with the order given by refinements. \v{C}ech (co)homology with other coefficients is defined analogously. 

Rather than working with open covers and their nerves, it is often simpler to study \emph{discrete chains}, which we now introduce. We will relate them to \v{C}ech (co)homology in Subsection~\ref{subsec:super-refinements}. 

We define a chain complex $\mscr{C}_*(X)$ as follows. For each $n\geq 0$, the $\Z$--module $\mscr{C}_n(X)$ is freely generated by ordered $(n+1)$--tuples $[x_0,\dots,x_n]$, where $x_0,\dots,x_n\in X$. As customary, the boundary operator is the $\Z$--linear extension of:
\[ \partial[x_0,\dots,x_n]=\sum_{i=0}^n (-1)^i\cdot [x_0,\dots,\widehat{x_i},\dots,x_n]. \]
We refer to tuples $[x_0,\dots,x_n]$ as \emph{discrete $n$--simplices}, and to general elements of $\mscr{C}_n(X)$ as \emph{discrete $n$--chains}. A discrete chain $c$ is a \emph{cycle} if $\partial c=0$, and a \emph{boundary} if $c=\partial d$ for some discrete chain $d$. It is convenient to also introduce the module of \emph{reduced} $n$--chains $\overline{\mscr{C}}_n(X)$: for $n\geq 1$ we have $\overline{\mscr{C}}_n(X)=\mscr{C}_n(X)$, while $\overline{\mscr{C}}_0(X)$ is the kernel of the coefficient-sum map $\mscr{C}_0(X)\ra\Z$.

The \emph{support} of a discrete simplex $[x_0,\dots,x_n]$ is the set $\{x_0,\dots,x_n\}\sq X$. If $c=\sum_i\s_i$ is a discrete chain, written as a sum of discrete simplices without any cancellations, we define its \emph{support} as $\supp(c):=\bigcup_i\supp(\s_i)$. 

With a slight abuse, we will often confuse a discrete chain $c$ with its support. For instance, for a subset $A\sq X$, we will write ``$c\sq A$'' meaning ``$\supp(c)\sq A$''. Similarly, we will speak of ``vertices $x\in c$'', rather than ``points $x\in\supp(c)$''.

We will often be interested in chains that are ``small'' relative to a fixed cover, in one of two possible different ways, as described in the following definition.

\begin{defn}\label{defn:tiny/fine}
    Given an open cover $\mc{O}$ of $X$, a subset $A\sq X$ is \emph{$\mc{O}$--tiny} if there exists $O\in\mc{O}$ such that $A\sq O$. A discrete chain is \emph{$\mc{O}$--tiny} if its support is. A discrete chain is \emph{$\mc{O}$--fine}\footnote{In the literature, it is common to refer to both $\mc{O}$--tiny subsets and $\mc{O}$--fine chains simply as ``$\mc{O}$--small''. We prefer to use distinct terms for the two concepts, also to highlight the analogy between $\mc{O}$--fine chains and ``$\delta$--fine'' chains introduced in Subsection~\ref{subsec:chains_metric}.} if all of its simplices are $\mc{O}$--tiny. We denote by $\mscr{C}_*(X,\mc{O})\sq\mscr{C}_*(X)$ the chain subcomplex of $\mc{O}$--fine discrete chains. For a refinement $\mc{O}'<\mc{O}$, we have $\mscr{C}_*(X,\mc{O}')\sq\mscr{C}_*(X,\mc{O})$.
\end{defn}

Although not strictly necessary for our purposes, the following definitions are helpful to illustrate our viewpoint on \v{C}ech (co)homology. Let $\mscr{H}_*(X,\mc{O})$ denote the homology of the chain complex $\mscr{C}_*(X,\mc{O})$. For each refinement $\mc{O}'<\mc{O}$, the inclusion $\mscr{C}_*(X,\mc{O}')\sq\mscr{C}_*(X,\mc{O})$ gives rise to a homomorphism $\mscr{H}_*(X,\mc{O}')\ra\mscr{H}_*(X,\mc{O})$. We can then define the \emph{discrete homology} of $X$ as:
\[ \mscr{H}_*(X):=\varprojlim_{\mc{O}} \mscr{H}_*(X,\mc{O}),\]
where the inverse limit is again taken over the directed set of open covers $\mc{O}$ of $X$ with the order given by refinements. For nice spaces, we will show in Corollary~\ref{cor:Cech_vs_discrete} that $\mscr{H}_*(X)$ is naturally isomorphic to $\cH_*(X,\Z)$ (so discrete homology also does not satisfy the Eilenberg--Steenrod axioms).

We conclude this subsection with a couple of definitions and an observation, which will be used multiple times throughout the paper.

\begin{defn}[Cones]\label{defn:cone}
For a discrete simplex $\s=[x_0,\dots,x_n]$ and a point $x\in X$, the \emph{cone} over $\s$ from $x$ is the simplex $\cone(x,\s):=[x,x_0,\dots,x_n]$. In particular, if $\s$ is the empty $-1$--simplex, we define $\cone(x,\s):=[x]$. 

For a discrete chain $c=\sum_j\s_j$, where the $\s_j$ are discrete simplices, the \emph{cone} over $c$ from $x$ is the chain $\cone(x,c):=\sum_j\cone(x,\s_j)$.
\end{defn}

\begin{lem}\label{lem:cone}
Let $c$ be a discrete chain in a topological space $X$. For every point $x\in X$, we have $\partial\cone(x,c)=c-\cone(x,\partial c)$. In particular, if $c\in\overline{\mscr{C}}_n(X)$ is a reduced cycle, then $\partial\cone(x,c)=c$.
\end{lem}
\begin{proof}
For a discrete simplex $\s$, a simple computation shows that $\partial\cone(x,\s)=\s-\cone(x,\partial\s)$. The statement about general chains then follows from linearity of the coning operator.
\end{proof}

\begin{defn}[Face complex]\label{defn:face_complex}
Let $c=\sum_i\s_i$ be a discrete chain, written as a sum of discrete simplices without cancellations. The \emph{face complex} of $c$ is the chain subcomplex $\mscr{F}_*(c)\sq\mscr{C}_*(X)$ generated by the $\s_i$ and all their lower-dimensional faces.
\end{defn}

\subsection{Super-refinements}\label{subsec:super-refinements}

Let $X$ be a topological space.

The following strengthening of the notion of refinement will be crucial for us to relate properties of nerves of covers to properties of discrete chains, and hence to reduce \v{C}ech cohomology statements to statements about discrete chains.

\begin{defn}\label{defn:super-refinement}
Let $\mathcal O$ be an open cover of $X$. An open cover $\mc{O}'$ is a \emph{super-re\-fine\-ment} of $\mathcal O$ (denoted $\mc{O}'\ll\mc{O}$) if, for each $x\in X$, there exists $O\in\mc{O}$ with:
\begin{align*}%\label{super_refinement_parent}
    \bigcup \left\{O' \mid O'\in\mathcal O',\ x\in O'\right\} \subseteq O .
\end{align*}

An open set $O$ satisfying the above inclusion is called a \emph{parent} of $x$.
\end{defn}

\begin{defn}
\label{defn:super-refinable}
A topological space is \emph{super-refinable} if every open cover admits a super-refinement.
\end{defn}

It is readily seen that compact metric spaces are super-refinable due to the existence of Lebesgue numbers for covers (that is, the fact that, for every open cover, there exists $\epsilon>0$ such that every $\epsilon$--ball is contained in an open set from the cover). We will show in Subsection~\ref{subsec:direct_limits} that countable direct limits of compact metrisable spaces are super-refinable as well.

In order to connect discrete chains in $X$ to simplicial chains in nerves of open covers of $X$, we introduce the following terminology.

\begin{defn}\label{defn:family}
Let $\mc{O}$ be an open cover of $X$.
\begin{itemize}
\item A map $f\colon\mc{O}\ra X$ is a \emph{child map} if $f(O)\in O$ for all $O\in\mc{O}$.
\end{itemize}
Let $\mc{O}'<\mc{O}$ be a refinement.
\begin{itemize}
\item A map $f\colon\mc{O}'\ra\mc{O}$ is a \emph{spouse map} if $O'\sq f(O')$ for all $O'\in\mc{O}'$.
\end{itemize}
Suppose now that $\mc{O}'\ll\mc{O}$ is a super-refinement.
\begin{itemize}
\setlength\itemsep{.1cm}
\item A map $f\colon X\ra\mc{O}$ is an \emph{$(\mc{O}',\mc{O})$--parent map} if $O'\sq f(x)$ for all $O'\in\mc{O}'$ with $x\in O'$ (that is, $f(x)$ is a parent of $x$ according to Definition~\ref{defn:super-refinement}).
\item A map $f\colon X\ra X$ is an \emph{$(\mc{O}',\mc{O})$--sibling map} if it is the composition of an $(\mc{O}',\mc{O})$--parent map $X\ra\mc{O}$ and a child map $\mc{O}\ra X$.
\end{itemize}
\end{defn}

Each of the maps appearing in Definition~\ref{defn:family} naturally induces a chain map between the corresponding chain complexes. This is the content of the next lemma.

\begin{lem}\label{lem:family}
Let $\mc{O}$ be an open cover of $X$ with a refinement $\mc{O'}<\mc{O}$.
\begin{enumerate}
\item Every spouse map $f\colon\mc{O}'\ra\mc{O}$ naturally induces a chain map $f_*\colon C_*(N(\mc{O}'))\ra C_*(N(\mc{O}))$. If $g\colon\mc{O}'\ra\mc{O}$ is another spouse map, then $f_*$ and $g_*$ are chain-homotopic.
\end{enumerate}
Suppose now that $\mc{O}'\ll\mc{O}$ is a super-refinement.
\begin{enumerate}
\setcounter{enumi}{1}
\setlength\itemsep{.1cm}
\item Every child map $f\colon\mc{O}'\ra X$ naturally induces a chain map $f_*\colon C_*(N(\mc{O}'))\ra \mscr{C}_*(X,\mc{O})$.
If $g\colon\mc{O}'\ra X$ is another child map, then $f_*$ and $g_*$ are chain-homotopic.
\item Every $(\mc{O}',\mc{O})$--parent map $f\colon X\ra\mc{O}$ naturally induces a chain map $f_*\colon \mscr{C}_*(X,\mc{O}')\ra C_*(N(\mc{O}))$.
If $g\colon X\ra\mc{O}$ is another $(\mc{O}',\mc{O})$--parent map, then $f_*$ and $g_*$ are chain-homotopic.
\end{enumerate}

Finally, consider a chain of super-refinements $\mc{O}''\ll\mc{O}'\ll\mc{O}$.
\begin{enumerate}
\setcounter{enumi}{3}
\item Every $(\mc{O}'',\mc{O}')$--sibling map $f\colon X\ra X$ naturally induces a chain map $f_*\colon \mscr{C}_*(X,\mc{O}'')\ra \mscr{C}_*(X,\mc{O})$ and this is chain-homotopic to the standard inclusion. 
\end{enumerate}
\end{lem}
\begin{proof}
Part~(1) is classical (see e.g.\ \cite[Lemma~10.4.2]{Bott-Tu}) and the proof of the other three parts is very similar, so we will omit most of the computations.

Recall that $\mscr{C}_n(X,\mc{O})$ is the free $\Z$--module generated by $\mc{O}$--tiny discrete simplices $[x_0,\dots,x_n]$, while $C_n(N(\mc{O}))$ is the free $\Z$--module generated by the elements $[O_0,\dots,O_n]$, where $O_0,\dots,O_n\in\mc{O}$ and $O_0\cap\dots\cap O_n\neq\emptyset$.

\smallskip
{\bf Part~(1).} The chain map $f_n\colon C_n(N(\mc{O}'))\ra C_n(N(\mc{O}))$ takes $[O_0',\dots,O_n']$ to $[f(O_0'),\dots,f(O_n')]$. If $g\colon\mc{O}'\ra\mc{O}$ is another spouse map, a chain homotopy $H_n\colon C_n(N(\mc{O}'))\ra C_{n+1}(N(\mc{O}))$ between $f_*$ and $g_*$ is given by the formula:
\[ H_n[O_0',\dots,O_n']:=\sum_{i=0}^n (-1)^i \cdot \left[ f(O_0'),\dots,f(O_i'),g(O_i'),\dots,g(O_n') \right] . \]

\smallskip
{\bf Part~(2).} The chain map $f_n\colon C_n(N(\mc{O}'))\ra \mscr{C}_n(X,\mc{O})$ takes $[O_0',\dots,O_n']$ to $[f(O_0'),\dots,f(O_n')]$. We prove that $f_n$ is well-defined. Since $\mc{O}'$ is a super-refinement of $\mc{O}$, there exists a parent $O\in\mc{O}$ of some point $x\in O_0'\cap\dots\cap O_n'$, and we have: 
\[\{f(O_0'),\dots,f(O_n')\}\sq O_0'\cup\dots\cup O_n' \sq O.\] 
Thus, the discrete simplex $[f(O_0'),\dots,f(O_n')]$ is indeed $\mc{O}$--tiny.

Now, if $g\colon\mc{O}'\ra X$ is another child map, a chain homotopy $H_n\colon C_n(N(\mc{O}'))\ra\mscr{C}_{n+1}(X,\mc{O})$ between $f_*$ and $g_*$ is given by the formula:
\[ H_n[O_0',\dots,O_n']:=\sum_{i=0}^n (-1)^i \cdot \left[ f(O_0'),\dots,f(O_i'),g(O_i'),\dots,g(O_n') \right] . \]
This map is again well-defined because all $f(O_i')$ and $g(O_i')$ lie in any parent $O\in\mc{O}$ of any point of $O_0'\cap\dots\cap O_n'$. We leave to the reader the straightforward check that $\partial H_n+H_{n-1}\partial=g_n-f_n$ for all $n\geq 0$.

\smallskip
{\bf Part~(3).} The chain map $f_n\colon\mscr{C}_n(X,\mc{O}')\ra C_n(N(\mc{O}))$ takes $[x_0,\dots,x_n]$ to $[f(x_0),\dots,f(x_n)]$. If $\{x_0,\dots,x_n\}$ is contained in $O'\in\mc{O}'$, then 
\[f(x_0)\cap\dots\cap f(x_n)\supseteq O'\neq\emptyset,\] 
by definition of parent map. This proves that $f_n$ is well-defined.

If $g\colon X\ra\mc{O}$ is another parent map, a chain homotopy $H_n\colon\mscr{C}_n(X,\mc{O}')\ra C_{n+1}(N(\mc{O}))$ between $f_*$ and $g_*$ is given by the formula:
\[ H_n[x_0,\dots,x_n]:=\sum_{i=0}^n (-1)^i \cdot \left[ f(x_0),\dots,f(x_i),g(x_i),\dots,g(x_n) \right] . \]
Again, this is a well-defined map because all $f(x_i)$ and $g(x_i)$ contain any $O'\in\mc{O}'$ containing all $x_i$. As above, we have $\partial H_n+H_{n-1}\partial=g_n-f_n$ for all $n\geq 0$.

\smallskip
{\bf Part~(4).} The chain map $f_n\colon\mscr{C}_n(X,\mc{O}'')\ra \mscr{C}_n(X,\mc{O})$ takes $[x_0,\dots,x_n]$ to $[f(x_0),\dots,f(x_n)]$. Let us show that this is well-defined. First, by definition of sibling map, there exists an $(\mc{O}'',\mc{O}')$--parent map $F\colon X\ra\mc{O}'$ such that $f(x)\in F(x)$ for all $x\in X$. If $\{x_0,\dots,x_n\}$ is $\mc{O}''$--tiny, we have $F(x_0)\cap\dots\cap F(x_n)\neq\emptyset$ as in part~(3). It follows that $F(x_0)\cup\dots\cup F(x_n)$ is $\mc{O}$--tiny as in part~(2). Since this union contains all $f(x_i)$, we conclude that the discrete simplex $[f(x_0),\dots,f(x_n)]$ is indeed $\mc{O}$--tiny.

Now, if $g\colon X\ra X$ is another sibling map, a chain homotopy $H_n\colon\mscr{C}_n(X,\mc{O}')\ra \mscr{C}_{n+1}(X,\mc{O})$ between $f_*$ and $g_*$ is given by the formula:
\[ H_n[x_0,\dots,x_n]:=\sum_{i=0}^n (-1)^i \cdot \left[ f(x_0),\dots,f(x_i),g(x_i),\dots,g(x_n) \right] . \]
The map $H_n$ is well-defined by the argument in the previous paragraph, and the usual computation shows that $\partial H_n+H_{n-1}\partial=g_n-f_n$ for all $n\geq 0$.

This concludes the proof of the lemma.
\end{proof}

Part~(1) of Lemma~\ref{lem:family} shows that each refinement $\mc{O}'<\mc{O}$ gives rise to well-defined maps between the (co)homology groups of the nerves of $\mc{O}$ and $\mc{O}'$, which we used to define the \v{C}ech (co)homology groups of $X$ in Subsection~\ref{subsec:discrete_chains}.

The rest of Lemma~\ref{lem:family} allows us to relate the \v{C}ech homology of $X$ to its discrete homology, provided that $X$ is super-refinable. While we will not need the next result (preferring instead its reformulation given by Corollary~\ref{cor:refinable-cohomology}), it motivates our approach to \v{C}ech (co)homology.

\begin{cor}\label{cor:Cech_vs_discrete}
If $X$ is super-refinable, its \v{C}ech homology $\cH_*(X,\Z)$ is naturally isomorphic to its discrete homology $\mscr{H}_*(X)$.
\end{cor}
\begin{proof}
Recall that we have:
\begin{align*}
\cH_*(X,\Z)=\varprojlim_{\mc{O}} H_*(N(\mc{O}),\Z), & &\mscr{H}_*(X)=\varprojlim_{\mc{O}} \mscr{H}_*(X,\mc{O}).
\end{align*}
For each super-refinement $\mc{O}'\ll\mc{O}$, parts~(2) and~(3) of Lemma~\ref{lem:family} give maps:
\begin{align*}
H_*(N(\mc{O}'),\Z)&\ra \mscr{H}_*(X,\mc{O}), & \mscr{H}_*(X,\mc{O}')&\ra H_*(N(\mc{O}),\Z).
\end{align*}
In addition, for a chain of super-refinements $\mc{O}''\ll\mc{O}'\ll\mc{O}$, parts~(1) and~(4) of Lemma~\ref{lem:family} guarantee that the compositions
\begin{align*}
H_*(N(\mc{O}''),\Z)&\ra \mscr{H}_*(X,\mc{O}')\ra H_*(N(\mc{O}),\Z), \\ \mscr{H}_*(X,\mc{O}'')&\ra H_*(N(\mc{O}'),\Z)\ra \mscr{H}_*(X,\mc{O}).
\end{align*}
coincide with the standard homomorphisms. 

Since $X$ is super-refinable, the above maps give rise to homomorphisms
\begin{align*}
\cH_*(X,\Z)&\ra \mscr{H}_*(X), & \mscr{H}_*(X)&\ra \cH_*(X,\Z),
\end{align*}
and both their compositions are the identity. This proves the corollary.
\end{proof}

In order to show vanishing of $\cH_n(X,\Z)$ and $\mscr{H}_n(X)$ for $n\geq 1$, it is sometimes more practical to check that $X$ satisfies the following (a priori stronger) condition\footnote{The letters D and A stand for Discrete and Acyclic.}:
\begin{enumerate}
\item[$(DA_n)$] \emph{for each open cover $\mathcal O$ there exists a refinement $\mathcal O'$ such that every discrete $\mathcal O'$--fine $n$--cycle is the boundary of a discrete $\mathcal O$--fine $(n+1)$--chain.}
\end{enumerate}
We only define $(DA_n)$ as above for $n\geq 1$, declaring instead $(DA_0)$ to be vacuous.

\begin{cor}\label{cor:refinable-cohomology}
Let $X$ be super-refinable and let $n\geq 1$. Let $\mc{F}$ be a field and $\mc{R}$ a principal ideal domain.
\begin{enumerate}
\item If $X$ satisfies $(DA_n)$, then $\cH_n(X,\Z)=\{0\}$ and $\cH^n(X,\mc{F})=\{0\}$.
\item If $X$ satisfies $(DA_n)$ and $(DA_{n-1})$, then $\cH^n(X,\mc{R})=\{0\}$.
\end{enumerate}
\end{cor}
\begin{proof}
Suppose that $X$ satisfies $(DA_n)$ and let $\mathcal{O}$ be an open cover of $X$. 

For part~(1), we want to show that there exists a refinement $\mathcal{O}^\prime<\mathcal O$ and a spouse map $s\colon \mathcal O^\prime \to\mathcal O$ such that the induced maps
\begin{align*}
s_n&\colon H_n(N(\mathcal O'),\Z)\to H_n(N(\mathcal O),\Z), & s^n&\colon H^n(N(\mathcal O),\mc{F})\to H^n (N(\mathcal O ^\prime),\mc{F})
\end{align*} 
are trivial. It then follows by definition of \v{C}ech (co)homology that $\cH_n(X,\Z)=0$ and $\cH^n(X,\mc{F})=0$.

Consider a chain of refinements $\mc{O}'\ll \tilde{\mathcal O}^\prime < \tilde{\mathcal O}\ll\mc{O}$, where $\tilde{\mathcal O}^\prime < \tilde{\mathcal O}$ is a refinement provided by $(DA_n)$. Let $f\colon\mc{O}'\ra X$ be a child map and let $g\colon X\ra\mc{O}$ be an $(\tilde{\mathcal O}, \mc{O})$--parent map. Let $f_*,g_*$ be the chain maps provided by Lemma~\ref{lem:family}. Observe that $s:=g\circ f\colon \mc {O}'\to \mc {O}$ is a spouse map. 

Now, let $c\in C_n(N(\mc{O}'))$ be a cycle. Then $f_*(c)\in\mscr{C}_n(X,\tilde{\mc{O}}')$ is an $\tilde{\mc{O}}'$--fine discrete cycle and hence there exists an $\tilde{\mc{O}}$--fine discrete chain $d$ with $f_*(c)=\partial d$. Since $f_*$ and $g_*$ are chain maps, we have that:
\[\partial g_*(d)=g_*(\partial d)=g_*f_*(c)=s_*(c).\]

This shows that the induced map in homology $s_n\colon H_n(N(\mathcal O'),\Z)\to H_n(N(\mathcal O),\Z)$ is trivial. The induced map in cohomology $s^n\colon H^n(N(\mathcal O),\mc{F})\to H^n(N(\mathcal O'),\mc{F})$ is then also trivial, by the universal coefficients theorem. This proves part~(1).

In order to prove part~(2), let the refinement $\mc{O}'$ and the spouse map $s\colon\mc{O}'\ra\mc{O}$ be those constructed in part~(1). Then consider a refinement $\mc{O}''<\mc{O}'$ and a spouse map $t\colon\mc{O}''\ra\mc{O}$ obtained by applying the same procedure to the cover $\mc{O}'$. We will show that, when $X$ satisfies $(DA_{n-1})$, the spouse map $s\o t$ induces the trivial map $H^n(N(\mathcal O),\mc{R})\to H^n(N(\mathcal O''),\mc{R})$.

The universal coefficients theorem (in the form stated e.g.\ in \cite[Section~5.5, Theorem~3]{Spanier}) gives the following commutative diagram with exact rows:
\[
\begin{tikzcd}
0 \arrow[r] & {\rm Ext}(H_{n-1}(N(\mathcal O),\Z),\mc{R}) \arrow[r] \arrow[d, "(s_{n-1})^*"] & H^n(N(\mathcal O),\mc{R}) \arrow[r] \arrow[d, "s^n"] & {\rm Hom}(H_n(N(\mathcal O),\Z),\mc{R}) \arrow[r] \arrow[d, "(s_n)^*"] & 0 \\
0 \arrow[r] & {\rm Ext}(H_{n-1}(N(\mathcal O'),\Z),\mc{R}) \arrow[r] \arrow[d, "(t_{n-1})^*"] & H^n(N(\mathcal O'),\mc{R}) \arrow[r] \arrow[d, "t^n"] & {\rm Hom}(H_n(N(\mathcal O'),\Z),\mc{R}) \arrow[r] \arrow[d, "(t_n)^*"] & 0 \\
0 \arrow[r] & {\rm Ext}(H_{n-1}(N(\mathcal O''),\Z),\mc{R}) \arrow[r] & H^n(N(\mathcal O''),\mc{R}) \arrow[r] & {\rm Hom}(H_n(N(\mathcal O''),\Z),\mc{R}) \arrow[r] & 0
\end{tikzcd}
\]
Since $X$ satisfies $(DA_n)$ and $(DA_{n-1})$, we have shown that the maps in homology $s_n$ and $t_{n-1}$ are trivial, provided that $n\geq 2$. Hence, in the above diagram, the left-hand map $(t_{n-1})^*$ and the right-hand map $(s_n)^*$ are trivial. This latter statement also holds for $n=1$, since homology is free in degree $0$, which leads to vanishing of the ${\rm Ext}$ terms.

A quick diagram chase then shows that the composition $t^n\o s^n$ is trivial, as required.
\end{proof}

We conclude this subsection with the following observation needed in Section~\ref{sect:vanishing}.

\begin{rmk}\label{rmk:super_super_refinement}
Consider an open cover $\mc{O}$ and a double super-refinement $\mc{O}''\ll\mc{O}'\ll\mc{O}$. Then, for any pairwise-intersecting collection $\{O_i''\}_{i\in I}\sq\mc{O}''$, there exists some $O\in\mc{O}$ containing all $O_i''$.

Indeed, fixing some $i_0\in I$, there exist elements $O_i'\in\mc{O}'$ with $O_{i_0}''\cup O_i''\sq O_i'$, since $\mc{O}''\ll\mc{O}'$ and $O_{i_0}''\cap O_i''\neq\emptyset$. Now, all $O_i'$ contain any given point of $O_{i_0}''$, hence their union is contained in an element of $\mc{O}$, again because $\mc{O}'\ll\mc{O}$.
\end{rmk}

\subsection{Direct limits of compact spaces}\label{subsec:direct_limits}

We will prove the vanishing part of Theorems~\ref{thm:main_intro} and~\ref{thm:general_intro} by checking that our Morse boundaries satisfy condition $(DA_n)$ in the required range and applying Corollary~\ref{cor:refinable-cohomology}. Thus, we first need to show that these spaces are super-refinable, which is the goal of this subsection (Proposition~\ref{prop:limit-refinable}). 

Recall that, given a sequence of topological spaces $X_n\subseteq X_{n+1}$, the direct limit $X=\varinjlim X_n$ is the union of the $X_n$, endowed with the topology where $U\subseteq X$ is open if and only if $U\cap X_n$ is open for all $n$. All our Morse boundaries will be spaces of this form.

\begin{prop}\label{prop:limit-refinable}
Any countable direct limit of compact metrisable spaces is super-refinable. 
\end{prop}

Before proving the proposition, we need to make a couple of observations. First, we can and will assume in the following discussion that the union of the $X_n$ is equipped with a single metric $d(\cdot,\cdot)$ that induces the topology of each $X_n$. (Of course, we stress that $d(\cdot,\cdot)$ does not induce the topology of the direct limit $X$, which is strictly finer.)

In general, the existence of the metric $d(\cdot,\cdot)$ is guaranteed by the following result; see \cite[Lemma~1.4]{Isbell} or \cite{Zarichnyi} for short proofs. For our Morse boundaries, we will actually simply consider the restriction of a visual metric on the Bowditch boundary.

\begin{lem}
Let $A$ be a closed subspace of a compact metrisable space $B$. Every metric $d_A\colon A\x A\ra\R$ inducing the topology of $A$ can be extended to a metric $d_B\colon B\x B\ra\R$ inducing the topology of $B$.
\end{lem}

The following will be our main recipe to construct open covers of countable direct limits of compact metrisable topological spaces. 

\begin{cst}\label{cover construction}
Consider a countable direct limit $X=\varinjlim X_n$ of metric spaces and a sequence of positive reals $\underline{\eps}=(\eps_n)_{n\geq 0}$. We describe a family $\mscr{U}(\underline{\eps})$ of open sets of $X$ that will come in handy multiple times in the future. 

There is an element $U_p\in\mscr{U}(\underline{\eps})$ for each point $p\in X$, constructed as follows. If $p\in X_k-X_{k-1}$ for some $k\geq 0$, set $U_p(n)=\emptyset$ for $n<k$, and define $U_p(k)$ as the open $\eps_k$--ball around $p$ within $X_k$. For $n>k$, we define iteratively $U_p(n)$ as the open $\eps_n$--neighbourhood of $U_p(n-1)$ within $X_n$. Finally, set $U_p:=\bigcup_n U_p(n)$.

Note that, since $U_p(n)$ is an open subset of $X_n$, we have that $U_p\cap X_n$ is the union of the open sets $U_p(m)\cap X_n$ for $m\geq n$, and hence it is open. This shows that $U_p$ is open in $X$, by definition of the the direct limit topology. 

Note that it is possible for $U_p(n)$ to be strictly contained in the intersection $U_p\cap X_n$. For instance, if $p\in X_0$, then $U_p\cap X_0$ is a ball of radius $\sum_{n\geq 0}\eps_n$, whereas $U_p(0)$ is just an $\eps_0$--ball. It is also useful to introduce the notation $\mscr{U}_k(\underline{\eps}):=\{U_p \mid p\in X_k-X_{k-1}\}$ and $\mscr{U}_k^n(\underline{\eps})=\{U_p(n) \mid p\in X_k-X_{k-1}\}$.
\end{cst}

If $\mc{U}$ is a family of subsets of a topological space $X$ and $A\sq X$ is another subset, it is convenient to write:
\[ \mc{U}[A]:=\bigcup_{U\in\mc{U},\ U\cap A\neq\emptyset} \overline U; \]
here $\overline U$ denotes the closure of $U$ in $X$. We are finally ready to prove Proposition~\ref{prop:limit-refinable}.

\begin{proof}[Proof of Proposition~\ref{prop:limit-refinable}]
As observed above, it suffices to consider a sequence of isometric embeddings of compact metric spaces $\dots\hookrightarrow X_n\hookrightarrow X_{n+1}\hookrightarrow\dots$ with $n\geq 0$. Let $X=\varinjlim X_n$ be their direct limit. Let $\mc{O}$ be an open cover of $X$. Our aim is to construct a super-refinement of $\mc{O}$.

We rely on Construction~\ref{cover construction}, for a sequence $\underline{\eps}=(\eps_n)_{n\geq 0}$ to be determined later. For every $k\geq 0$, choose a finite subset $I_k\sq X_k-X_{k-1}$ such that the $\eps_k$--balls $\{U_i(k)\in \mscr{U}_k^k(\underline{\eps})  \mid i\in I_k\}$ cover $X_k - X_{k-1}$. (For instance, cover $X_k$ by finitely many $\eps_k/2$--balls, select only those that intersect $X_k-X_{k-1}$, for each of them pick a point in its intersection with $X_k-X_{k-1}$, then take the finitely many $\eps_k$--balls with these centres.) 

Define $I:=\bigsqcup_{n\geq 0}I_n$ and $\mc{U}_k^n:=\{U_i(n) \mid i\in I_k\}$ (that is, the elements of $\mc{U}_k^n$ are open sets of $X_n$ constructed as in Construction \ref{cover construction} starting from the points of $I_k$). We say that a finite subset $J\sq I_0\cup\dots\cup I_k$ is \emph{$k$--good} if $J\cap I_k\neq\emptyset$ and the union of closures $\bigcup_{j\in J}\overline{U_j(k)}$ is contained in some $O\in\mc{O}$; we fix one such open set $O_J$ for every good set $J$.

\smallskip
{\bf Claim.} \emph{
There exist numbers $\eps_n,\eta_n>0$ such that the following hold for $n\geq 0$.
\begin{itemize}
\item[(A)] If $B$ is an $\eta_n$--ball in $X_n$, then the set $\big(\mc{U}_0^n\cup\dots\cup\mc{U}_n^n\big)[B]$ is contained in an element of $\mc{O}$.
\item[(B)] For every $0\leq k<n$ and every $k$--good set $J$, we have $\bigcup_{j\in J}\overline{U_j(n)}\sq O_J$.
\end{itemize}
}

\smallskip
First, assuming the Claim, we construct a super-refinement of $\mc{O}$. For each $i\in I_k$, set $V_i:=U_i-X_{k-1}$ and define $\mc{V}=\{V_i \mid i\in I\}$. It is clear that this is an open cover of $X$. Let us show that it super-refines $\mc{O}$.

Consider a point $x\in X_n$ and let $J\sq I$ be the subset of indices $j$ such that $x\in V_j$. Note that $J\sq I_0\cup\dots\cup I_n$, so $J$ is finite. It follows that, for some $m>n$, the point $x$ lies in $U_j(m)$ for all $j\in J$. By property~(A), there exists $O\in\mc{O}$ containing $\bigcup_{j\in J}\overline{U_j(m)}$. Thus, $J$ is $k$--good for some $k$ and property~(B) ensures that $O_J\in\mc{O}$ contains $\bigcup_{j\in J}\overline{U_j(m')}$ for all $m'>k$. This implies that $O_J$ contains all $U_j$ with $j\in J$, hence all $V_j$ with $j\in J$, as we wanted.

Note that we have only used property~(A) with the ball $B$ replaced by a single point, and the need for closures of the $U_j(n)$ is also not apparent yet. The full strength of property~(A) is needed in the proof of the Claim.

\smallskip
We conclude by proving the Claim, that is, by showing that $\eps_n,\eta_n$ can be chosen in a way that (A) and (B) hold. Let $\lambda_n$ be a Lebesgue number for the open cover $\{O\cap X_n \mid O\in\mc{O}\}$ of $X_n$. We proceed by induction on $n$.

Consider the base step $n=0$. Property~(B) holds vacuously in this case (note the ``$k<n$''). Regarding property~(A), consider an $\eta_0$--ball $B\sq X_0$ and recall that $\mc{U}_0^0$ is a finite cover of $X_0$ by $\eps_0$--balls. The diameter of $\mc{U}_0^0[B]$ is at most $4\eps_0+2\eta_0$. Choosing $\eps_0,\eta_0$ with $4\eps_0+2\eta_0<\lambda_0$, we ensure that $\mc{U}_0^0[B]$ is contained in an element of $\mc{O}$. 

For the inductive step, suppose that $\eps_0,\dots,\eps_n$ and $\eta_0,\dots,\eta_n$ have been chosen to satisfy (A) and (B), and let us look for suitable values of $\eps_{n+1}$ and $\eta_{n+1}$.

Observe that there are only finitely many $k$--good sets $J$ with $k\leq n$ and, for each of them, we have
\[ d\Big(\bigcup_{j\in J}\overline{U_j(n)}\ ,\ X_{n+1}-O_J\Big)>0,\]
as this is the distance between two disjoint compact sets. Thus, we can choose a number $\delta_{n+1}>0$ that is smaller than all these distances. We then choose $\eps_{n+1}$ and $\eta_{n+1}$ so that the following inequalities hold:
\begin{align*}
4\eps_{n+1}+2\eta_{n+1}&<\lambda_{n+1}, & 2\eps_{n+1}+2\eta_{n+1}&<\eta_n, & 3\eps_{n+1}+2\eta_{n+1}&<\delta_{n+1}.
\end{align*}
This immediately guarantees that property~(B) is satisfied, as $\eps_{n+1}<\delta_{n+1}$ and each $U_j(n+1)$ is the $\eps_{n+1}$--neighbourhood of $U_j(n)$ in $X_{n+1}$.

Regarding property~(A), consider an $\eta_{n+1}$--ball $B\sq X_{n+1}$. Let $J\sq I_0\cup\dots\cup I_n$ be the set of indices $j$ such that $U_j(n+1)\cap B\neq\emptyset$.

If $J=\emptyset$, then $\big(\mc{U}_0^{n+1}\cup\dots\cup\mc{U}_{n+1}^{n+1}\big)[B]=\mc{U}_{n+1}^{n+1}[B]$. As in the base step, this is a set of diameter at most $4\eps_{n+1}+2\eta_{n+1}<\lambda_{n+1}$, which ensures that it is contained in an element of $\mc{O}$. 

If instead $J\neq\emptyset$, we can choose for every $j\in J$ a point $x_j\in U_j(n)$ with $d(x_j,B)<\eps_{n+1}$. The set $\{x_j \mid j\in J\}$ is contained in $X_n$ and it has diameter at most $2\eps_{n+1}+2\eta_{n+1}<\eta_n$, so it is contained in an $\eta_n$--ball $B'\sq X_n$. For every $j\in J$, the set $U_j(n)$ meets $B'$, so property~(A) for $n$ implies that $J$ is $k$--good for some $k\leq n$. Let $\mc{O}_J\in\mc{O}$ be the corresponding open set.

Now, the set $\big(\mc{U}_0^{n+1}\cup\dots\cup\mc{U}_{n+1}^{n+1}\big)[B]$ is contained in the $(3\eps_{n+1}+2\eta_{n+1})$--neighbourhood of $\bigcup_{j\in J} \overline{U_j(n)}$ within $X_{n+1}$. Since $3\eps_{n+1}+2\eta_{n+1}<\delta_{n+1}$, this set is contained in $O_J\in\mc{O}$, as required.
\end{proof}

\subsection{Discrete chains in metric spaces}\label{subsec:chains_metric}

Let $(X,\rho)$ be a metric space. In this short subsection, we introduce some important notation relating to chains in the presence of a metric.

If $c$ is a discrete or singular chain in $X$, we write $\diam(c):=\diam(\supp(c))$ for short. With a slight abuse, we will sometimes also write things like ``$\rho(x,c)\geq 1$'' for a point $x\in X$, rather than the more correct ``$\rho(x,\supp(c))\geq 1$''.

If $c=\sum_i\s_i$, written as a sum of discrete (resp.\ singular) simplices without cancellations, we define 
\[\fin(c):=\sup_i\diam(\s_i).\] 
If $\fin(c)\leq\delta$ for some $\delta>0$, we say that $c$ is \emph{$\delta$--fine}.

\subsection{Discretisation of singular simplices}

Given a singular chain $d$, one can associate to it a discrete chain that we will denote by $\disc(d)$, informally by only keeping track of the vertices. More precisely, we have the following.

\begin{defn}
     If $\sigma\colon\Delta^i\ra X$ is a singular simplex, then $\disc(\sigma_i)$ is the discrete simplex $[x_0,\dots,x_i]$, where $x_i$ is the image under $\sigma$ of the $i$--th vertex of the standard simplex $\Delta^i$. This definition is then extended linearly; namely, for $d=\sum c_i\sigma_i$ we set $\disc(d):=\sum c_i \disc(\sigma_i)$.
\end{defn}

\begin{lem}\label{singular->fine}
Let $X$ be a metric space and let $\hat d\sq X$ be a singular chain. Then there exists a discrete chain $d\sq X$ with $\partial d=\disc(\partial\hat d)$ and  $\fin(d)\leq 2\fin(\partial\hat d)$.
\end{lem}

\begin{proof}
It suffices to perform an iterated relative barycentric subdivision of $\hat d$ without ever altering the cell structure on $\partial\hat d$. More precisely, recall that, in the usual barycentric subdivision, the new vertices are the barycentres of all simplices. Instead, in this relative subdivision, the new vertices are the barycenters of simplices that are not part of $\partial \hat d$, as well as the vertices of simplices that are part of $\partial \hat d$. 

Iterating this relative subdivision procedure, the fineness of $\hat d$ converges to the fineness of $\partial\hat d$. Thus, we can just take $d=\disc(\hat d')$ for a sufficiently high subdivision $\hat d'$ of the original $\hat d$.
\end{proof}

\section{Filling discrete cycles in the Bowditch boundary}\label{sect:Bowditch_filling}

The goal of this section is to prove Proposition \ref{prop:Bowditch_filling} below, which gives a way of filling cycles in the Bowditch boundary of certain relatively hyperbolic groups. We now recall some background on relatively hyperbolic groups, introduce the setup and state the proposition. We note that no other results from this section are used in the rest of the paper.

\subsection{Cusped space and Bowditch boundary}

We now briefly discuss cusped spaces and their boundaries, and we refer the reader to \cite[Subsection 2.1]{Mackay-Sisto} for more details, as we will mostly use techniques from that paper. These notions were first considered in \cite{Bow:rel-hyp}.

Given a connected graph $\Gamma$, one can construct a hyperbolic space, called the \emph{combinatorial horoball} $\mc{H}(\Gamma)$, by gluing strips of the hyperbolic plane onto the edges of the graph, see \cite[Definition 2.1]{Mackay-Sisto}. We refer to the copy of the graph contained in the combinatorial horoball as the horosphere, and we recall that, for $x,y\in \Gamma$, we have that $d_{\mc{H}(\Gamma)}(x,y)$ and $2\log (d_{\Gamma}(x,y))$ differ by a uniformly bounded amount, see \cite[Lemma 2.2]{Mackay-Sisto}.

A relatively hyperbolic group pair $(G,\mc{P})$ is a finitely generated group $G$ together with a collection of finitely generated subgroups $\mc{P}$, called peripheral subgroups, such that one obtains a hyperbolic space by gluing combinatorial horoballs onto the cosets of the peripheral subgroups, see \cite[Definition 2.3]{Mackay-Sisto}. This glued-up space is the \emph{cusped space} $\mc{X}(G,P)$, and its Gromov boundary is the \emph{Bowditch boundary} $\partial_B(G,\mc{P})$. The combinatorial horoballs inside the cusped space lie at bounded Hausdorff distance from horoballs as defined in terms of Busemann functions, see \cite[Lemma 2.5]{Bow:rel-hyp}.

\subsection{Setup and statement of main proposition}\label{subsec:setup}

Let $(G,\mc{P})$ be a relatively hyperbolic pair. We assume that $G$ is finitely generated and that $\mc{P}$ consists of finitely many finitely generated, virtually nilpotent subgroups. In addition, we assume that the Bowditch boundary $\mc{S}:=\partial_B(G,\mc{P})$ is homeomorphic to a sphere $S^{k+1}$ for some $k\geq 1$.

For $a,b>0$ and $K\geq 1$, we write $a\sim_Kb$ as shorthand for $1/K\leq a/b\leq K$. As customary, we also write $a\vee b:=\max\{a,b\}$ and $a\wedge b:=\min\{a,b\}$.

Let $(\mc{X},d_\mc{X})$ be the cusped space for $(G,\mc{P})$ and let $\delta_0$ be its hyperbolicity constant. For simplicity we choose, for all points $x\in \mc{X}$ and $y\in \mc{X}\cup\partial \mc{X}$, a geodesic from $x$ to $y$, which we denote by $[x,y]$. 

Let $\eps_0$ be a parameter for visual metrics, which can be chosen depending only on $\delta_0$ (e.g.\ as in \cite[Proposition~III.H.3.21]{BH}). Denote by $\rho_x$ the visual metric on $\mc{S}=\partial \mc{X}$ based at a point $x\in \mc{X}$. There exists a constant $C$ (depending only on $\delta_0,\eps_0$) such that $\rho_x(\xi,\xi')\sim_Ce^{-\eps_0(\xi|\xi')_x}$ for all $x\in \mc{X}$ and $\xi,\xi'\in\partial \mc{X}$. Here $(\cdot|\cdot)_x$ denotes Gromov products based at $x$. Choose a basepoint $o\in \mc{X}$ and set $\rho:=\rho_o$.

For each parabolic point $p\in\partial \mc{X}$, denote by $\mc{H}_p\sq \mc{X}$ the corresponding horoball and by $\mc{K}_p\sq\mc{H}_p$ its boundary horosphere. Choose a point $e_p\in\mc{K}_p$ nearest to $o$ and set $r_p:=e^{-\eps_0d_\mc{X}(o,\mc{H}_p)}$.

If $Z$ is a metric space, $x\in Z$ and $0<r_1<r_2$, we denote by $A_Z(x;r_1,r_2)$ the closed metric annulus around $x$ with inner radius $r_1$ and outer radius $r_2$. When the metric space under consideration is clear, we simply write $A(x;r_1,r_2)$. Similarly, $B_Z(x,r)$ or $B(x,r)$ is the closed $r$--ball around $x$.

We are ready to state the main result of Section~\ref{sect:Bowditch_filling}:

\begin{prop}\label{prop:Bowditch_filling}
Given the setup above, there exist a function $g\colon\R^+\ra\R^+$ and a constant $K\geq 1$ such that the following statements hold for every $\delta>0$. Let $c\in\overline{\mscr{C}}_i(\mc{S})$ be a reduced discrete cycle with $\fin(c)\leq g(\delta)$.
\begin{enumerate}
\setlength\itemsep{.1cm}
\item[(0)] If $i\leq k$, then $c=\partial d$ for a $\delta$--fine discrete chain $d\sq\mc{S}$ with $\diam(d)\leq K\sqrt{\diam(c)}$.
\item If $i<k$ and $c\sq A(p;r_1,r_2)$ for a parabolic point $p\in\mc{S}$ and radii $\delta\leq r_1<r_2\leq r_p/K$, then $c=\partial d$ for a $\delta$--fine discrete chain $d\sq\mc{S}$ with $\diam(d)\leq K\sqrt{\diam(c)}$ and $d\sq A(p;r_1/K,Kr_2)$.
%the chain $d$ provided by Item~(0) can be constructed so that, in addition, $d\sq A(p;r_1/K,Kr_2)$.
\item If $i\leq k$ and $c\sq B(p,r_2)$ for a parabolic point $p\in\mc{S}$ and a radius $0<r_2\leq r_p/K$, then $c=\partial d$ for a $\delta$--fine discrete chain $d\sq\mc{S}$ with $\diam(d)\leq K\sqrt{\diam(c)}$ and $d\sq B(p,Kr_2)$.
\end{enumerate}
\end{prop}

We note that, if we are only interested in fundamental groups of finite-volume real hyperbolic manifolds, then we can use $\mathbb H^n$ instead of the cusped space $\mc{X}$. In this case, the boundary is isometric to a round sphere and the proposition is almost immediate. 

In particular, the reader only interested in the proof of Theorem~\ref{thm:main_intro} can safely skip the rest of Section~\ref{sect:Bowditch_filling}, which is more technical than the rest of the paper.

\subsection{Geometry of nilpotent groups}

The following lemma can be easily deduced from a result of Karidi on balls in nilpotent Lie groups \cite{Karidi}. Roughly speaking it says that, in an $n$--dimensional, connected, nilpotent Lie group, discrete $i$--cycles can be filled within annuli by chains of roughly the same fineness for $i<n-1$, while $(n-1)$--cycles can be ``discretely homotoped'' far away, again with the same fineness. 

\begin{lem}\label{nilpotent Lie group}
Let $\mc{N}$ be an $n$--dimensional, connected, nilpotent Lie group, equipped with a left-invariant Riemannian metric. There exists $L\geq 1$ such that the following holds for all $r_2>r_1\geq L$, all $x\in\mc{N}$ and all discrete cycles $c\in\overline{\mscr{C}}_i(A_{\mc{N}}(x;r_1,r_2))$.
\begin{enumerate}
\setlength\itemsep{.1cm}
\item If $0\leq i<n-1$ and $\fin(c)\leq r_1/L$, then $c=\partial d$ for some discrete chain $d\sq A_{\mc{N}}(x;r_1/L,Lr_2)$ with $\fin(d)\leq L(\fin(c)\vee 1)$.
\item If $i=n-1$ and $\fin(c)\leq r_1/L$, then $c-\partial d\sq A_{\mc{N}}(x;r_2/L,Lr_2)$ for some discrete chain $d\sq A_{\mc{N}}(x;r_1/L,Lr_2)$ with $\fin(d)\leq L(\fin(c)\vee 1)$.
\end{enumerate}
\end{lem}
\begin{proof}
For simplicity, we write $B(r):=B_{\mc{N}}(1,r)$ and $A(r_1,r_2):=A_{\mc{N}}(1;r_1,r_2)$, where $1\in\mc{N}$ is the identity. By the main result of \cite{Karidi}, there exist a constant $a>1$ and a collection of topological $n$--discs $1\in\mf{B}(r)\sq\mc{N}$ such that, for every $r>1$, we have $\mf{B}(r/a)\sq B(r)\sq\mf{B}(ar)$. In addition, the sets $\mf{A}(r_1,r_2):=\mf{B}(r_2)-\mf{B}(r_1)$ are homeomorphic to $S^{n-1}\x(0,1]$.

First of all, observe that for every (reduced) singular cycle $\hat c\sq\mc{N}$, we have $\hat c=\partial\hat d$ for a singular chain $\hat d$ with $\diam(\hat d)\leq 2a^2\cdot (\diam(\hat c)\vee 2)$. Indeed, setting $D:=\diam(\hat c)\vee 2$, we have $\hat c\sq B(D)\sq\mf{B}(aD)$ up to left multiplication. Since $\mf{B}(aD)$ is a topological disc, there exists a singular chain $\hat d\sq\mf{B}(aD)$ with $\partial\hat d=\hat c$. Finally, since $\mf{B}(aD)\sq B(a^2D)$, we have $\diam(\hat d)\leq 2a^2D$.

Consider now a discrete chain $c$. We want to argue that there exists a singular chain $\hat{c}$ with $\disc(\hat{c})=c$ and fineness controlled in terms of the fineness of $c$. We can obtain this by ``filling in'' discrete $1$--simplices with singular $1$--simplices, at which point discrete $2$--simplices yield singular $1$--cycles, and we can fill in those as well, and proceed inductively. It follows that, for every discrete $i$--chain $c\sq\mc{N}$, there exists a singular $i$--chain $\hat c\sq\mc{N}$ with $c=\disc(\hat c)$ and 
\[\fin(\hat c)\leq \tfrac{1}{2}(2a)^{2i}\cdot(\fin(c)\vee 2).\]

Now, consider a discrete cycle $c\in\overline{\mscr{C}}_i( A_{\mc{N}}(x;r_1,r_2))$ with $\fin(c)\vee 2\leq (2a)^{-2i}r_1$ and $i\geq 0$. Up to left multiplication, we can assume that $c\sq A(r_1,r_2)$. By the previous paragraph, we have $c=\disc(\hat c)$ for a (reduced) singular cycle $\hat c$ with $\fin(\hat c)\leq r_1/2$. In particular, we have 
\[\hat c\sq A(r_1/2,2r_2)\sq\mf{A}(r_1/2a,2ar_2):=\mf{A}.\] 
Note that $\mf{A}\sq A(r_1/2a^2,2a^2r_2)$.

Recall that $\mf{A}$ is homeomorphic to $S^{n-1}\x(0,1]$. Thus, if $i\neq n-1$, there exists a singular chain $\hat d\sq\mf{A}$ with $\hat c=\partial\hat d$. If $i=n-1$, there exists nevertheless $\hat d\sq\mf{A}$ with $\partial \hat d=\hat c-\hat c'$ where $\hat c'\sq\mf{A}(r_2/2a,2ar_2)$ is a singular chain homologous to $\hat c$, which we can and will arrange to be finer than $\hat c$ (e.g.\ by barycentric subdivision). Note that $\hat c'$ exists since $\mf{A}(r_2/2a,2ar_2)$ is a deformation retract of $\mf{A}$.

Finally, since $\fin(\partial\hat d)=\fin(\hat c)\leq \frac{1}{2}(2a)^{2i}(\fin(c)\vee 2)$, Lemma~\ref{singular->fine} yields a discrete chain $d\sq\mf{A}$ with $\partial d=\disc(\partial\hat d)$ and $\fin(d)\leq (2a)^{2i}(\fin(c)\vee 2)$. For $i\neq n-1$, we have $\partial d=c$ and, for $i=n-1$, we have $c-\partial d=\disc(\hat c')\sq\mf{A}(r_2/2a,2ar_2)$, as required.
\end{proof}

\subsection{Preliminary lemmas}\label{subsec:prelim_lemmas}

In the rest of Section~\ref{sect:Bowditch_filling}, we consider the setup described in Subsection~\ref{subsec:setup}. Also recall that $\overline{\mscr{C}}_*(\mc{S})$ denotes the chain complex of reduced discrete chains in the Bowditch boundary $\mc{S}$, where $0$--chains are required to have zero coefficient sum.

The following lemma says that we can fill discrete cycles while staying away from a specified point in the boundary, and that each point has neighborhoods where cycles can be filled.

For short, we write $f=o(1)$ for a function $f:\mathbb R^+\to \mathbb R^+$ with $\lim_{t\to 0} f(t)=0$.

\begin{lem}\label{b}
There exist functions $f,h=o(1)$ such that the following statements hold for every point $\xi\in\mc{S}$ and every $\delta,\eps>0$.
\begin{enumerate}
\setlength\itemsep{.1cm}
\item If $c\in\overline{\mscr{C}}_*(\mc{S})$ is a discrete $\delta$--fine cycle with $\rho(\xi,c)>\epsilon$, then there exists a discrete chain $d\sq\mc{S}$ with $\partial d=c$ and $\fin(d)\leq f(\delta)$ and $\rho(\xi,d)>h(\epsilon)$.
\item There exists an open subset $\xi\in W\sq\mc{S}$ such that $\diam(W)\leq\eps$ and such that, for every discrete $\delta$--fine cycle $c\in\overline{\mscr{C}}_*(W)$, there exists a discrete chain $d\sq W$ with $\partial d=c$ and $\fin(d)\leq f(\delta)$.
\end{enumerate}
\end{lem}

The reason why, in the above statement, we introduce $\delta$ and conclude that $\fin(d)\leq f(\delta)$, rather than simply writing $\fin(d)\leq f(\fin(c))$, is to deal with $0$--cycles. These are $0$--fine but cannot be filled with a $0$--fine $1$--chain.

\begin{proof}
The lemma follows from the fact that $\mc{S}$ is homeomorphic to a sphere, and the fact that the statement holds for a sphere with its round metric. More precisely, let us consider two metrics on $\mc{S}$, namely $\rho$ and a metric that makes it isometric to a round sphere. 

For item (2) we can take $W$ to be a ball in the round metric. In said ball, any $c\in\overline{\mscr{C}}_*(W)$ can be filled by a chain with the same fineness (with respect to the round metric). Comparing $\rho$ to the round metric, using the fact that they induce the same topology on $\mc{S}$, yields item (2).

Item (1) is similar, with the small difference that, in the complement of a ball in the round metric, any $c\in\overline{\mscr{C}}_*(W)$ can be filled by a chain whose fineness is at most a fixed multiple of $\fin(c)$.
\end{proof}

For each parabolic point $p\in\mc{S}=\partial \mc{X}$, we can define two maps projecting the Bowditch boundary to the horosphere around $p$ and vice versa:
\begin{align*}
\Psi_p&\colon\mc{S}-\{p\}\ra\mc{K}_p, & \Phi_p\colon\mc{K}_p\ra\mc{S}-\{p\}.
\end{align*}
More precisely, we define these maps as follows. For every $\xi\in\mc{S}-\{p\}$, the point $\Psi_p(\xi)$ is within distance $1$ of a bi-infinite geodesic from $p$ to $\xi$. This would not be strictly necessary, but for convenience, we will require $\Psi_p$ to be injective. This can always be arranged, since we allow $\Psi_p$ to take values outside the vertex set of $\mc{K}_p$ and any set of cardinality $2^{\aleph_0}$ (a suitable neighbourhood of $\xi$ in $\mc{S}$) can be injected into any set of the same cardinality (a suitable neighbourhood of a point in $\mc{K}_p$). Conversely, for every $x\in\mc{K}_p$, the point $\Phi_p(x)$ is the endpoint at infinity of a bi-infinite geodesic from $p$ that passes within distance $1$ of $x$. Again for convenience, we will require that $\Phi_p\Psi_p(\xi)=\xi$ for all $\xi\in\mc{S}-\{p\}$.

The next three lemmas are general facts about cusped spaces, which do not require the Bowditch boundary to be a sphere, nor the peripherals to be virtually nilpotent. They are similar to \cite[Lemma 5.6]{Mackay-Sisto} and \cite[Lemma 6.1]{MS2}, and can be proven with the same method. Namely, one can use that finite configurations of points in a hyperbolic space can be approximated by trees, and compute all relevant Gromov products in said trees (see \cite[Figures 1 and 2]{Mackay-Sisto}). We give more details below only for the first lemma. 

We denote by $d_{\mc{K}_p}(\cdot,\cdot)$ the intrinsic path metric of a horosphere $\mc{K}_p$.

\begin{lem}\label{lem:Psi}
There exists a constant $C\geq 1$, depending only on $\delta_0,\eps_0$, such that the following statements hold for all $x,y\in\mc{S}-\{p\}$: 
\begin{enumerate}
\item $d_{\mc{K}_p}(e_p,\Psi_p(x))\sim_C 1\vee\left(\frac{r_p}{\rho(p,x)}\right)^{1/\eps_0}$;
\item if $\rho(x,y)\leq \rho(p,\{x,y\})$, then $d_{\mc{K}_p}(\Psi_p(x),\Psi_p(y))\sim_C 1\vee\left(\frac{r_p\cdot\rho(x,y)}{\rho(p,x)^2}\right)^{1/\eps_0}$.
\end{enumerate}
\end{lem}
\begin{proof}[Proof sketch]
The outline for item (1) is as follows, in the case that $\Psi_p(x)$ is sufficiently far from $e_p$. Figure \ref{horob1} shows an approximating tree for the relevant points. The two highlighted segments have the same length $\ell$ because their endpoints lie on a horosphere. We have that $(p|x)_o$ is equal to $d(o,e_p)+\ell$ up to bounded additive error, and we also have that $d_{\mc{K}_p}(e_p,\Psi_p(x))\sim_D e^\ell$ for some $D$ depending only on the cusped space $\mc{X}$. Recalling the definition of the visual metric, we have that $\rho(p,x)\sim_{D'} r_p e^{-\epsilon_0\ell}$ for some $D'$. Hence, $(r_p/\rho(p,x))^{1/\epsilon_0}\sim_{D''} e^\ell$ for some $D''$ and, as mentioned above, the right-hand side is $\sim_D d_{\mc{K}_p}(e_p,\Psi_p(x))$, as required.

In item (2), in principle there would be two approximating trees to consider, up to swapping $x$ and $y$. These are pictured in Figure \ref{horob2}. However, in the right-hand configuration, $(p|x)_o-(x|y)_o$ is approximately $t$ as marked in the Figure \ref{horob2}. Thus, if $t$ is large we have $\rho(x,y)\gg \rho(p,x,y)$, while for small $t$ the right-hand configuration collapses into the left-hand one. Estimates analogous to those used for item~(1) conclude the proof.
\end{proof}

\begin{figure}[ht] 
	%	\hspace{5cm}
	%	\captionsetup{width=1.8\linewidth}
	\centering
	\includegraphics{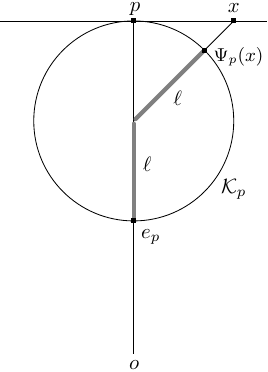}
	\caption{A point in the boundary and its ``projection'' to the horosphere corresponding to the parabolic point $p$, as in the proof of item~(1) of Lemma~\ref{lem:Psi}.}\label{horob1}
\end{figure} 

\begin{figure}[ht] 
%	\hspace{5cm}	%	\captionsetup{width=1.8\linewidth}
\centering
\includegraphics{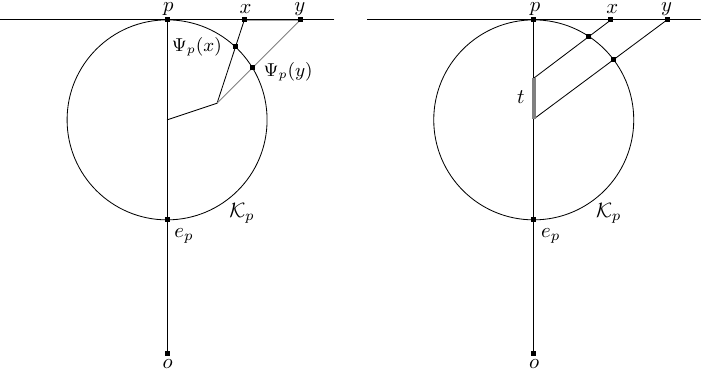}
\caption[align=left]{Illustration of the proof idea of item~(2) of Lemma~\ref{lem:Psi}.}\label{horob2}
\end{figure}

\begin{lem}\label{lem:Phi}
There exists a constant $C\geq 1$, depending only on $\delta_0,\eps_0$, such that the following statements hold for all points $u,v\in\mc{K}_p$ outside the $C$--ball around $e_p$:
\begin{enumerate}
\item $\rho(p,\Phi_p(u))\sim_C\frac{r_p}{d_{\mc{K}_p}(e_p,u)^{\eps_0}}$;
\item if $d_{\mc{K}_p}(u,v)\leq \tfrac{1}{2}\cdot d_{\mc{K}_p}(e_p,u)$, then $\rho(\Phi_p(u),\Phi_p(v))\leq C\cdot \frac{r_p\cdot \left(1\vee d_{\mc{K}_p}(u,v)\right)^{\eps_0}}{d_{\mc{K}_p}(e_p,u)^{2\eps_0}}$.
\end{enumerate}
\end{lem}

For a discrete chain $c\sq\mc{S}$, we define the \emph{barycentre} $\mf{b}(c)\in \mc{X}\cup \mc{S}$ as follows: choose a vertex $\xi\in c$ and let $\mf{b}(c)\in[o,\xi]$ be the point with $d_\mc{X}(o,\mf{b}(c))=\inf_{\eta,\eta'\in c}(\eta|\eta')_o$. Observe that $\diam(c)\sim_C e^{-\eps_0d_\mc{X}(o,\mf{b}(c))}$, for a constant $C$ only depending on $\delta_0,\eps_0$.

\begin{lem}\label{lem:barycentre_vs_horoball}
For every $K\geq 1$, there exists $D\geq 0$ such that the following holds for all parabolic points $p\in\mc{S}$, all points $\xi\in B(p,r_p)\sq\mc{S}$ and all points $z\in [0,\xi]$.
\begin{enumerate}
\item If $\frac{1}{K}\cdot\frac{\rho(p,\xi)^2}{r_p}\leq e^{-\eps_0d_\mc{X}(o,z)}\leq r_p$, then $d_\mc{X}(z,\mc{H}_p)\leq D$. 
\item If $e^{-\eps_0d_\mc{X}(o,z)}\leq K\cdot\frac{\rho(p,\xi)^2}{r_p}$, then $d_\mc{X}(z,\mc{X}-\mc{H}_p)\leq D$. 
\end{enumerate}
In addition, there exists $K_0$, depending only on $\delta_0,\eps_0$, with the following property.
\begin{enumerate}
\item[(3)] If $z\in\mc{H}_p$, then $\frac{1}{K_0}\cdot\frac{\rho(p,\xi)^2}{r_p}\leq e^{-\eps_0d_\mc{X}(o,z)}$.
\end{enumerate}
\end{lem}

Finally, the next lemma is a consequence of Lemma~\ref{nilpotent Lie group} and we provide a complete proof.

\begin{lem}\label{a'}
There exists a constant $L\geq 1$ such that the following holds for every horosphere $\mc{K}=\mc{K}_p\sq \mc{X}$ with $p\in\mc{S}$ a parabolic point, all $r_2>r_1\geq L$ and all discrete cycles $c\in\overline{\mscr{C}}_i (A_{\mc{K}}(e_p;r_1,r_2))$.
\begin{enumerate}
\setlength\itemsep{.1cm}
\item If $i<k$ and $\fin(c)\leq r_1/L$, then $c=\partial d$ for some discrete chain $d\sq A_{\mc{K}}(e_p;r_1/L,Lr_2)$ with $\fin(d)\leq L(1\vee\fin(c))$.
\item If $i=k$ and $\fin(c)\leq r_1/L$, then $c-\partial d\sq A_{\mc{K}}(e_p;r_2/L,Lr_2)$ for some discrete chain $d\sq A_{\mc{K}}(e_p;r_1/L,Lr_2)$ with $\fin(d)\leq L(1\vee\fin(c))$.
\end{enumerate}
\end{lem}
\begin{proof}
Since there are only finitely many $G$--orbits of horospheres in $\mc{X}$, and peripherals are virtually nilpotent, it suffices to prove the statement in every finitely generated nilpotent group $N$ (equipped with a word metric). Passing to a finite-index subgroup, we can assume that $N$ is torsion-free; see e.g.\ \cite[Theorem~2.1]{Baumslag}. Set $n:=\sum{\rm rk} (N^{(i)}/N^{(i+1)})$.

The group $N$ is a uniform lattice in its Mal'cev completion $\mc{N}$, which is an $n$--dimensional, connected, nilpotent Lie group; see e.g.\ \cite[Theorem~13.40]{Drutu-Kapovich}.

Observe that $n=k+1$. Indeed, $N$ acts freely and cocompactly both on $\mc{N}$ and $\mc{S}-\{\ast\}$ (since the parabolic point corresponding to $N$ is a \emph{bounded parabolic point}, see e.g. \cite[Section 6]{Bow:rel-hyp}). These spaces are homeomorphic, respectively, to $\R^n$ and $\R^{k+1}$, hence the cohomological dimension of $N$ equals both $n$ and $k+1$.

Now, since the inclusion $N\hookrightarrow\mc{N}$ is a quasi-isometry, the statement follows from Lemma~\ref{nilpotent Lie group}. 
\end{proof}

\subsection{The main argument}

In order to simplify notation throughout this subsection, we will write $a\sim b$, for two quantities $a,b>0$, when there exists a constant $C\geq 1$, depending only on $\delta_0,\eps_0$, such that $a\sim_C b$. Similarly, we write $a\lesssim b$ when $a\leq Cb$ for such a constant $C$. 

For a constant $D\geq 0$, we say that a point $x\in \mc{X}$ is \emph{$D$--deep} in a horoball $\mc{H}_p$ if $x\in\mc{H}_p$ and $d_\mc{X}(x,\mc{K}_p)\geq D$. Recall that $\overline{\mscr{C}}_*(\mc{S})$ denotes \emph{reduced} discrete chains in $\mc{S}$; this means that $0$--chains are required to have zero coefficient sum.

We now embark in the proof of Proposition~\ref{prop:Bowditch_filling}. The following result gives us some initial control on fineness and diameter of fillings.

\begin{lem}\label{lem:shallow_barycentre}
For every $D\geq 0$, there exist a function $f=o(1)$ and a constant $K\geq 1$ with the following property. Let $c\in\overline{\mscr{C}}_*(\mc{S})$ be a discrete cycle such that some point $z\in [o,\mf{b}(c)]$ is not $D$--deep in any horoball. If $c$ is $\delta$--fine for some $\delta>0$, then $c=\partial d$ for a discrete chain $d\sq\mc{S}$ with $\fin(d)\leq f(e^{\eps_0d_\mc{X}(o,z)}\cdot\delta)$ and $\diam(d)\leq Ke^{-\eps_0d_\mc{X}(o,z)}$.
\end{lem}

\begin{proof}
Throughout the proof, balls and diameters will usually refer to the metric $\rho=\rho_o$; in case of ambiguity, we will indicate the relevant metric as a subscript. 

The reason for introducing $\delta$ in the statement is the same as for Lemma \ref{b}. In the argument below, we will instead use $\fin(c)$ with a slight abuse, even though for $0$--cycles we should use some $\delta>0$ instead.

We distinguish two cases in the proof.

\smallskip\noindent
{\bf Case~(1):} \emph{the geodesic $[o,z]$ goes at least $(D+K'+10\delta_0)$--deep into a horoball $\mc{H}_p$, for a constant $K'$ only depending on $D$ to be determined below.} 

\smallskip
Consider the metrics $\rho:=\rho_o$ and $\rho':=\rho_{z}$. Since $[0,z]\sq [0,\mf{b}(c)]$ goes $(D+10\delta_0)$--deep into $\mc{H}_p$, while $z$ is not $D$--deep in $\mc{H}_p$, all geodesic lines from $p$ to the vertices of $c$ intersect the $10\delta_0$--neighbourhood of $z$. Hence $\rho'(p,c)$ is bounded away from zero purely in terms of $\delta_0,\eps_0$, see Figure \ref{MorseinHoroballs}.

Since $z$ is not $D$--deep in any horoball, there exists $g\in G$ such that the metrics $\rho'$ and $\rho_{go}$ are $K$--bi-Lipschitz equivalent, for a constant $K$ only depending on $D$. In addition, $g$ gives an isometry between $(\mc{S},\rho_o)$ and $(\mc{S},\rho_{go})$, and Lemma~\ref{b}(1) holds in the former space. We conclude that there exists $d\sq\mc{S}$ with $\partial d=c$ and
\begin{align*}
\fin_{\rho'}(d)&\leq f_D(\fin_{\rho'}(c)), & \rho'(p,d)&\geq\eta_D,
\end{align*} 
for a function $f_D=o(1)$ and a constant $\eta_D>0$ depending on $D$. In particular, we have $\sup_{\xi\in d} (\xi|p)_{z}\leq K'$, for a constant $K'$ only depending on $D$ and $\delta_0$.

Let $x,y\in[o,z]$ be the points with $d_\mc{X}(o,x)=(p|z)_o$ and $d_\mc{X}(y,z)=\sup_{\xi\in d} (\xi|p)_{z}$. Assuming now that $[o,z]$ goes at least $(D+K'+10\delta_0)$--deep into $\mc{H}_p$, we have $d_\mc{X}(o,y)\geq d_\mc{X}(o,x)+10\delta_0$. Thus, all rays from $o$ to the vertices of $d$ track the geodesic $[o,z]$ at least up to $y$, hence $\inf_{\xi,\xi'\in d} (\xi|\xi')_o\geq d_\mc{X}(o,y)-10\delta_0$ and $\diam_{\rho}(d)\lesssim  e^{-\eps_0 d_\mc{X}(o,y)}\leq e^{\eps_0K'}e^{-\eps_0 d_\mc{X}(o,z)}$, as required.

\begin{figure}[ht] 
	%	\hspace{5cm}
	%	\captionsetup{width=1.8\linewidth}
	\centering
	\includegraphics{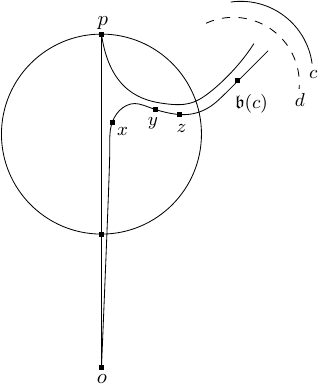}
	\caption{The points relevant for case (1) in the proof of Lemma \ref{lem:shallow_barycentre}.}\label{MorseinHoroballs}
\end{figure} 

We are left to control the fineness of $d$. Let $E(d)$ be the set of pairs $(\xi,\xi')\in\mc{S}^2$ corresponding to the edges of $d$. Recall that $\fin_{\rho'}(d)\leq f_D(\fin_{\rho'}(c))$. In addition, $\rho$ and $\rho'$ are $e^{\eps_0d_\mc{X}(o,z)}$--bi-Lipschitz equivalent. It follows that, for some constant $C$ depending only on $\delta_0,\eps_0$, we have:
\begin{align*}
\inf_{(\xi,\xi')\in E(d)} (\xi|\xi')_{z} \geq -\frac{\log f_D(\fin_{\rho'}(c))}{\eps_0}-C\geq -\frac{\log f_D(e^{\eps_0d_\mc{X}(o,z)}\fin_{\rho}(c))}{\eps_0}-C.
\end{align*}
Suppose for a moment that, for some $t\geq 1$, the right-hand side is at least $K'+(\log t)/\eps_0$. Then, recalling that $K'\geq \sup_{\xi\in d} (\xi|p)_{z}=d_\mc{X}(y,z)$, we obtain:
\[ \inf_{(\xi,\xi')\in E(d)} (\xi|\xi')_o \geq d_\mc{X}(o,y)+(\log t)/\eps_0-10\delta_0.\]
We then have $\fin_{\rho}(d)\lesssim  e^{-\eps_0d_\mc{X}(o,y)}/t\lesssim 1/t$. In conclusion, if $e^{\eps_0d_\mc{X}(o,z)}\fin_{\rho}(c)$ is smaller than a certain constant depending on $D$ and $t$, then $\fin_{\rho}(d)\leq 1/t$. 

This settles case~(1).

\smallskip\noindent
{\bf Case~(2):} \emph{the geodesic $[o,z]$ does not go $(D+K'+10\delta_0)$--deep into any horoball.}

\smallskip
In this case we will use a form of (partial) self-similarity of the boundary made precise in \cite[Lemma~4.6]{Mackay-Sisto}.

Let $D'\geq 0$ be a constant such that the geodesic $[o,z]$ is contained in the $D'$--neigh\-bour\-hood of the orbit $Go\sq \mc{X}$. Let $L_0=L_0(D')$ be the constant provided by \cite[Lemma~4.6]{Mackay-Sisto}. Let $f$ be the function provided by Lemma~\ref{b}, and choose an open cover $\mscr{W}$ of $\mc{S}$ by the sets provided by Lemma~\ref{b}(2) with $\eps=1/L_0$. Let $\alpha$ be a Lebesgue number for $\mscr{W}$.

Let $y\in[o,z]$ be the point with $d_\mc{X}(y,z)=D''$, where $D''$ is a sufficiently large constant to be determined below. (If $d_\mc{X}(o,z)< D''$, we can directly apply Lemma \ref{b}(1) because, in this case, the condition on the diameter is vacuous if $K$ is large enough.) Choose an element $g\in G$ with $d_\mc{X}(g^{-1}o,y)\leq D'$. Choose a vertex $\xi\in c$.

By \cite[Lemma~4.6]{Mackay-Sisto}, $g$ induces an $L_0$--bi-Lipschitz map from the rescaled ball $(B(\xi,r),\rho_o/r)$ to an open subset $U\sq\mc{S}$ containing $B(g\xi,1/L_0)$, for some radius $r\sim e^{-\eps_0d_\mc{X}(o,y)}$. Recall that $\diam(c)\sim e^{-\eps_0d_\mc{X}(o,\mf{b}(c))}\leq e^{-\eps_0D''}e^{-\eps_0d_\mc{X}(o,y)}$. Thus, if $D''$ is sufficiently large, we have $\diam(c)\leq r$, hence $c\sq B(\xi,r)$ and:
\[ \diam(gc)\leq L_0\diam(c)/r\lesssim L_0e^{-\eps_0d_\mc{X}(y,\mf{b}(c))}\leq L_0 e^{-\eps_0D''}.\]

Now, if $D''$ is sufficiently large (depending only on $D',\delta_0,\eps_0$), we also have $\diam(gc)\leq\alpha$. Since $\alpha$ is a Lebesgue number for $\mscr{W}$, the cycle $gc$ is then contained in some $W\in\mscr{W}$. Since $\diam(W)\leq 1/L_0$ and $g\xi\in gc\sq W$, we have $W\sq B(g\xi,1/L_0)\sq U$. In addition, by our choice of $\mscr{W}$, there exists a discrete chain $d'\sq W$ with $\partial d'=gc$ and $\fin(d')\leq f(\fin(gc))$.

In conclusion, setting $d:=g^{-1}d'$, we have $c=\partial d$ and:
\begin{align*}
\diam(d) &\leq rL_0\diam(d')\leq rL_0\diam(\mc{S})\lesssim L_0e^{-\eps_0d_\mc{X}(o,y)}= L_0e^{\eps_0D''}e^{-\eps_0d_\mc{X}(o,z)}, \\
\fin(d)&\leq rL_0\fin(d')\leq rL_0f(\fin(gc))\lesssim L_0f(L_0\fin(c)/r), 
% wlog f is weakly increasing; same observation omitted multiple times in various places
\end{align*}
where $r\sim e^{\eps_0D''}e^{-\eps_0d_\mc{X}(o,z)}$. Since $L_0$ and $D''$ only depend on $D'$ (and $\delta_0,\eps_0$), they only depend on $D$. This settles case~(2), proving the lemma.
\end{proof}

As a consequence of the previous lemma, we can already fill discrete cycles $c$ in the Bowditch boundary with good fineness and diameter control, provided that the barycentre $\mf{b}(c)$ is not deep in any horoball. This is the content of the following corollary.

\begin{cor}\label{cor:shallow_barycentre}
For every $D\geq 0$, there exist a function $g:\R^+ \to \R^+$ and $K\geq 1$ such that the following hold for all $\delta>0$ and all discrete cycles $c\in\overline{\mscr{C}}_*(\mc{S})$ with $\fin(c)\leq g(\delta)$.
\begin{enumerate} 
\setlength\itemsep{.1cm}
\item If some point $z\in [o,\mf{b}(c)]$ is not $D$--deep in any horoball, then $c=\partial d$ for a discrete chain $d\sq\mc{S}$ with $\fin(d)\leq\delta$ and $\diam(d)\leq Ke^{-\eps_0d_\mc{X}(o,z)}$.
\item If $\mf{b}(c)$ is not $D$--deep in any horoball, then $c=\partial d$ for a discrete chain $d\sq\mc{S}$ with $\fin(d)\leq\delta$ and $\diam(d)\leq K\diam(c)$.
\end{enumerate}
\end{cor}
\begin{proof}
As in the previous proof, we will use $\fin(c)$ with a slight abuse, rather than introducing a positive fineness constant.

Observe that part~(2) follows from part~(1) simply taking $z=\mf{b}(c)$ and recalling that $\diam(c)\sim e^{-\eps_0d_\mc{X}(o,\mf{b}(c))}$. (This is assuming that $\mf{b}(c)\in\mc{X}$. If instead $\mf{b}(c)\in\mc S$, the support of $c$ is a single point, and we can take $d$ supported on the same point.) We now prove part~(1). 

Let $f$ be as in Lemma~\ref{lem:shallow_barycentre}. Define $g(\delta):=\delta\cdot \tau$, for any $\tau>0$ such that $f(\tau')\leq \delta$ for all $\tau'\leq \tau$. By the lemma, we have $c=\partial d$ for a discrete chain with $\diam(d)\leq Ke^{-\eps_0d_\mc{X}(o,z)}$ and $\fin(d)\leq f(e^{\eps_0d_\mc{X}(o,z)}\fin(c))$. If $e^{\eps_0d_\mc{X}(o,z)}\fin(c)\leq \tau$, we have $\fin(d)\leq\delta$ as required. 

Otherwise $e^{-\eps_0d_\mc{X}(o,z)}\leq \fin(c)/\tau\leq \delta$ (since $\fin(c)\leq g(\delta)$ by assumption), so we can just define $d'$ as the cone over $c$ from any of its vertices (Definition~\ref{defn:cone}). Then $\partial d'=c$ and $\fin(d')=\diam(d')=\diam(c)$, where $\diam(c)\sim e^{-\eps_0d_\mc{X}(o,\mf{b}(c))}\leq e^{-\eps_0d_\mc{X}(o,z)}\leq\delta$.
\end{proof}

Just like Proposition \ref{prop:Bowditch_filling}, the next two lemmas roughly say that discrete cycles in annuli around parabolic points can be filled in slightly larger annuli. However, these two lemmas require an additional technical condition (cycle barycentres should lie in horoballs) that we can only drop later on by a bootstrapping argument.

Recall that the face complex $\mscr{F}_*(d)$ of a discrete chain $d$ was introduced in Definition~\ref{defn:face_complex}.

\begin{lem}\label{lem:deep_barycentre_1}
There exists a constant $K\geq 1$ such that the following holds for every parabolic point $p\in\mc{S}$ and all $0<r_1<r_2\leq r_p/K$. 

If $c$ is a discrete cycle in $\overline{\mscr{C}}_i(A(p;r_1,r_2))$ with $\mf{b}(c)\in\mc{H}_p$ and $\fin(c)\leq r_1^2/r_p$, then there exists a discrete chain $d\sq A(p;r_1/K,Kr_2)$ with $\diam(d)\leq K\sqrt{\diam(c)}$ and the following properties.
\begin{enumerate}
\item[$(*)$] For every simplex $\s\in\mscr{F}_*(d)$, we have $\diam(\s)\leq K\cdot\frac{\rho(p,\s)^2}{r_p}$.
\item[$(**)$] If $i<k$, we have $\partial d=c$. If $i=k$, we have $c-\partial d\sq B(p, Kr_1)$.
\end{enumerate}
\end{lem}
\begin{proof}
Choose a vertex $\xi\in c$ maximising the distance $\rho(p,\xi)$. Without loss of generality, we can assume that $r_2=\rho(p,\xi)$. Since $\mf{b}(c)\in\mc{H}_p$, Lemma~\ref{lem:barycentre_vs_horoball}(3) implies that $\diam(c)\gtrsim r_2^2/r_p$.

By Lemma~\ref{lem:Psi}(1), we have $\Psi_p(c)\sq A_{\mc{K}_p}(e_p; R_1,R_2)$ with $R_1\sim(r_p/r_2)^{1/\eps_0}$ and $R_2\sim(r_p/r_1)^{1/\eps_0}$. Since $\fin(c)\leq r_1^2/r_p$ by assumption, Lemma~\ref{lem:Psi}(2) also yields $\fin(\Psi_p(c))\lesssim 1$. 

In view of Lemma~\ref{a'} and the previous paragraph, there exists a constant $L$ such that the following holds. Provided that $r_2\leq r_p/L$, there exists a chain $\mc{D}\sq A_{\mc{K}_p}(e_p;R_1/L,LR_2)$ such that $\fin(\mc{D})\leq L$ and:
\begin{itemize}
\item if $i<k$, then $\partial\mc{D}=\Psi_p(c)$; 
\item if $i=k$, then $\Psi_p(c)-\partial\mc{D}$ is disjoint from $B_{\mc{K}_p}(e_p,R_2/L)$.
\end{itemize}

Now, set $d:=\Phi_p(\mc{D})$. Lemma~\ref{lem:Phi} guarantees that there exists a constant $K$ depending on $L,\delta_0,\eps_0$ such that, if $r_2\leq r_p/K$, the following holds: $d\sq A_{\mc{S}}(p; r_1/K, Kr_2)$ and, for every simplex $\s\sq d$, we have $\diam(\s)\leq K\cdot\frac{\rho(p,\s)^2}{r_p}$. In particular, we have: 
\[\diam(d)\leq 2Kr_2\lesssim 2K\sqrt{r_p\diam(c)}\leq 2K\sqrt{\diam(c)}.\]

Finally, if $i<k$, we have $\partial d=\Phi_p(\partial\mc{D})=\Phi_p\Psi_p(c)=c$. If instead $i=k$, we obtain $c-\partial d\sq B(p,Kr_1)$ by another application of Lemma~\ref{lem:Phi}(1).
\end{proof}

The fineness of the filling $d$ provided by the previous lemma is not small enough for our purposes. The next result will solve this problem with a refining procedure.

\begin{lem}\label{lem:deep_barycentre_2}
There exist $g$ and $K\geq 1$ such that the following holds for all parabolic points $p\in\mc{S}$ and all $\delta>0$. 

Let $c\in\overline{\mscr{C}}_i(\mc{S})$ be a cycle with $\fin(c)\leq g(\delta)$ and $\mf{b}(c)\in\mc{H}_p$. Then:
\begin{enumerate}
\setlength\itemsep{.1cm}
\item if $i<k$ and $c\sq A(p;r_1,r_2)$ for radii $\delta\leq r_1<r_2\leq r_p/K$, we have $c=\partial d$ for some $d\sq A(p;r_1/K,Kr_2)$ with $\fin(d)\leq\delta$ and $\diam(d)\leq K\sqrt{\diam(c)}$;
\item if $i\leq k$ and $c\sq B(p,r_2)$ for a radius $0<r_2\leq r_p/K$, we have $c=\partial d$ for some $d\sq B(p,Kr_2)$ with $\fin(d)\leq\delta$ and $\diam(d)\leq K\sqrt{\diam(c)}$.
\end{enumerate}
\end{lem}
\begin{proof}
We prove in detail only part~(1), since the argument for part~(2) is identical.

\smallskip
{\bf Part~(1).} Assume that $c\sq A(p;r_1,r_2)$ and $i<k$.

Ensuring that $g(\delta)\leq\delta^2$, we have $\fin(c)\leq \delta^2\leq r_1^2\leq r_1^2/r_p$. Thus, Lemma~\ref{lem:deep_barycentre_1} gives $c=\partial d$ for some $d\sq A(p;r_1/K,Kr_2)$ with $\diam(d)\leq K\sqrt{\diam(c)}$. In addition, for every simplex $\s\in\mscr{F}_*(d)$, we have $\diam(\s)\leq K\cdot\frac{\rho(p,\s)^2}{r_p}$. 

Our only task is now to replace $d$ by a $\delta$--fine chain with the same properties. We will achieve this by constructing a chain map $R_*\colon\mscr{F}_*(d)\ra \mscr{C}_*(\mc{S})$ such that:
\begin{itemize}
\item $R_*(\s)=\s$ for every $\eta_0$--fine simplex $\s\in\mscr{F}_*(d)$ (in particular, $R_0={\rm id}$);
\item for every simplex $\s\in\mscr{F}_j(d)$ with $j\geq 1$, we have $\fin(R_j(\s))\leq\eta_j$ and $\diam(R_j(\s)\cup\s)\leq M_j\cdot\frac{\rho(p,\s)^2}{r_p}$.
\end{itemize}
The constants $\eta_j$ and $M_j$ are defined inductively as follows. Set $M_0:=K$ and $\eta_k:=\delta$. For each $j\geq 0$, let $D_j$ be the constant provided by Lemma~\ref{lem:barycentre_vs_horoball} in relation to the constant $2^5M_j$. We then let $K_{j+1}$ and $g_j$ be the constant and function provided by Corollary~\ref{cor:shallow_barycentre} applied to the constant $D_j$. Finally, we set $\eta_j:=g_j(\eta_{j+1})$ and $M_{j+1}:=3^5K_{j+1}M_j+K$.

We now construct the chain map $R_*$. We proceed by induction on $j$, setting $R_0={\rm id}$. For the inductive step, suppose that the degree--$j$ map $R_j$ has been defined with the above properties and let us construct $R_{j+1}$.

Without loss of generality, assume that $r_2\leq r_p/2M_j$. This implies that, for every simplex $\tau\in\mscr{F}_j(d)$, we have $\diam(R_j(\tau)\cup\tau)\leq \frac{M_jr_2}{r_p}\cdot\rho(p,\tau)\leq \frac{1}{2}\rho(p,\tau)$. In particular, this guarantees that $\rho(p,R_j(\tau))\geq\frac{1}{2}\rho(p,\tau)$. Similarly, assuming that $r_2\leq r_p/2K$, we have $\diam(\s)\leq\frac{1}{2}\rho(p,\s)$ for every simplex $\s\in\mscr{F}_*(d)$.

As a consequence, for each simplex $\s\in\mscr{F}_{j+1}(d)$ we have:
\begin{align*}
\diam(R_j(\partial\s))&\leq 2M_j\cdot\max_{\tau\sq\partial\s}\tfrac{\rho(p,\tau)^2}{r_p}\leq 2^3M_j\cdot\tfrac{\rho(p,\partial\s)^2}{r_p}\leq 2^5M_j\cdot\tfrac{\rho(p,R_j(\partial\s))^2}{r_p}, \\
\rho(p,R_j(\partial\s))&=\min_{\tau\sq\partial\s}\rho(p,R_j(\tau))\leq\min_{\tau\sq\partial\s}\tfrac{3}{2}\rho(p,\tau)\leq \tfrac{9}{4}\rho(p,\s).
\end{align*}

Now, for a simplex $\s\in\mscr{F}_{j+1}(d)$, we need to define $R_{j+1}(\s)$ so that $\partial R_{j+1}(\s)=R_j(\partial\s)$. If $\s$ is $\eta_0$--fine, then so is $\partial\s$, hence $R_j(\partial\s)=\partial\s$ and we set $R_{j+1}(\s):=\s$.

Suppose instead that $\s$ is not $\eta_0$--fine. By the inductive hypothesis, we have $\fin(R_j(\partial\s))\leq\eta_j=g_j(\eta_{j+1})$. In addition, by the above inequalities, there exists a point $z\in [o,\mf{b}(R_j(\partial\s))]$ with $e^{-\eps_0d_\mc{X}(o,z)}=2^5M_j\cdot\tfrac{\rho(p,R_j(\partial\s))^2}{r_p}$.
% slightly cheating: there should be a multiplicative factor depending only on $\delta_0,\eps_0$, but there are too many constants already.
By Lemma~\ref{lem:barycentre_vs_horoball} and our choice of $D_j$, we then have $d_\mc{X}(z,\mc{K}_p)\leq D_j$. In particular, $z$ is not $D_j$--deep in any horoball, so Corollary~\ref{cor:shallow_barycentre}(1) implies that $R_j(\partial\s)=\partial\s'$ for a discrete $(j+1)$--chain $\s'\sq\mc{S}$ with $\fin(\s')\leq\eta_{j+1}$ and $\diam(\s')\leq K_{j+1}e^{-\eps_0d_\mc{X}(o,z)}$. It follows that:
\[\diam(\s\cup\s')\leq K\cdot\tfrac{\rho(p,\s)^2}{r_p}+2^5K_{j+1}M_j\cdot\tfrac{\rho(p,R_j(\partial\s))^2}{r_p}\leq (K+3^5K_{j+1}M_j)\cdot\tfrac{\rho(p,\s)^2}{r_p}.\]
Thus, we can set $R_{j+1}(\s):=\s'$.

This completes the construction of $R_*$. We now bound the diameter of $R_{i+1}(d)$ and show that this chain is contained in a suitable annulus around $p$.

Recall that $d\sq A(p;r_1/K,Kr_2)$ and that, for every simplex $\s\in\mscr{F}_{i+1}(d)$, we have $\diam(R_{i+1}(\s)\cup\s)\leq\frac{1}{2}\rho(p,\s)$ and $\rho(p,R_{i+1}(\s))\geq\frac{1}{2}\rho(p,\s)$. This immediately implies that $R_{i+1}(d)\sq A(p;r_1/2K,2Kr_2)$.

Without loss of generality, assume that $r_2=\max_{\xi\in c}\rho(p,\xi)$. Since $\mf{b}(c)\in\mc{H}_p$, Lemma~\ref{lem:barycentre_vs_horoball}(3) guarantees that $\diam(c)\gtrsim r_2^2/r_p$. Set $M:=\max_{j\leq k}M_j$. For every simplex $\s\in\mscr{F}_*(d)$, we have: 
\[\diam(R_*(\s)\cup\s)\leq M\cdot\tfrac{\rho(p,\s)^2}{r_p}\leq MK^2\cdot\tfrac{r_2^2}{r_p}\lesssim MK^2\diam(c)\lesssim MK^2\sqrt{\diam(c)}.\]
Since $\diam(d)\leq K\sqrt{\diam(c)}$, we obtain $\diam(R_{i+1}(d))\lesssim (MK^2+K)\sqrt{\diam(c)}$.

Finally, define $g(\delta)=\delta^2\wedge \eta_0$, recalling that $\eta_0=g_0g_1\dots g_{k-1}(\delta)$. Also note that $c\in\mscr{F}_*(d)$ since $c=\partial d$. Thus, if $c$ is $g(\delta)$--fine, we have $R_i(c)=c$ and, in particular, $c=\partial R_{i+1}(d)$. As shown above, $R_{i+1}(d)$ is $\delta$--fine, it is contained in a suitable annulus around $p$ and it admits the required diameter bound. This proves part~(1).

\smallskip
{\bf Part~(2).} Assume that $c\sq B(p,r_2)$ and $i\leq k$. Set $r_1:=\delta$.

Exactly as in part~(1), the combination of Lemma~\ref{lem:deep_barycentre_1} and the construction of a refining chain map $R_*$ (this time defined on the chain complex $\mscr{F}_*(d)\oplus\mscr{F}_*(c)$) yields a discrete chain $d'$ and a constant $K'\geq 1$ such that $\fin(d')\leq\delta$ and $\diam(d')\leq K'\sqrt{\diam(c)}$ and $c-\partial d'\sq B(p,K'\delta)$. If $r_2\leq\delta$, we simply take $d'=0$.

Now, define $d''$ as the cone over $c-\partial d'$ from $p$ (Definition~\ref{defn:cone}). We have $\partial d''=c-\partial d'$ and $d''\sq B(p,K'\delta)$, hence $\fin(d'')\leq\diam(d'')\leq 2K'\delta$.

In conclusion, setting $d:=d'+d''\sq B(p,K'r_2)$, we have $\partial d=c$ and $\fin(d)\leq 2K'\delta$. In addition, note that we can assume that $\diam(c)\geq\delta^2$ (otherwise, we could have just taken $d$ to be instead the cone over $c$ from any of its vertices, since $\delta^2\leq\delta$). Thus, we have $\diam(d)\leq 2K'\delta+K'\sqrt{\diam(c)}\leq 3K'\sqrt{\diam(c)}$. 

This proves part~(2).
\end{proof}

Finally, we can use Corollary~\ref{cor:shallow_barycentre} to remove from Lemma~\ref{lem:deep_barycentre_2} the requirement that $\mf{b}(c)\in\mc{H}_p$. That is, we can prove Proposition \ref{prop:Bowditch_filling}.

\begin{proof}[Proof of Proposition \ref{prop:Bowditch_filling}]
Let $g_1,K_1$ be the function and constant provided by Lem\-ma~\ref{lem:deep_barycentre_2}. Choose $D\geq 0$ such that, if $\mf{b}(c)$ is $D$--deep in some horoball $\mc{H}_q$, then $c\sq B(q,r_q/K_1)$. Now, let $g_2,K_2$ be the function and constant provided by Corollary~\ref{cor:shallow_barycentre} in relation to the constant $D$. Finally, set $g:=g_1\wedge g_2$ and $K':=K_1\vee K_2$. 

Suppose that $c\in\overline{\mscr{C}}_i(\mc{S})$ is $g(\delta)$--fine and $i\leq k$. If $\mf{b}(c)$ is not $D$--deep in any horoball, then $c=\partial d$ with $\fin(d)\leq\delta$ and $\diam(d)\leq K'\diam(c)$, by Corollary~\ref{cor:shallow_barycentre}. If $\mf{b}(c)$ is $D$--deep in a horoball $\mc{H}_q$, then $c=\partial d$ with $\fin(d)\leq\delta$ and $\diam(d)\leq 2r_q\wedge K'\sqrt{\diam(c)}$, by Lemma~\ref{lem:deep_barycentre_2}(2). 

This proves part~(0). Regarding parts~(1) and~(2), we are left to show that the filling $d$ constructed in the previous paragraph lies in the required annulus/ball around the parabolic point $p$. For this, we can assume that $\mf{b}(c)\not\in\mc{H}_p$, otherwise a suitable filling of $c$ is provided by Lemma~\ref{lem:deep_barycentre_2}. 

Since $\mf{b}(c)\not\in\mc{H}_p$, there exists a constant $C$, depending only on $\delta_0,\eps_0$, such that $\diam(c)\leq C \rho(p,c)^2/r_p$. In addition, if $\mf{b}(c)\in\mc{H}_q$ for some parabolic point $q\in\mc{S}$, then we similarly have $r_q\leq C\rho(p,c)^2/r_p$. Note that $\rho(p,c)\leq r_2$ in part~(2), while we can assume without loss of generality that $\rho(p,c)=r_1$ in part~(1). 

Now, recall that the filling $d$ satisfies $\diam(d)\leq 2r_q$ if $\mf{b}(c)$ is deep in a horoball $\mc{H}_q$, and $\diam(d)\leq K'\diam(c)$ otherwise. Thus, provided that $r_2\leq \frac{r_p}{2K'C} \wedge \frac{r_p}{4C}$, we have $\diam(d)\leq r_2/2$ in part~(2), and $\diam(d)\leq r_1/2$ in part~(1). This implies that $d$ is contained in the ball $B(p,2r_2)$ in part~(2), and in the annulus $A(p; r_1/2,2r_2)$ in part~(1).
\end{proof}

\section{Filling discrete cycles in the Morse boundary.}\label{sect:Morse_filling}

This section has two main goals, namely Proposition \ref{diameter bound} and Proposition \ref{prop:representing_homology}, which both hold for relatively hyperbolic groups with sphere boundary (see below for the precise assumptions). Proposition \ref{diameter bound} roughly says that sufficiently low-dimensional cycles in a given stratum of the Morse boundary can be filled in a controlled stratum and with controlled fineness and diameter. While Proposition \ref{diameter bound} will be used for our vanishing results for \v{C}ech cohomology, Proposition \ref{prop:representing_homology} will be used for our non-vanishing results. This proposition says that, for a given finite collection of parabolic points $F$ in the Bowditch boundary $\mc{S}$, we can represent the entire homology of $\mc{S}-F$ by cycles of arbitrarily small fineness that are contained in some fixed stratum of the Morse boundary.

\subsection{Notation}

If $B$ is a metric ball in a metric space, we denote by $\mf{r}(B)$ its radius, and by $\lambda B$ the ball with the same centre as $B$ and $\mf{r}(\lambda B)=\lambda\mf{r}(B)$. If $\mscr{B}$ is a family of metric balls, we write $\lambda\mscr{B}=\{\lambda B\mid B\in\mscr{B}\}$, and denote by $\mf{r}_{\max}(\mscr{B})$ and $\mf{r}_{\min}(\mscr{B})$, respectively, the supremum and the infimum of radii of balls in $\mscr{B}$.

Recall that we use the notation $a\vee b:=\max\{a,b\}$ and $a\wedge b:=\min\{a,b\}$.

\subsection{Assumptions}\label{subsect:ass_3}

Let $(G,\mc{P})$ be a relatively hyperbolic pair where $G$ is finitely generated and $\mc{P}$ consists of finitely many finitely generated subgroups. Throughout most of Section~\ref{sect:Morse_filling}, we will further assume as in Section~\ref{sect:Bowditch_filling} that the elements of $\mc{P}$ are virtually nilpotent and that the Bowditch boundary $\mc{S}$ is homeomorphic to a sphere $S^{k+1}$ for some $k\geq 1$. The only exceptions will be Proposition \ref{prop:Morse_Bow} and Theorem \ref{thm:path_conn}, which we state under more general hypotheses.

Choose a visual metric $\rho$ on $\mc{S}$. Let $g$ and $K$ be the corresponding function and constant provided by Proposition~\ref{prop:Bowditch_filling}. Without loss of generality, we can assume that $g(\delta)\leq\delta$ for all $\delta>0$. To simplify the proof of Lemma~\ref{lem:general_detour} below, we also assume that $K\geq 9$.

Each parabolic point $p\in\mc{S}$ comes with a characteristic radius $r_p$, as defined in Subsection~\ref{subsec:setup}. We fix a constant $M\geq 2K^2$ and consider the following collection of balls in $\mc{S}$:
\[ \mf{B}:=\{ B(p,r_p/M) \mid p\in\mc{S} \text{ parabolic}\}. \]
Up to enlarging $M$, we can and will assume (see e.g. \cite[Lemma 3.2]{Mackay-Sisto}) that, whenever two balls $B,B'\in\mf{B}$ with $\mf{r}(B)\leq \mf{r}(B')$ have nonempty intersection, we have:
\[\tag{$\ast$} \mf{r}(B) \leq \frac{1}{K^4} \cdot \mf{r}(B'). \label{eq:small_min} \]

For every $n\geq 1$, we define the sets:
\[ X_n:=\mc{S}-\bigcup_{B\in\mf{B}}\tfrac{1}{n}B. \]
We endow each $X_n$ with the subspace topology induced by $\mc{S}$. 

Under weak assumptions on the peripherals, the Morse boundary of a relatively hyperbolic group can be described as the direct limit of the strata $X_n$ (we do not yet assume that the Bowditch boundary is a sphere or that the peripherals are virtually nilpotent).

\begin{prop}
\label{prop:Morse_Bow}
Let $(G,\mc{P})$ be a relatively hyperbolic pair where $G$ is finitely generated and $\mc{P}$ consists of finitely many finitely generated groups with empty Morse boundary. Then the Morse boundary $\partial_*G$ is homeomorphic to $\varinjlim X_n$.
\end{prop}

\begin{proof}
This is proven for fundamental groups of finite volume hyperbolic manifolds in \cite[Proposition 6.4]{charney2019complete}, but the argument works more generally replacing $\mathbb H^3$ with the cusped space for $G$, up to minor modifications that we now discuss. All references below are from \cite{charney2019complete}. 

Lemma 6.1 uses CAT(0)-ness of the neutered space under consideration, but only to deduce that Morse geodesics are strongly contracting (quantitatitvely). This still holds since the neutered space is relatively hyperbolic, and by \cite{Sisto-distformrelhyp} a geodesic in a relatively hyperbolic space is strongly contracting if for every $C>0$ there exists $B>0$ such that the intersections of the geodesic with the $C$–neighborhoods of the peripheral subsets have diameter bounded by $B$. Lemma 6.1 also uses that horospheres are flats, but only to the extent that they do not contain Morse rays. In Lemma 6.2 the visual metric used is $e^{-(\cdot|\cdot)}$, while in general it is $e^{-\epsilon(\cdot|\cdot)}$ for some small positive $\epsilon$; this does not affect the argument. The proof of part (2) of the same lemma uses convexity of horoballs to conclude that the intersection of the ray $\gamma$ under consideration and the horoball is a subgeodesic. In a general cusped space, this intersection lies within uniformly bounded Hausdorff distance of a subgeodesic, so this does not cause issues. In the same part of the proof, the argument involving the midpoint of a geodesic in a horoball with endpoints on the horosphere can be performed equally well in a combinatorial horoball rather than in the half-space model for $\mathbb H^3$. Finally, the argument in the proof of Proposition 6.4 again uses that horospheres are flats, but again to the same extent as the proof of Lemma 6.1
\end{proof}

As an aside, in view of the description of Morse boundaries given above, we point out that results of \cite{Mackay-Sisto} yield path-connectedness (under weaker assumptions than the boundary being a sphere):

\begin{thm}
\label{thm:path_conn}
    Let $(G,\mc{P})$ be a relatively hyperbolic pair where $G$ is finitely generated and $\mc{P}$ consists of finitely many finitely generated, virtually nilpotent subgroups. Suppose that $G$ is one-ended and does not split over a subgroup conjugate into some peripheral group. Then $\partial_*G$ is path-connected.
\end{thm}

\begin{proof}
    The hypotheses of the theorem are the same as \cite[Theorem 1.2]{Mackay-Sisto}, which is a special case of \cite[Theorem 1.3]{Mackay-Sisto}. In the proof of said theorem it is argued that \cite[Corollary 7.4]{Mackay-Sisto} applies, which is a result that allows one to connect pairs of points in the Bowditch boundary with arcs that avoid balls around parabolic points. By the description above, each of these arcs is contained in (a stratum of) the Morse boundary.
\end{proof}

The next lemma allows us to approximate each stratum $X_n$ by the complement of a \emph{finite} collection of balls in $\mc{S}$. 

\begin{lem}\label{lem:avoid_small_balls}
For every $n\geq 1$ and $\eps>0$ there exists $r=r(n,\eps)>0$ such that $X_n$ is $\eps$--dense in the set:
\[\mc{S}-\bigcup \left\{\tfrac{1}{n}B \mid B\in\mf{B},\ \mf{r}(B)\geq r\right\} .\]
\end{lem}
\begin{proof}
Fix $n\geq 1$ and $\eps>0$. Consider, for $m\geq 1$, the compact sets:
\[ K_m:=\mc{S}-\bigcup \left\{\tfrac{1}{n}B \mid B\in\mf{B},\ \mf{r}(B)\geq\tfrac{1}{m}\right\}.\]
Note that we have $\bigcap_{m\geq 1} K_m=X_n$. If there exists $\overline m$ such that $X_n$ is $\eps$--dense in $K_{\overline m}$, we can set $r(n,\eps):=1/{\overline m}$. Otherwise, for each $m\geq 1$ there exists a point $x_m\in K_m$ with $\rho(x_m,X_n)\geq \eps$. Passing to a subsequence, we have $x_m\ra x$ with both $x\in\bigcap_{m\geq 1} K_m=X_n$ and $\rho(x,X_n)\geq\eps>0$, a contradiction.
\end{proof}

\subsection{Detouring}

Recall that, from now on, we assume that $\mc{P}$ consists of virtually nilpotent subgroups and the Bowditch boundary $\mc{S}$ is homeomorphic to a sphere $S^{k+1}$ with $k\geq 1$.

In order to construct fillings in the Morse boundary $X=\varinjlim X_n$, our strategy will be to construct fillings in the Bowditch boundary $\mc{S}$ using Proposition~\ref{prop:Bowditch_filling}, and then to detour these around suitable collections of balls in $\mc{S}$, thus pushing the fillings into a stratum of $X$. The final result of this type will be Proposition~\ref{diameter bound}.

To begin with, the following lemma explains how to detour chains in $\mc{S}$ around collections of \emph{pairwise disjoint} balls. Recall that $g$ and $K$ were introduced in Proposition~\ref{prop:Bowditch_filling}. 

\begin{lem}\label{lem:disjoint_detour}
Consider a finite subset $\mscr{B}\sq\tfrac{1}{N}\mf{B}$, for some $N\geq 1$. Suppose that $2K\mscr{B}$ is pairwise disjoint. Let $d\in\mscr{C}_i(\mc{S})$ be a $g(\delta)$--fine discrete chain, where $1\leq i\leq k$ and $0<\delta\leq\mf{r}_{\min}(\mscr{B})$. Suppose that $\supp(\partial d)$ is disjoint from all balls in $\mscr{B}$. 

Then there exists a discrete chain $d'$ in $\mc{S}$ such that:
\begin{enumerate}
\item $d'$ is $\delta$--fine and $\partial d'=\partial d$;
\item $\supp(d')$ is disjoint from all balls in $\tfrac{1}{K}\mscr{B}$;
\item $d$ and $d'$ coincide outside the union of the balls in $2K\mscr{B}$;
\item $\diam(d')\leq \diam(d)+4K\sqrt{\mf{r}_{\max}(\mscr{B})}$;
% a sharper, but unnecessary bound is: $\diam(d')\leq \diam(d)+2K\sqrt{\diam(d)\wedge 4\mf{r}_{\max}(\mscr{B})}$;
\item if $d$ is $g(g(\delta))$--fine, then $d-d'=\partial e$ for a $\delta$--fine discrete $(i+1)$--chain $e$ contained in the union of the balls in $2K^2\mscr{B}$.
\end{enumerate}
\end{lem}
\begin{proof}
We begin by noting that, for every ball $B\in\mscr{B}$ centred at a parabolic point $p\in\mc{S}$, we have $2K\mf{r}(B)\leq r_p/K$, since we have chosen $M\geq 2K^2$ in the definition of $\mf{B}$. This will allow us to freely apply Proposition~\ref{prop:Bowditch_filling} later in the proof.

\begin{figure}[ht] 
	%	\hspace{5cm}
	%	\captionsetup{width=1.8\linewidth}
	\centering
	\includegraphics{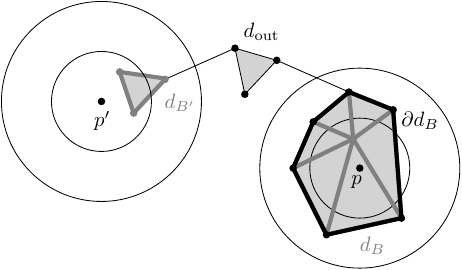}
	\caption{The decomposition of the discrete chain $d$ into parts close to parabolic points plus the rest $d_{\rm out}$.}\label{dout}
\end{figure} 

Write $d=d_{\rm out}+\sum_{B\in\mscr{B}}d_B$, where $d_B$ is obtained by grouping all simplices in a reduced expression for $d$ whose support is entirely contained in $2B$; this is illustrated in Figure \ref{dout}. Note that the balls in $2\mscr{B}$ are pairwise disjoint, since $K\geq 1$.

Observe that $\supp(d_{\rm out})$ is disjoint from all balls $B\in\mscr{B}$. Indeed, since $d$ is $g(\delta)$--fine and $g(\delta)\leq\delta\leq\mf{r}_{\min}(\mscr{B})$, every simplex $\s$ appearing in $d_{\rm out}$ has diameter at most $\mf{r}(B)$ and intersects the complement of $2B$, for each $B\in\mscr{B}$.

Similarly, each $\supp(\partial d_B)$ is contained in the annulus $2B-B$. To see this, consider a codimension--$1$ face $\tau$ of a simplex $\s$ appearing in $d$, and suppose that $\supp(\tau)\cap B\neq\emptyset$. Then $\supp(\s)\sq 2B$, since $d$ is $\mf{r}(B)$--fine, as observed above. It follows that all simplices of $d$ containing $\tau$ appear in $d_B$. Since $\supp(\partial d)\cap B=\emptyset$, the simplex $\tau$ does not appear in $\partial d$, hence it cannot appear in $\partial d_B$ either.

Now, $\partial d_B$ is a $g(\delta)$--fine discrete, reduced $(i-1)$--chain in the annulus $2B-B$ and we have $i-1\leq k-1$. Proposition~\ref{prop:Bowditch_filling}(1) yields a $\delta$--fine discrete $i$--chain $d_B'$ in the annulus $2KB-\tfrac{1}{K}B$ with $\partial d_B'=\partial d_B$ and $\diam(d_B')\leq K\sqrt{\diam(\partial d_B)}$. If $\partial d_B=0$, we simply take $d_B'=0$.
% requirement: $2\mf{r}(B)\leq r_p/K$.
Define $d':=d_{\rm out}+\sum_{B\in\mscr{B}}d_B'$.

Checking that $d'$ satisfies (1) and (3) is straightforward. Moreover, recall that $\supp(d_{\rm out})$ is disjoint from all balls $B\in\mscr{B}$. Since $\supp(d_B')\sq 2KB-\tfrac{1}{K}B$ and the balls in $2K\mscr{B}$ are pairwise disjoint, we see that $d'$ also satisfies (2).

Regarding property~(4), note that every point of $\supp(d_B')$ is at distance at most $K\sqrt{\diam(\partial d_B)}$ from a point of $\supp(\partial d_B)\sq\supp(d)$. Since we have both $\diam(\partial d_B)\leq\diam(d)$ and $\diam(\partial d_B)\leq \diam(2B)\leq 4\mf{r}_{\max}(\mscr{B})$, this yields the required inequality.

Finally, we prove property~(5). If $d$ is $g(g(\delta))$--fine, Pro\-po\-si\-tion~\ref{prop:Bowditch_filling}(2) ensures that $d_B-d_B'=\partial e_B$ for a $\delta$--fine discrete chain $e_B$ contained in the ball $2K^2B$. 
% requirement: $2K\mf{r}(B)\leq r_p/K$.
Thus, it suffices to set $e:=\sum_{B\in\mscr{B}}e_B$. 
\end{proof}

We now discuss detouring around general collections of balls, which are allowed to intersect each other. For a function $f$ and an integer $n\geq 0$, we denote by $f^{(n)}$ the $n$--fold composition of $f$ (in particular, $f^{(0)}={\rm id}$ and $f^{(1)}=f$).

\begin{lem}\label{lem:general_detour}
Consider a finite subset $\mscr{B}\sq\tfrac{1}{N}\mf{B}$ for some $N\geq 2K$. Define 
\[m:=1+\lfloor\log_K(\mf{r}_{\max}(\mscr{B})/\mf{r}_{\min}(\mscr{B}))\rfloor.\] 
Let $d\in\mscr{C}_i(\mc{S})$ be a $g^{(m)}(\delta)$--fine discrete chain with $1\leq i\leq k$ and $0<\delta\leq\mf{r}_{\min}(\mscr{B})$. Suppose that $\supp(\partial d)$ is disjoint from all balls in $\mscr{B}$.

Then there exists a discrete chain $d'$ in $\mc{S}$ such that:
\begin{enumerate}
\item $d'$ is $\delta$--fine and $\partial d'=\partial d$;
\item $\supp(d')$ is disjoint from all balls in $\tfrac{1}{2K}\mscr{B}$;
\item $\supp(d')$ is contained in the $8K\mf{r}_{\max}(\mscr{B})$--neighbourhood of $\supp(d)$;
% in general, this is not true of the Hausdorff distance: to get $d'$, we might have to erase some bits of $d$ that are *cycles* inside balls B
\item if $\partial d\neq 0$, then $\diam(d')\leq \diam(d)+8K\sqrt{2K\diam(d)}$;
\item if $d$ is $g^{(m+1)}(\delta)$--fine, then $d-d'=\partial e$ for a $\delta$--fine discrete chain $e$ contained in the $(4K^2+8K)\cdot\mf{r}_{\max}(\mscr{B})$--neighbourhood of $\supp(d)$.
\end{enumerate}
\end{lem}
\begin{proof}
We begin by partitioning $\mscr{B}=\mscr{B}_1\sqcup\dots\sqcup\mscr{B}_m$ so that a ball $B$ lies in $\mscr{B}_i$ exactly when $\mf{r}(B)$ lies in the half-open interval $\left[\mf{r}_{\max}(\mscr{B})/K^{i-1},\mf{r}_{\max}(\mscr{B})/K^i\right)$. Note that, in view of Equation~(\ref{eq:small_min}) and the fact that $N\geq 2K$, each set $2K\mscr{B}_i$ consists of pairwise disjoint balls.

We then set $d_0:=d$ and iteratively define $d_i$ as the chain obtained by applying Lemma~\ref{lem:disjoint_detour} to $d_{i-1}$ and the family of balls $\mscr{B}_i$. Finally, we set $d':=d_m$.

Note that each application of Lemma~\ref{lem:disjoint_detour} is allowed. Indeed, the chain $d_i$ is $g^{(m-i)}(\delta)$--fine and $g^{(m-i)}(\delta)\leq\delta\leq \mf{r}_{\min}(\mscr{B})\leq \mf{r}_{\min}(\mscr{B}_{i+1})$. In addition, we have $\partial d_i=\partial d$, which is disjoint from all balls in $\mscr{B}$. 

We are left to check that the final chain $d'=d_m$ satisfies all required properties. The previous paragraph shows that $d_m$ is $\delta$--fine and $\partial d_m=\partial d$, proving (1).

Regarding (3), note that $d_{i+1}$ and $d_i$ only differ within a union of balls that intersect $\supp(d_i)$ and lie in $2K\mscr{B}_{i+1}$. Hence $\supp(d_{i+1})$ is contained in a neighbourhood of $\supp(d_i)$ of radius at most $2\cdot\mf{r}_{\max}(2K\mscr{B}_{i+1})\leq 4K\cdot\mf{r}_{\max}(\mscr{B})/K^i$. Recalling that $K\geq 2$, we deduce that $\supp(d_m)$ is contained in a neighbourhood of $\supp(d)$ of radius at most:
\[ 4K\cdot\mf{r}_{\max}(\mscr{B})\cdot\sum_{i\geq 0}\frac{1}{K^i}\leq8K\cdot\mf{r}_{\max}(\mscr{B}).\]

Now, we prove (2). Consider a ball $B\in\mscr{B}$ and say it lies in $\mscr{B}_i$. By Lemma~\ref{lem:disjoint_detour}, the chain $d_i$ is supported outside $\tfrac{1}{K}B$. We need to check that this property is not affected too much by subsequent detouring.

First, observe that the chain $d_{i+3}$ is also supported outside $\tfrac{1}{K}B$. Indeed, setting for a moment $\mscr{B}':=\mscr{B}_{i+1}\cup\mscr{B}_{i+2}\cup\mscr{B}_{i+3}$, the chains $d_i$ and $d_{i+3}$ only differ within the balls in $2K\mscr{B}'$. And, for every $B'\in\mscr{B}'$, the ball $2KB'$ is disjoint from $\tfrac{1}{K}B$. Indeed, if this were not the case, the balls $2KB'$ and $2KB$ would also intersect. These two balls are rescalings of elements of $\mf{B}$ by the same factor $\geq 1$, because $N\geq 2K$ and $B,B'\in\tfrac{1}{N}\mf{B}$. Thus, Equation~(\ref{eq:small_min}) would imply that $\mf{r}(B')\leq \mf{r}(B)/K^4$, contradicting the fact that, since $B'\in\mscr{B}'$, we have $\mf{r}(B')>\mf{r}(B)/K^4$.

Now, recalling that $K\geq 9$, we obtain that $\supp(d_m)$ is contained in a neighbourhood of $\supp(d_{i+3})$ of radius at most:
\begin{align*}
4K\cdot\sum_{j\geq i+4}\mf{r}_{\max}(\mscr{B}_j)&\leq 4K\cdot\mf{r}_{\max}(\mscr{B}_{i+1})\cdot\sum_{j\geq 3}\tfrac{1}{K^j}\\
&\leq 4K\cdot\mf{r}_{\max}(\mscr{B}_{i+1})\cdot\tfrac{1}{8K^2}<\tfrac{1}{2K}\cdot\mf{r}(B).
\end{align*}
In conclusion, $\supp(d_m)$ is disjoint from $\tfrac{1}{2K}B$, for every $B\in\mscr{B}$. 

Let us prove (4). Let $1\leq k\leq m$ be the smallest integer such that $d_k\neq d_{k-1}$. In particular, we have $d=d_{k-1}$ and there exists a ball $B\in\mscr{B}_k$ such that $\supp(d)\cap \tfrac{1}{K}B\neq\emptyset$ (otherwise we could have simply taken $d_k=d_{k-1}$, rather than applying Lemma~\ref{lem:disjoint_detour}). Since $\supp(\partial d)$ is nonempty and disjoint from $B$, it follows that $\diam(d)\geq\tfrac{1}{2}\mf{r}(B)\geq\tfrac{1}{2K}\mf{r}_{\max}(\mscr{B}_k)$.
% the assumption that $\partial d\neq\emptyset$ shouldn't really be necessary here, but it does simplify the proof a lot

Now, by part~(4) of Lemma~\ref{lem:disjoint_detour} and recalling that $\sqrt{K}\geq 3$, we obtain: \begin{align*}
\diam(d_m)&\leq \diam(d) + 4K\sum_{j\geq k}\sqrt{\mf{r}_{\max}(\mscr{B}_j)} \\ 
&\leq\diam(d)+4K\sqrt{\mf{r}_{\max}(\mscr{B}_k)}\sum_{j\geq 0}K^{-j/2}\leq \diam(d)+ 8K\sqrt{2K\diam(d)}.
\end{align*}

Finally, we prove (5). Part~(5) of Lemma~\ref{lem:disjoint_detour} gives $\delta$--fine chains $e_1,\dots,e_k$ with $d_i-d_{i-1}=\partial e_i$ and $\supp(e_i)$ contained in the union of the balls in $2K^2\mscr{B}_i$ that intersect $\supp(d_{i-1})$. In particular, $\supp(e_i)$ is contained in a neighbourhood of $\supp(d_{i-1})$ of radius at most 
$4K^2\cdot\mf{r}_{\max}(\mscr{B})$, and hence in a neighbourhood of $\supp(d)$ of radius at most $(4K^2+8K)\cdot\mf{r}_{\max}(\mscr{B})$, by the proof of property~(3). We conclude by setting $e:=e_1+\dots + e_k$.
\end{proof}

The detouring construction developed in Lemma~\ref{lem:general_detour} allows us to fill discrete cycles in a given stratum of $X$ while remaining in a controlled larger stratum of $X$. This is the content of the next proposition, which is the main ingredient for the vanishing part of Theorems~\ref{thm:main_intro} and~\ref{thm:general_intro}.

Recall that $\overline{\mscr{C}}_*(\mc{S})$ denotes the chain complex of reduced discrete chains in $\mc{S}$, where $0$--chains are required to have zero coefficient sum.

\begin{prop}\label{diameter bound} 
There exist functions $f_1(n)\geq n$, $0<g_1(\delta,n)\leq\delta$ and $\mathfrak{h}(r)$ (going to $0$ for $r\ra 0$) such that the following holds for all $\delta>0$, $n\geq 1$, $0\leq i<k$.

If $c\in\overline{\mscr{C}}_i(\mc{S})$ is a $g_1(\delta,n)$--fine discrete cycle supported in $X_n$, then there exists a $\delta$--fine chain $d$ supported in $X_{f_1(n)}$ with $\partial d=c$ and $\diam(d)\leq \mathfrak{h}(\diam(c))$.
\end{prop}
\begin{proof}
Define $f_1(n):=2K(n\vee 2K)$ and set $N:=f_1(n)$ for simplicity. Define $\overline r:=\tfrac{1}{N}\cdot r\big(N,\delta/3\big)$, where $r(\cdot,\cdot)$ is the function in Lem\-ma~\ref{lem:avoid_small_balls}, and consider:
\[ \mscr{B}:=\left\{B\in\tfrac{1}{N}\mf{B} \mid \mf{r}(B)\geq \overline r \right\}. \] 
Set $m=m(\delta,n):=1+\lfloor\log_K(\mf{r}_{\max}(\mscr{B})/\mf{r}_{\min}(\mscr{B}))\rfloor$ and $g_1(\delta,n):=g^{(m+1)}(\overline\delta)$, where $\overline\delta:=\delta/3\wedge 2K\overline r$. Finally, define $\mathfrak{h}(r):=r+Kr^{1/2}+16K^2r^{1/4}$.

Now, let $c$ be a $g_1(\delta,n)$--fine, reduced, discrete $i$--cycle with $i<k$ and $\supp(c)\sq X_n$. Without loss of generality, we can assume that $\delta\leq\diam(c)$. Otherwise, we could simply define $d$ as the cone over $c$ from any of its vertices (Definition~\ref{defn:cone}), which would satisfy $\supp(d)\sq X_n$, $\fin(d)=\diam(d)=\diam(c)$ and $c=\partial d$.

By Proposition~\ref{prop:Bowditch_filling}(0), there exists a $g^{(m)}(\overline\delta)$--fine discrete chain $d'$ in $\mc{S}$ with $\partial d'=c$ and $\diam(d')\leq K\sqrt{\diam(c)}$. Note that $c=\partial d'$ is disjoint from all balls in $2K\mscr{B}$, since $\supp(c)\sq X_n$ and every ball of $2K\mscr{B}$ is contained in one of $\tfrac{1}{n}\mf{B}$.

Applying Lemma~\ref{lem:general_detour} to $d'$ and the family of balls $2K\mscr{B}$, we obtain a $\overline\delta$--fine (hence $\delta/3$--fine) discrete chain $d''$ in $\mc{S}$ such that $\partial d''=\partial d'$ and $\supp(d'')$ is disjoint from all balls in $\mscr{B}$. In addition:
\[\diam(d'')\leq \diam(d')+8K\sqrt{2K\diam(d')}\leq K\sqrt{\diam(c)}+16K^2(\diam(c))^{1/4} .\]

Finally, by Lemma~\ref{lem:avoid_small_balls} and our choice of $\mscr{B}$, every point of $\supp(d'')$ is $\delta/3$--close to a point of $X_N$. Replacing every point of $\supp(d'')$ with a closest point in $X_N$, we obtain a $\delta$--fine chain $d$ with $\partial d=\partial d''=c$ and $\supp(d)\sq X_N$. Finally, since $\delta\leq\diam(c)$, we have $\diam(d)\leq 2\delta/3+\diam(d'')\leq \mathfrak{h}(\diam(c))$, concluding the proof.
\end{proof}

The next proposition is required for the non-vanishing part of Theorems~\ref{thm:main_intro} and~\ref{thm:general_intro}. This is the only place in Section~\ref{sect:Morse_filling} where we explicitly use the fact that $\mc{S}$ is homeomorphic to a sphere. In the rest of the section, we have only used this assumption indirectly, in the form of Proposition~\ref{prop:Bowditch_filling}. 

Before proving Proposition~\ref{prop:representing_homology}, we need the following observation.

\begin{rmk}[Straight chains and straightening]\label{rmk:straight}
Let $(\mathbb{S},d_{\o})$ be a constant-cur\-va\-ture sphere. If $D\sq\mathbb{S}$ is a finite subset with $\diam(D)\leq\diam(\mathbb{S})/2$, then there is a unique singular simplex $\s_D$ that has vertex set $D$ and maps every affine line in the standard simplex to a geodesic in $\mathbb{S}$ parametrised proportionally to arc-length. 

We refer to simplices obtained in this way as \emph{straight simplices}. A singular chain in $(\mathbb{S},d_{\o})$ is \emph{straight} if all its simplices are straight.

If $c$ is a $\diam(\mathbb{S})/2$--fine discrete chain in $(\mathbb{S},d_{\o})$, we define its \emph{straightening} $\wh c$. This is the singular chain obtained by replacing every discrete simplex of $c$ with the straight singular simplex with the same (ordered) vertex set. Note that $\disc(\wh c)=c$.

If $c$ is $\delta$--fine in $(\mathbb{S},d_{\o})$, then $\supp(\wh c)$ is contained in the $\delta$--neighbourhood of $\supp(c)$. Indeed, all metric balls in $(\mathbb{S},d_{\o})$ with radius $\leq\diam(\mathbb{S})/2$ are convex.
\end{rmk}

\begin{prop}\label{prop:representing_homology}
For every finite (nonempty) set of parabolic points $F\sq\mc{S}$, there exists a stratum $X_N$ with the following property. 

For every $\delta>0$, every class in $H_k(\mc{S}-F,\Z)$ is represented by a singular cycle $c$ such that $\disc(c)$ is $\delta$--fine and contained in $X_N$.
\end{prop}

\begin{proof}
Let $d_{\o}$ be a constant-curvature (i.e.\ round) metric on $\mc{S}$. We continue denoting by $\rho$ our chosen visual metric on $\mc{S}$. When speaking of neighbourhoods, radii and fineness, we will always refer to $\rho$ unless explicitly stated otherwise.

Note that $\rho$ and $d_{\o}$ induce the same topology on $\mc{S}$, which is compact. Thus, there exists a weakly increasing function $\xi$ such that $\xi(t)\ra 0$ for $t\ra 0$ and, for all $x,y\in\mc{S}$, we have:
\[ d_{\o}(x,y)\leq \xi(\rho(x,y)).\]

Let $3D>0$ be the minimum distance $d_{\o}(x,y)$ for distinct points $x,y\in F$. Let $\Sigma$ be the union of the $D$--spheres around the points of $F$ in the metric $d_{\o}$. Note that, since $F$ is nonempty, the homology of $\mc{S}-F$ is entirely supported on $\Sigma$.

Up to shrinking $\delta$, we can assume that $\xi(\delta)< D/3$. We then choose the integer $N$ so that:
\[ \xi\Big( \frac{\delta}{3}+(4K^2+8K)\cdot\frac{\diam(\mc{S},\rho)}{N/2K}\Big)\leq \frac{D}{3}.\]
Enlarging $N$ if necessary, we further assume that $N\geq 4K^2$.

As in the proof of Proposition~\ref{diameter bound}, set $\overline r:=\tfrac{1}{N}\cdot r(N,\delta/3)$, where $r(\cdot,\cdot)$ is the function in Lemma~\ref{lem:avoid_small_balls}, and define:
\[ \mscr{B}:=\left\{B\in\tfrac{1}{N}\mf{B} \mid \mf{r}(B)\geq \overline r \right\}. \]
Set $m:=1+\lfloor\log_K(\mf{r}_{\max}(\mscr{B})/\mf{r}_{\min}(\mscr{B}))\rfloor$ and $\delta':=g^{(m+1)}(\delta/3 \wedge 2K\overline r)$. 

Now consider any nonzero class $\alpha\in H_k(\mc{S}-F,\Z)$. Represent $\alpha$ by a straight, singular $k$--cycle $\wh c$ such that $c:=\disc(\wh c)$ is supported on $\Sigma $ and $\delta'$--fine (with respect to the visual metric $\rho$). This possible in view of Remark~\ref{rmk:straight}. 

Applying Lemma~\ref{lem:general_detour} to $c$ and the collection of balls $2K\mscr{B}$, we obtain a $\delta/3$--fine discrete cycle $c'$ supported outside all balls in $\mscr{B}$. In addition, using part~(5) of Lemma~\ref{lem:general_detour}, there exists a $\delta/3$--fine discrete chain $e$ with $\partial e=c-c'$ and $\supp(e)$ contained in the $(4K^2+8K)\cdot\mf{r}_{\max}(2K\mscr{B})$--neighbourhood of $\supp(c)\sq\Sigma$. Note that:
\[ \mf{r}_{\max}(2K\mscr{B})\leq \frac{\diam(\mc{S},\rho)}{N/2K}.\]

By Lemma~\ref{lem:avoid_small_balls} and our choice of $\mscr{B}$, every point of $\supp(c')$ is $\delta/3$--close to a point of $X_N$. Thus, we can modify $c'$ to a $\delta$--fine discrete cycle $c''$, and $e$ to a $\delta$--fine discrete chain $e'$, so that $\partial e'=c-c''$ and that $\supp(c'')\sq X_N$. Note that $\supp(e')$ is contained in the $\delta/3$--neighbourhood of $\supp(e)$. 

Since $\xi(\delta)<D/3$, the chain $e'$ is $D/3$--fine with respect to the metric $d_{\o}$. In addition, our choice of $N$ guarantees that $\supp(e')$ is contained in the $D/3$--neighbourhood of $\supp(c)\sq\Sigma$ with respect to the metric $d_{\o}$. Remark~\ref{rmk:straight} then shows that the straightening $\wh e'$ of $e'$ is supported in the $2D/3$--neighbourhood of $\Sigma$, which is disjoint from $F$.
 
Finally, denoting by $\wh c''$ the straightening of $c''$, and recalling that $\wh c$ was straight to begin with, we obtain $\partial \wh e'=\wh c-\wh c''$. Since $\supp(\wh e')\cap F=\emptyset$, this shows that $\alpha=[\wh c]=[\wh c'']$ in $H_k(\mc{S}-F,\Z)$. Since $\disc(\wh c'')=c''$, which is $\delta$--fine and supported on $X_N$, this concludes the proof.
\end{proof}

We will also need the following relation between singular and discrete cycles, for which it is convenient to use the constant-curvature metric $d_\o$.

\begin{lem} \label{lem:convex_cover}
Let $F$ be a finite subset of $(\mc{S}, d_\o)$. Every open cover $\mc U$ of $\mc {S} - F$ containing only convex open sets of diameter $\leq \diam(\mc{S})/2$ satisfies the following property. 

If $c$ is a straight $\mc{U}$--fine singular cycle in $\mc{S}-F$ and $\disc(c)$ is the boundary of a $\mc{U}$--fine discrete chain, then $[c]=0$ in the singular homology of $\mc{S}-F$.
\end{lem}

\begin{proof}
    Suppose that $\disc(c)=\partial d$, for a $\mc{U}$--fine discrete chain $d$. Let $\wh c$ and $\wh d$ be the straightenings of $\disc(c)$ and $d$, as defined in Remark~\ref{rmk:straight}. By convexity of the elements of $\mc{U}$, the singular chains $\wh c$ and $\wh d$ are again $\mc{U}$--fine (and contained in $\mc{S}-F$). We have $\partial\wh d=\wh c$ by construction and, moreover, the singular cycle $c-\wh c$ is also the boundary of a singular chain that is $\mc{U}$--fine (hence contained in $\mc{S}-F$), again by convexity of $\mc{U}$. In conclusion, $c$ is the boundary of singular chain contained in $\mc{S}-F$.
\end{proof}

\begin{rem}\label{rem:convex_cover_exists}
It is easy to see that, for every finite subset $F\sq\mc{S}$, the complement $\mc{S} - F$ admits a convex cover as in the lemma.
\end{rem}

\section{Vanishing of \v{C}ech cohomology}\label{sect:vanishing}

Throughout Section~\ref{sect:vanishing}, we again consider a relatively hyperbolic group $G$, its Morse boundary $X:=\partial_*G$ and the filtration $X=\varinjlim X_n$. Notation and assumptions are exactly the same as in Section~\ref{sect:Morse_filling} (in particular, see Subsection~\ref{subsect:ass_3} for the precise assumptions on $G$).

We fix a visual metric $\rho$ on the Bowditch boundary and equip $X$ and all strata $X_n$ with its restriction. We stress once more that $\rho$ does not induce the topology of $X$, though it does induce that of each $X_n$. When speaking e.g.\ of fineness of chains below, this will always be meant with respect to this metric.

Our only tool in Section~\ref{sect:vanishing} is Proposition~\ref{diameter bound}, so all results in this section hold more generally for any countable direct limit $X=\varinjlim X_n$ of compact metric spaces $X_n$ satisfying Proposition~\ref{diameter bound}. 

The main result for this section is the vanishing of \v{C}ech cohomology in low degrees, proving the first part of Theorem~\ref{thm:general_intro} and, consequently, also the first part of Theorem \ref{thm:main_intro}.

\begin{thm}\label{thm:vanishing}
Let $\mc{R}$ be a principal ideal domain. Then, for $X=\partial_*G$, we have $\cH^i(X,\mc{R})=\{0\}$ for $1\leq i<k$.
\end{thm}

All chains and simplices in the following discussion are assumed to be discrete without explicit mention.

\subsection{Subdivision results}

Let $\mscr{C}_i(X_n)$ be the free $\Z$--module generated by (discrete) $i$--simplices
with support contained in $X_n$. For $\delta>0$, let $\mscr{C}_i(X_n;\delta)$ be the submodule of $\mscr{C}_i(X_n)$ generated by $\delta$--fine simplices. For an open cover $\mc{O}$ of $X$, we denote by $\mscr{C}_i(X_n,\mc{O})$ the submodule of $\mscr{C}_i(X_n)$ formed by $\mc{O}$--fine chains, and by $\mscr{C}_i(X_n,\mc{O};\delta)$ its intersection with $\mscr{C}_i(X_n;\delta)$.

The main aim of this subsection is the following result, proved at the end.

\begin{prop}[Arbitrarily fine subdivisions]\label{arbitrarily fine subdivisions}
For every open cover $\mc{O}$ of $X$, there exist a refinement $\mc{O}'$ and a function $f_2(n)\geq n$ such that the following hold.
\begin{enumerate}
\item For every $\delta>0$, there is a chain map $\phi^{\delta}\colon \mscr{C}_*(X_n,\mc{O}')\ra \mscr{C}_*(X_{f_2(n)},\mc{O};\delta)$ defined in degrees $\leq k$. \label{part1}
\item The composition $\iota\phi^{\delta}\colon \mscr{C}_*(X_n,\mc{O}')\ra \mscr{C}_*(X_{f_2(n)},\mc{O})$ is chain-homotopic to $\iota$, where $\iota$ denotes all standard inclusions. \label{part2}
\end{enumerate}
\end{prop}

In particular, part~(2) of Proposition~\ref{arbitrarily fine subdivisions} implies that, for every cycle $c\in \mscr{C}_i(X_n,\mc{O}')$, the cycles $c$ and $\phi^{\delta}(c)$ are \emph{$\mc{O}$--homologous}, in the sense that there exists an $\mc{O}$--fine chain $d$ such that $\partial d=c-\phi^{\delta}(c)$. 

As a first step towards the proof of Proposition~\ref{arbitrarily fine subdivisions}, we need Proposition~\ref{arbitrarily fine subdivisions 2}, which is an analogue disregarding covers. The proof of Proposition~\ref{arbitrarily fine subdivisions 2} is almost exactly the same, but it relies on Proposition~\ref{diameter bound} rather than Proposition~\ref{moreover} below.

Recall that the face complex $\mscr{F}_*(d)$ of a discrete chain $d$ was introduced in Definition~\ref{defn:face_complex}.

\begin{defn}
Consider $r>0$ and a subset $Y\sq X$. A chain $c$ is \emph{$(r,Y)$--related} to a chain $c'$ if there exists a chain map $T:\mscr{F}_*(c)\to \mscr{C}_*(Y)$ such that $T(c)=c'$ and, for each simplex $\sigma\in \mscr{F}_*(c)$, the support of the chain $T(\sigma)$ is contained in the $r$--neighbourhood of $\supp(\sigma)$.
\end{defn}

\begin{rmk}[Projecting discrete chains]\label{rmk:chain_projection}
Let $c$ be an $i$--chain in $X$. If $\supp(c)$ is contained in the $\eps$--neighbourhood of a subspace $Y\sq X$, for some $\eps>0$, we can project $c$ to a chain in $Y$ as follows. 

For every point $x\in\supp(c)$, choose a point $x'\in Y$ with $\rho(x,x')<\eps$. Write $c=\sum_j\s_j$, where each $\s_j$ is an $i$--simplex. For each $\s_j=[x_0,\dots,x_i]$, define the new simplex $\s_j'=[x_0',\dots,x_i']$ and set $c':=\sum_j\s_j'$. We will refer to $c'$ as an \emph{$\eps$--projection of $c$ to $Y$}. Note that, by construction, $c$ is $(\eps,Y)$--related to $c'$ and $\fin(c')\leq\fin(c)+2\eps$.
\end{rmk}

\begin{prop}\label{arbitrarily fine subdivisions 2}
There exist functions $f_3(n)\geq n$ and $H(r)$, $H'(r)$ (both going to $0$ for $r\ra 0$) such that the following holds.
\begin{enumerate}
\item For all $\delta>0$ and $n\geq 1$, there is a chain map $\phi^{\delta,n}\colon \mscr{C}_*(X_n)\ra \mscr{C}_*(X_{f_3(n)};\delta)$ defined in degrees $\leq k$. \label{1}
\item If $c\in \mscr{C}_i(X_n;r)$ for some $r>0$ and $0\leq i\leq k$, then $\phi^{\delta,n}_i c$ is contained in the $H(r)$--neighbourhood of $\supp(c)$. 
\item If a cycle $c\in \mscr{C}_i(X_{n+1};r)$ with $0\leq i<k$ is $(r,X_n)$--related to a cycle $c'\in \mscr{C}_i(X_n;r)$, then $\phi^{\delta,n+1}_ic-\phi^{\delta,n}_ic'=\partial B$ for a $\delta$--fine chain $B$ contained in the $H'(r)$--neighbourhood of $\supp(c)$ within $X_{f_3(n)}$.\label{3}
\end{enumerate}
\end{prop}

\begin{proof}
Fix $n$ and $\delta$ for the duration of the proof. We will write $\phi_i$ for $\phi_i^{\delta,n}$, that is, for the restriction of $\phi^{\delta,n}$ to $i$--chains.

Let $f_1,g_1,\mathfrak{h}$ be the functions appearing in Proposition~\ref{diameter bound}. Without loss of generality, we can assume that $g_1$ is weakly decreasing in its second argument and that $f_1$ is weakly increasing.

Define $n_i:=f_1^{(i)}(n)$ and $m_i:=f_1^{(i)}(n+1)$, recalling that $f_1^{(i)}$ denotes the $i$--fold composition of $f_1$. Define $\delta_{k+1}:=\delta$ and inductively $\delta_{i-1}:=g_1(\delta_i,m_{i-1})$; the fact that we are using $m_{i-1}$ rather than $n_{i-1}$ will only matter in part~(3). Define $r_0=0$ and inductively $r_{i+1}:=\mathfrak{h}(r+2r_i)+r_i$; each $r_i$ is a function of $r$. Note that $n_i$, $\delta_i$ and $r_i$ all increase with $i$.

We now proceed by induction on $0\leq i\leq k$ and construct maps
\[ \phi_i\colon \mscr{C}_i(X_n) \longrightarrow \mscr{C}_i(X_{n_i};\delta_i) \]
so that $\partial\phi_i=\phi_{i-1}\partial$. In addition, if $\s\in \mscr{C}_i(X_n;r)$ is a simplex, then $\supp(\phi_i\s)$ will be contained in the $r_i$--neighbourhood of $\supp(\s)$. See Figure~\ref{fig:Prop4.5_1}.

\begin{figure}[ht]
   \centering
   %\hspace{3cm}
    \includegraphics{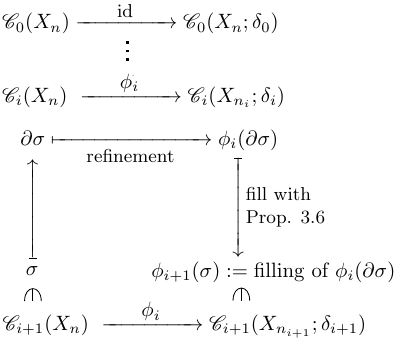}
    \caption{Illustrative diagram of the proof of part (\ref{1}) of Proposition~\ref{arbitrarily fine subdivisions 2}, showing the inductive construction of the maps $\phi_i$.}
    \label{fig:Prop4.5_1}
\end{figure}

In the base step, we simply take $\phi_0$ to be the identity. In the inductive step, suppose that, for some $0\leq i\leq k-1$, we have defined all maps up to $\phi_i$ so that they satisfy the required properties. We want to define $\phi_{i+1}$.

If $\s\in \mscr{C}_{i+1}(X_n)$ is a simplex, then $\phi_i\partial\s$ is a cycle in $\mscr{C}_i(X_{n_i};\delta_i)$, by the inductive hypothesis. Note that $\delta_i\leq g_1(\delta_{i+1},n_i)$, since $m_i\geq n_i$. In addition, if $\s\in \mscr{C}_{i+1}(X_n;r)$, then $\supp(\phi_i\partial\s)$ is contained in the $r_i$--neighbourhood of $\supp(\s)$. In this case, the diameter of $\supp(\phi_i\partial\s)$ is at most $r+2r_i$.

By Proposition~\ref{diameter bound}, the cycle $\phi_i\partial\s$ is the boundary of a chain in $\mscr{C}_{i+1}(X_{n_{i+1}};\delta_{i+1})$ with diameter at most $\mathfrak{h}(r+2r_i)$. We define $\phi_{i+1}\s$ as this chain.

Note that $\supp(\phi_{i+1}\s)\supseteq\supp(\partial\phi_{i+1}\s)=\supp(\phi_i\partial\s)$. We have already observed that the latter is contained in the $r_i$--neighbourhood of $\supp(\s)$. It follows that $\supp(\phi_{i+1}\s)$ is contained in the neighbourhood of $\supp(\s)$ of radius $r_i+\mathfrak{h}(r+2r_i)=r_{i+1}$.

Defining $f_3(\cdot)$ as the function $n\mapsto n_k$ and $H(\cdot)$ as the function $r\mapsto r_k$, this proves parts~(1) and~(2). In fact, in view of the proof of part~(3) below, we should define $f_3(\cdot)$ as the (larger) function $n\mapsto m_k$.

We now prove part~(3). Suppose that we have chain subcomplexes (each generated by a set of simplices)
\begin{align*} 
E_*&\sq \mscr{C}_*(X_{n+1};r), & F_*&\sq \mscr{C}_*(X_n;r),
\end{align*} 
and a chain map $T_*\colon E_*\ra F_*$ such that, for every simplex $\s\in E_*$, the support of $T(\s)$ is contained in the $r$--neighbourhood of $\supp(\s)$.

Define $R_{-1}:=r+H(r)$, then inductively $R_{i+1}:=\mathfrak{h}(r+2R_i)+R_i$. Proceeding by induction on $0\leq i<k$, we will construct a chain homotopy
\[H_i\colon E_i\ra \mscr{C}_{i+1}(X_{m_{i+1}};\delta_{i+1}) \]
so that $\partial H_i+H_{i-1}\partial=\phi_i^{n+1,\delta}-\phi_i^{n,\delta}T_i$. In addition, for every simplex $\s\in E_i$, the support of $H_i\s$ will be contained in the $R_i$--neighbourhood of $\supp(\s)$. See Figure~\ref{fig:Prop4.5_2}.

\begin{figure}
    %\hspace{4cm}
    \includegraphics{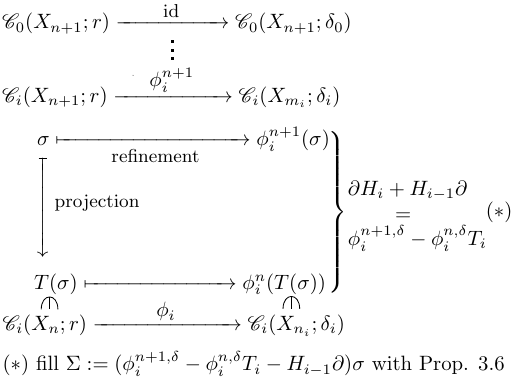}
    \caption{Illustrative diagram of the proof of part (\ref{3}) of Proposition~\ref{arbitrarily fine subdivisions 2}.}
    \label{fig:Prop4.5_2}
\end{figure}

The base step ``$i=-1$'' is clear, taking $H_{-1}=0$. In the inductive step, suppose that, for some $0\leq i<k$, we have defined all maps up to $H_{i-1}$ so that they satisfy the required properties. 

In order to define $H_i$, consider a simplex $\s\in E_i$. Note that
\[\Sigma:=(\phi_i^{n+1,\delta}-\phi_i^{n,\delta}T_i-H_{i-1}\partial)\s\] 
is a cycle in $\mscr{C}_i(X_{m_i};\delta_i)$. (Since the function $g_1(\cdot,\cdot)$ is decreasing in its second argument, the ``$\delta_i$'' appearing in the definition of $\phi^{\delta,n+1}$ is smaller than the one appearing in the definition of $\phi^{\delta,n}$, which is the one used here. Thus, $\Sigma$ is indeed $\delta_i$--fine.) 

In addition, $\supp(\Sigma)$ is contained in the neighbourhood of $\supp(\s)$ of radius $R_{i-1}$, since $R_{i-1}\geq r+H(r)$ by construction. In particular, the diameter of $\supp(\Sigma)$ is at most $r+2R_{i-1}$.

By Proposition~\ref{diameter bound}, 
% here it matters that i<k strictly
we can fill $\Sigma$ by a chain in $\mscr{C}_{i+1}(X_{m_{i+1}};\delta_{i+1})$ with diameter at most $\mathfrak{h}(r+2R_{i-1})$. (Here it finally matters that we used $m_{i-1}$ in the definition of $\delta_i$ at the beginning of the whole proof.) We define $H_i\s$ as this chain which fills $\Sigma$. 

It is clear that $\supp(H_i\s)$ is contained in the neighbourhood of $\supp(\s)$ of radius:
\[\mathfrak{h}(r+2R_{i-1})+R_{i-1}=R_i.\]

This completes the construction of the chain homotopy $H_*$. If $c,c'$ are as in the statement of part~(3), we have $\phi^{\delta,n+1}_ic-\phi^{\delta,n}_ic'=\partial H_ic$. It is clear that $H_ic$ is $\delta$--fine and that its support is contained in the neighbourhood of $\supp(c)$ of radius $R_i$. We can thus define $H'(\cdot)$ as the function $r\mapsto R_k$. 

This concludes the proof of the proposition.
\end{proof}

Proposition~\ref{arbitrarily fine subdivisions 2} allows us to arbitrarily ``refine'' chains, but it completely disregards open covers of $X$. Thus, as a second step towards Proposition~\ref{arbitrarily fine subdivisions}, we will need to be able to fill cycles while remaining within a given open set. This is the content of Proposition~\ref{moreover} below, after which we will finally prove Proposition~\ref{arbitrarily fine subdivisions}.

Before stating Proposition~\ref{moreover}, we record the following observation.

\begin{lem}\label{related homotopy}
Let a cycle $c\in \mscr{C}_*(X_n;\delta)$ be $(\delta,X_n)$--related to a cycle $c'\in \mscr{C}_*(X_n)$, for some $\delta>0$ and $n\geq 1$. Then there exists a chain $d\in \mscr{C}_*(X_n;3\delta)$ such that $\partial d=c-c'$ and $\supp(d)$ is contained in the $\delta$--neighbourhood of $\supp(c)$.
% no restrictions on $i$ needed now
\end{lem}
\begin{proof}
The proof is similar to that of part~(3) of Proposition~\ref{arbitrarily fine subdivisions 2}, but simpler. 

Let $T\colon \mscr{F}_*(c)\ra \mscr{C}_*(X_n)$ be the chain map witnessing the fact that $c$ and $c'$ are $(\delta,X_n)$--related. For each simplex $\s\in \mscr{F}_*(c)$, denote by $\mc{T}(\s)$ the union of the sets $\supp(T\tau)$ where $\tau$ ranges over $\s$ and all its lower-dimensional faces. Thus, $\mc{T}(\s)$ contains $\supp(T\s)$, but it can be larger in general. Nevertheless, $\mc{T}(\s)$ is contained in the $\delta$--neighbourhood of $\supp(\s)$ in $X_n$. In addition, recalling that $c$ is $\delta$--fine, we have $\diam(\supp(\s)\cup\mc{T}(\s))\leq 3\delta$ for every simplex $\s\in \mscr{F}_*(c)$.

Now, for each $i\geq 0$, we inductively construct a chain homotopy 
\[ H_i\colon \mscr{F}_i(c)\longrightarrow \mscr{C}_{i+1}(X_n;3\delta) \]
satisfying $\partial H_i+H_{i-1}\partial={\rm id}-T_i$ and $\supp(H_i\s)\sq\supp(\s)\cup\mc{T}(\s)$ for every simplex $\s\in \mscr{F}_i(c)$.

In the base step ``$i=-1$'', we simply set $H_{-1}=0$. In the inductive step, suppose that, for some $i\geq 0$, we have defined all maps up to $H_{i-1}$ so that they satisfy the required properties.

In order to define $H_i$, consider a simplex $\s\in \mscr{F}_i(c)$ and observe that $\Sigma:=({\rm id}-T_i-H_{i-1}\partial)\s$ is a cycle with support contained in $\supp(\s)\cup\mc{T}(\s)\sq X_n$. Then we simply define $H_i\s$ as the cone over $\Sigma$ from any of its vertices (Definition~\ref{defn:cone}). We obtain $\partial H_i\s=\Sigma$ and $\supp(H_i\s)\sq\supp(\Sigma)\sq\supp(\s)\cup\mc{T}(\s)$. As observed above, this guarantees that $H_i\s$ has diameter at most $3\delta$ and that it is supported in the $\delta$--neighbourhood of $\supp(\s)$. 

Finally, set $d:=Hc$. Since $c$ is a cycle, we have $\partial d=c-Tc=c-c'$. In addition, $d$ is $3\delta$--fine and it is supported in the $\delta$--neighbourhood of $\supp(c)$.
\end{proof}

Recall that $\overline{\mscr{C}}_*(\mc{S})$ denotes the chain complex of reduced discrete chains in $\mc{S}$, where $0$--chains are required to have zero coefficient sum. Also recall the notion of $\mc{O}$--tiny chain from Definition~\ref{defn:tiny/fine}.

\begin{prop}\label{moreover}
For every open cover $\mc{O}$ of $X=\varinjlim X_n$, there exist a refinement $\mc{O}'$ and functions $f_4(n)\geq n$ and $0<g_2(\delta,n)\leq\delta$ such that the following holds.

For $0\leq i<k$, let $c\in\overline{\mscr{C}}_i(X_n)$ be an $\mc{O}'$--tiny, $g_2(\delta,n)$--fine cycle. Then $c=\partial d$ for an $\mc{O}$--tiny, $\delta$--fine chain $d$ supported within $X_{f_4(n)}$.
\end{prop}

The proposition is proved right after the following lemma.

\begin{lem}\label{lem:moreover}
For every open cover $\mc{O}$ of $X$, there exist functions $f_4(n)\geq n$ and $0<g_3(x,\delta,n)\leq\delta$ such that the following holds for all $x\in X$, $\delta>0$ and $n\geq 1$.

There exists an $\mc{O}$--tiny open set $U_x\ni x$ with the following property. If $c$ is a $g_3(x,\delta,n)$--fine, reduced $i$--cycle with $0\leq i<k$ and $c\sq X_n\cap U_x$, then $c=\partial d$ for an $\mc{O}$--tiny $\delta$--fine chain $d$ with $d\sq X_{f_4(n)}$.
\end{lem}

\begin{proof}
We will construct the open sets $U_x$ using the inductive Construction \ref{cover construction} (where the parameters are also determined inductively). Fix the open cover $\mc{O}$ and the point $x\in X$. Let $\ell\geq 1$ be the integer such that $x\in X_\ell-X_{\ell-1}$. Choose an element $O_x\in\mc{O}$ with $x\in O_x$, which will be fixed throughout the proof.

Let $f_1,\mathfrak{h}$ be the functions provided by Proposition~\ref{diameter bound}, and $f_3,H,H'$ those provided by Proposition~\ref{arbitrarily fine subdivisions 2}. Recall that these functions are all bounded below by the identity and, without loss of generality, weakly increasing. Define $f_4:=f_1\o f_3$. Also, recall that we are fixing the visual metric $\rho$ on $X$, although this only induces the topology of compact strata in $X$.

For a sequence of real numbers $\underline\eps=(\eps_j)_{j\geq \ell}$ and an integer $m\geq \ell$, we introduce the notation $\s_m(\underline\eps):=\sum_{j\geq m}\eps_j$. We also define the sets $U_x^{\underline\eps}(m)$ and $U_x^{\underline\eps}$ as in Construction~\ref{cover construction}. Namely, $U_x^{\underline\eps}(\ell)$ is the open $\eps_\ell$--ball around $x$ within $X_\ell$ and, for $m>\ell$, the set $U_x^{\underline\eps}(m)$ is obtained inductively as the open $\eps_m$--neighbourhood of $U_x^{\underline\eps}(m-1)$ within $X_m$. Finally, we have $U_x^{\underline\eps}:=\bigcup_{m\geq \ell}U_x^{\underline\eps}(m)$, which is open in $X$.

In the rest of the proof, we fix a sequence of positive numbers $\underline\eps=(\eps_j)_{j\geq \ell}$ such that the following conditions hold: 
\begin{enumerate}
\setlength\itemsep{.1cm}
\item the closure of $U_x^{\underline\eps}(j)$ is contained in $O_x$, for each $j\geq \ell$;
\item $H(3\s_{\ell+1}(\underline\eps))+\mathfrak{h}(2\eps_\ell+2H(3\s_{\ell+1}(\underline\eps)))<\rho(U_x^{\underline\eps}(\ell),X_{f_4(\ell)}-O_x)$;
\item $\s_{j+1}(\underline\eps)+H'(3\s_j(\underline\eps))<\rho(U_x^{\underline\eps}(j),X_{f_3(j)}-O_x)$, for each $j\geq \ell$.
\end{enumerate}

This is possible because the functions $\mathfrak{h}(r),H(r),H'(r)$ all go to zero for $r\ra 0$. Thus, there exists $\eps_\ell>0$ such that the sequence $(\eps_\ell,0,0,\dots)$ satisfies the above conditions. And, if a sequence $(\eps_\ell,\dots,\eps_m,0,0,\dots)$ satisfies these conditions, then there exists $\eps_{m+1}>0$ such that the sequence $(\eps_\ell,\dots,\eps_m,\eps_{m+1},0,\dots)$ also does.

Now, set $U_x:=U_x^{\underline\eps}$, for our fixed sequence $\underline\eps$. For simplicity, we will simply write $U_x(m)$ and $\s_m$ in the rest of the proof, rather than $U_x^{\underline\eps}(m)$ and $\s_m(\underline\eps)$.

By condition~(1), we have $U_x\sq O_x$. We now show that this open set satisfies the property claimed by the lemma.

Fix $n\geq 1$ and $\delta>0$. Choose $\eta>0$ small enough that:
\begin{enumerate}
\setlength\itemsep{.1cm}
\item[$(1')$] $\eta\vee H(\eta)\leq\delta/3$;
\item[$(2')$] $H(\eta+2\s_{\ell+1})+\mathfrak{h}(2\eps_\ell+2H(\eta+2\s_{\ell+1}))<\rho(U_x(\ell),X_{f_4(\ell)}-O_x)$;
\item[$(3')$] $\s_{j+1}+3H(\eta)+H'(\eta+2\s_j)<\rho(U_x(j),X_{f_3(j)}-O_x)$, for all $\ell\leq j\leq n$.
\end{enumerate}
Conditions~$(2')$ and~$(3')$ can be achieved because $\underline\eps$ was chosen to satisfy conditions~(2) and~(3) above. The resulting value of $\eta$ only depends on $x,n,\delta$, so we can define the function in the statement of the lemma as $g_3(x,\delta,n):=\eta$.

Now, consider an $\eta$--fine reduced $i$--cycle $c$ with $0\leq i<\ell$ and $\supp(c)\sq X_n\cap U_x$. Let us construct a $\delta$--fine chain $d$ with $\partial d=c$ and $\supp(d)\sq O_x\cap X_{f_4(n)}$. It is convenient to set $\delta':=g_1(\delta,f_3(n))\leq\delta$, where $g_1$ is the function provided by Proposition~\ref{diameter bound}.

We begin by constructing a sequence of cycles $(c_j)_{\ell\leq j\leq n-1}$ with $\supp(c_j)\sq U_x(j)$. First, we define $c_{n-1}$ as a $\s_n$--projection of $c$ to $U_x(n-1)$, in the sense of Remark~\ref{rmk:chain_projection}; this exists because $\supp(c)\sq X_n\cap U_x$ is contained in the $\s_{n+1}$--neighbourhood of $U_x(n)$, which is contained in the $\eps_n$--neighbourhood of $U_x(n-1)$. Then, for each $\ell\leq j\leq n-2$, we define inductively $c_j$ as an $\eps_{j+1}$--projection of $c_{j+1}$ to $U_x(j)$.

Note that $\fin(c_j)\leq \eta+2\s_{j+1}$ for each $\ell\leq j\leq n-1$. In addition, $c_j$ is $(\eps_j,X_j)$--related to $c_{j-1}$, and $c$ is $(\s_n,X_{n-1})$--related to $c_{n-1}$.

We now refine these cycles using the chain maps provided by Proposition~\ref{arbitrarily fine subdivisions 2}. Specifically, recall that $\delta'=g_1(\delta,f_3(n))$ and define $c_j':=\phi^{\delta',j}c_j$ for each $\ell\leq j\leq n-1$ and $c':=\phi^{\delta',n}c$. Part~(3) of the proposition guarantees the existence of $\delta'$--fine chains $B_j$ with $\partial B_j=c_{j+1}'-c_j'$ and $\supp(B_j)$ contained in the $H'(\eta+2\s_{j+1})$--neighbourhood of $\supp(c_{j+1})$ within $X_{f_3(j+1)}$, for each $\ell\leq j\leq n-2$. Since $\supp(c_{j+1})\sq U_x(j+1)$, this implies that $\supp(B_j)\sq O_x$, by applying condition~$(3')$ above.

Similarly, we have $\partial B_{n-1}=c'-c_{n-1}'$ and $\supp(B_{n-1})$ is contained in the neighbourhood of $\supp(c)\sq X_n\cap U_x$ of radius  $H'(\eta+2\s_n)$ within $X_{f_3(n)}$. Thus, $\supp(B_{n-1})$ is contained in the $(\s_{n+1}+H'(\eta+2\s_n))$--neighbourhood of $U_x(n)$ within $X_{f_3(n)}$. Hence $\supp(B_{n-1})\sq O_x$, again by condition~$(3')$.

In conclusion, the chain $(B_\ell+\dots+B_{n-1})$ is $\delta'$--fine, supported in $O_x\cap X_{f_3(n)}$ and it satisfies $\partial (B_\ell+\dots+B_{n-1})=c'-c_\ell'$. Thus, we are only left to similarly fill the cycles $c_\ell'$ and $c-c'$.

First, we deal with $c_\ell'$. Observe that $c_\ell'$ is $\delta'$--fine and, by part~(2) of Proposition~\ref{arbitrarily fine subdivisions 2}, it is supported in the $H(\eta+2\s_{\ell+1})$--neighbourhood of $c_\ell$ within $X_{f_3(\ell)}$. Since $\delta'=g_1(\delta,f_3(n))$ and $n\geq \ell$, Proposition~\ref{diameter bound} allows us to fill $c_\ell'$ by a $\delta$--fine chain $B_{\ell-1}$ supported in $X_{f_1(f_3(\ell))}=X_{f_4(\ell)}$ and with $\diam(B_{\ell-1})\leq \mathfrak{h}(\diam(c_\ell'))$. In particular, $\supp(B_{\ell-1})$ is contained in the neighbourhood of $\supp(c_\ell)\sq U_x(\ell)$ in $X_{f_4(\ell)}$ of radius:
\[H(\eta+2\s_{\ell+1})+\mathfrak{h}(2\eps_\ell+2H(\eta+2\s_{\ell+1})),\]
where we have used the fact that $\diam(c_\ell)\leq 2\eps_\ell$, since $c_\ell$ is contained in an $\eps_\ell$--ball. Now, condition~$(2')$ implies that $\supp(B_{\ell-1})$ is contained in $O_x$.

Finally, we fill $c-c'$. Note that $c$ is $\eta$--fine and $(H(\eta),X_{f_3(n)})$--related to $c'$, by part~(2) of Proposition~\ref{arbitrarily fine subdivisions 2}. 

Thus, Lemma~\ref{related homotopy} yields $c-c'=\partial B_n$ for a chain $B_n$ with $\fin(B_n)\leq 3 H(\eta)\leq\delta$ (by condition~$(1')$) 
% here I've used that $\eta\leq H(\eta)$
and $\supp(B_n)$ contained in the $H(\eta)$--neighbourhood of $\supp(c)$ in $X_{f_3(n)}$. Since $\supp(c)$ is contained in the $\s_{n+1}$--neighbourhood of $U_x(n)$, condition~$(3')$ again shows that $\supp(B_n)\sq O_x$.

Summing up, we have $c=\partial (B_{\ell-1}+B_\ell+\dots+B_{n-1}+B_n)$, where each $B_i$ is $\delta$--fine and supported within $X_{f_4(n)}\cap O_x$. This concludes the proof.
\end{proof}

\begin{proof}[Proof of Proposition~\ref{moreover}]
For each $\ell\geq 1$, choose a finite subset $A_\ell\sq X_\ell$ such that $X_\ell\sq\bigcup_{x\in A_\ell}U_x$, where the sets $U_x$ are provided by Lemma~\ref{lem:moreover}. Define $\mc{O}':=\bigcup_{\ell\geq 1}\{ U_x-X_{\ell-1} \mid x\in A_\ell\}$. It is clear that all elements of $\mc{O}'$ are open in $X$, that they cover $X$, and that $\mc{O}'$ refines $\mc{O}$.

Now, define $g_2(\delta,n):=\inf_{x\in A_0\cup\dots\cup A_n} g_3(x,\delta,n)$; since the infimum is taken over a finite set, we have $g_2(\delta,n)>0$. Let $c$ be a reduced $g_2(\delta,n)$--fine $i$--cycle supported within $X_n\cap O'$, for some $O'\in\mc{O}'$. We necessarily have $O'= U_y-X_{\ell-1}$ for some $y\in A_\ell$ with $\ell\leq n$. It follows that $c$ is, in particular, $g_3(y,\delta,n)$--fine, hence Lemma~\ref{lem:moreover} yields an $\mc{O}$--tiny $\delta$--fine chain $d$ with $\supp(d)\sq X_{f_4(n)}$ and $\partial d=c$.
\end{proof}

We are finally ready to prove Proposition~\ref{arbitrarily fine subdivisions}. The argument is almost identical to the one in the proof of Proposition~\ref{arbitrarily fine subdivisions 2}, except that we will use Proposition~\ref{moreover} in place of Proposition~\ref{diameter bound}.

\begin{proof}[Proof of Proposition~\ref{arbitrarily fine subdivisions}]
Fixing $\delta>0$, we will simply write $\phi$ for the chain map $\phi^{\delta}$. We will prove the proposition by induction on the degree $i$ of chains. The key ingredient in the construction of $\phi$ is Proposition~\ref{moreover}, while the chain homotopy will be obtained filling chains trivially via Definition~\ref{defn:cone}.

Set $\overline{\mc{O}_0}:=\mc{O}$. Define inductively refinements $\overline{\mc{O}}_{i+1}<\mc{O}_i'<\mc{O}_i<\overline{\mc{O}}_i$ as follows. First, $\mc{O}_i$ is a super-refinement of $\overline{\mc{O}}_i$, then $\mc{O}_i'$ is the refinement of $\mc{O}_i$ provided by Proposition~\ref{moreover}, and finally $\overline{\mc{O}}_{i+1}$ is a double super-refinement of $\mc{O}_i'$ (as in Remark~\ref{rmk:super_super_refinement}).

Let $f_i$ and $g_i$ be the functions appearing in Proposition~\ref{moreover} applied to the cover $\mc{O}_i$. Define $n_0:=n$ and inductively $n_i=f_{k+1-i}(n_{i-1})$. Define $\delta_k:=\delta$ and inductively $\delta_{i-1}=g_{k+1-i}(\delta_i,n_{i-1})$. Note that both $n_i$ and $\delta_i$ increase with $i$.

We begin by inductively constructing maps
\[ \phi_i\colon \mscr{C}_i(X_n,\mc{O}_{k+1}) \longrightarrow \mscr{C}_i(X_{n_i},\mc{O}_{k+1-i};\delta_i) \]
for $0\leq i\leq k$ so that $\partial\phi_i=\phi_{i-1}\partial$.
For the induction process to work, we additionally require that, for every simplex $\s\in \mscr{C}_i(X_n,\mc{O}_{k+1})$, the set $\supp(\phi_i\s)\cup\supp(\s)$ be $\overline{\mc{O}}_{k+1-i}$--tiny.  See Figure~\ref{fig:Prop4.2_1}
\begin{figure}
    %\hspace{4cm}
    \includegraphics{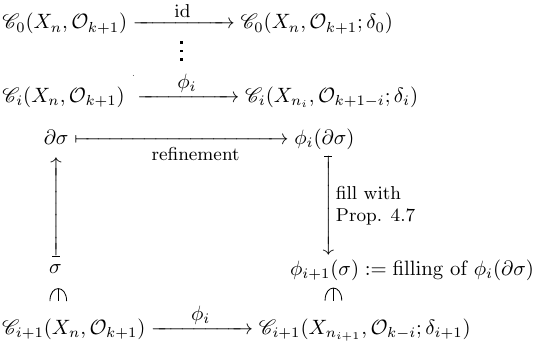}
    \caption{Proof of part~\ref{part1} of Proposition~\ref{arbitrarily fine subdivisions}.}
    \label{fig:Prop4.2_1}
\end{figure}

In the base step ``$i=0$'', we can simply take $\phi_0$ to be the identity. In the inductive step, suppose that, for some $0\leq i\leq k-1$, we are given all maps up to $\phi_i$ so that they satisfy the required properties. In order to define $\phi_{i+1}$, consider a simplex $\s\in \mscr{C}_{i+1}(X_n,\mc{O}_{k+1})$. Note that $\phi_i\partial\s$ is a cycle and lies in $\mscr{C}_i(X_{n_i},\mc{O}_{k+1-i};\delta_i)$.

We claim that $\supp(\phi_i\partial\s)\cup\supp(\s)$ is $\mc{O}_{k-i}'$--tiny. This is clear for $i=0$ since $\phi_0$ is the identity. For $i\geq 1$, let $\tau_0,\dots,\tau_{i+1}$ be the $i$--simplices in $\partial\s$. By our assumptions on $\phi_i$, each set $\supp(\phi_i\tau_j)\cup\supp(\tau_j)$ is contained in some $O_j\in\overline{\mc{O}}_{k+1-i}$. Thus, $O_0,\dots, O_{i+1}$ pairwise intersect, hence they are all contained in some $O'\in\mc{O}_{k-i}'$ by Remark~\ref{rmk:super_super_refinement}, since $\overline{\mc{O}}_{k+1-i}$ was defined as a double super-refinement of $\mc{O}_{k-i}'$. In conclusion, $O'$ contains $\supp(\phi_i\partial\s)\cup\supp(\s)$.

Now, we have shown that $\phi_i\partial\s$ is $\mc{O}_{k-i}'$--tiny, $\delta_i$--fine, and contained in $X_{n_i}$. Recalling that $n_{i+1}=f_{k-i}(n_i)$ and $\delta_i=g_{k-i}(\delta_{i+1},n_i)$, Proposition~\ref{moreover} implies that $\phi_i\partial\s$ (which is reduced as any boundary is reduced) is the boundary of an $\mc{O}_{k-i}$--tiny, $\delta_{i+1}$--fine chain contained in $X_{n_{i+1}}$. We define $\phi_{i+1}\s$ as this chain.

It is clear that we have $\partial\phi_{i+1}=\phi_i\partial$. We are left to show that $\supp(\phi_{i+1}\s)\cup\supp(\s)$ is $\overline{\mc{O}}_{k-i}$--tiny. Recall that $\supp(\phi_i\partial\s)\cup\supp(\s)$ is $\mc{O}_{k-i}'$--tiny, say contained in $O'\in\mc{O}_{k-i}'$. In addition, $\supp(\phi_{i+1}\s)$ is $\mc{O}_{k-i}$--tiny, say contained in $O\in\mc{O}_{k-i}$. Note that $O\cap O'\neq\emptyset$, since 
\[\supp(\phi_{i+1}\s)\supseteq\supp(\partial\phi_{i+1}\s)=\supp(\phi_i\partial\s).\]
Recalling that $\mc{O}_{k-i}'$ is a refinement of $\mc{O}_{k-i}$, which is a super-refinement of $\overline{\mc{O}}_{k-i}$, we conclude that $O\cup O'$ is contained in an element of $\overline{\mc{O}}_{k-i}$. This element contains $\supp(\phi_{i+1}\s)\cup\supp(\s)$, as required. 

Finally, we construct a chain-homotopy $H$. We will inductively define maps
\[ H_i\colon \mscr{C}_i(X_n,\mc{O}_{k+1}) \longrightarrow \mscr{C}_{i+1}(X_{n_i}, \overline{\mc{O}}_{k-i}) \]
for $0\leq i\leq k$ so that $\phi_i-{\rm id}=\partial H_i+H_{i-1}\partial$. Again, for the induction to work, we also require that, for every $\s\in \mscr{C}_i(X_n,\mc{O}_{k+1})$, the set $\supp(H_i\s)\cup\supp(\s)$ be $\overline{\mc{O}}_{k-i}$--tiny. See Figure~\ref{fig:Prop4.2_2}.

In the base step ``$i=0$'', we can just take $H_0$ to be the zero map, since $\phi_0$ is the identity. In the inductive step, suppose that we are given all maps up to $H_i$ so that they satisfy the required properties. In order to define $H_{i+1}$, consider a simplex $\s\in \mscr{C}_{i+1}(X_n,\mc{O}_{k+1})$. Note that:
\[\partial(\phi_{i+1}\s-\s-H_i\partial\s)=\phi_i\partial\s-\partial\s-(\phi_i-{\rm id}-H_{i-1}\partial)\partial\s=0.\]

We claim that the cycle $(\phi_{i+1}\s-\s-H_i\partial\s)$ is $\overline{\mc{O}}_{k-i-1}$--tiny. Recall that $\supp(\phi_{i+1}\s)\cup\supp(\s)$ is contained in some $O\in\overline{\mc{O}}_{k-i}$. In addition, if $\tau_0,\dots,\tau_{i+1}$ are the $i$--simplices in $\partial\s$, there exist $O_j\in\overline{\mc{O}}_{k-i}$ containing $\supp(H_i\tau_j)\cup\supp(\tau_j)$. Since $O,O_0,\dots,O_{i+1}$ are pairwise-intersecting elements of $\overline{\mc{O}}_{k-i}$, their union is contained in some $\overline O\in\overline{\mc{O}}_{k-i-1}$ (for this, note that $\overline{\mc{O}}_{k-i}$ is a double super-refinement of $\overline{\mc{O}}_{k-i-1}$ and apply Remark~\ref{rmk:super_super_refinement}). Now, $\overline O$ contains the support of $(\phi_{i+1}\s-\s-H_i\partial\s)$, as required.

Finally, by Lemma~\ref{lem:cone}, the cycle $(\phi_{i+1}\s-\s-H_i\partial\s)$ is the boundary of the cone over any of its vertices, which is an $(i+2)$--chain contained in $X_{n_{i+1}}\cap\overline O$. We define $H_{i+1}\s$ as this chain. It is clear from the construction that $\partial H_{i+1}+H_i\partial=\phi_{i+1}-{\rm id}$ and that $\supp(H_{i+1}\s)\cup\supp(\s)$ is contained in $\overline O$, hence it is an $\overline{\mc{O}}_{k-i-1}$--tiny set. 

This concludes the whole proof, setting $\mc{O}':=\mc{O}_{k+1}$ and $f_2(n):=n_k$.
\end{proof}

\begin{figure}
    %\hspace{4cm}
    \includegraphics{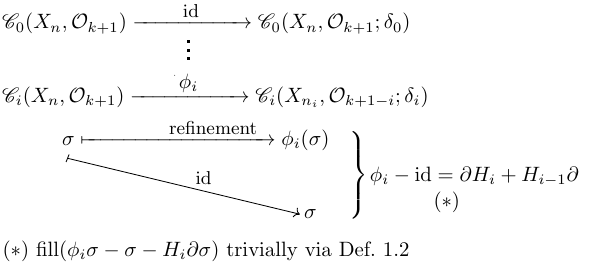}
    \caption{Proof of part~\ref{part2} of Proposition~\ref{arbitrarily fine subdivisions}.}
    \label{fig:Prop4.2_2}
\end{figure}

\subsection{Main result}

We are finally ready to prove Theorem~\ref{thm:vanishing}, which is the main result of this section. It claims that, $\cH^i(X,\mc{R})=\{0\}$ for a principal ideal domain $\mc{R}$ and for $1\leq i<k$, where $X=\partial_*G$.

First, combining Propositions~\ref{diameter bound} and~\ref{arbitrarily fine subdivisions}, we quickly obtain that $X$ satisfies Condition~$(DA_i)$ in low degrees, which was introduced in Subsection~\ref{subsec:super-refinements}.

\begin{prop}\label{prop:cycles_boundaries}
For every open cover $\mc{O}$ of $X$, there exists a refinement $\mc{O}'$ such that, for every $1\leq i<k$, every $\mc{O}'$--fine discrete $i$--cycle is the boundary of an $\mc{O}$--fine discrete chain. 
\end{prop}
\begin{proof}
Let $\lambda_n$ be a Lebesgue number for the cover $\{O\cap X_n \mid O\in\mc{O}\}$ of $X_n$. Let $f_1,g_1,f_2$ be the functions provided by Propositions~\ref{diameter bound} and~\ref{arbitrarily fine subdivisions}. Let $\mc{O}'$ be the refinement of $\mc{O}$ provided by Proposition~\ref{arbitrarily fine subdivisions}. Set $\delta:=g_1(\lambda_{f_1(f_2(n))},f_2(n))$.

Let $c$ be an $\mc{O}'$--fine cycle. By Proposition~\ref{arbitrarily fine subdivisions}, $c$ is $\mc{O}$--homologous to a $\delta$--fine cycle $c'$ with $\supp(c')\sq X_{f_2(n)}$. By Proposition~\ref{diameter bound}, we have $c'=\partial d'$ for a $\lambda_{f_1(f_2(n))}$--fine chain $d'$ with $\supp(d')\sq X_{f_1(f_2(n))}$. In particular, $d'$ is $\mc{O}$--fine, showing that $c$ is the boundary of an $\mc{O}$--fine chain.
\end{proof}

\begin{proof}[Proof of Theorem \ref{thm:vanishing}.]
By Proposition~\ref{prop:limit-refinable}, the topological space $X$ is super-refinable. Proposition~\ref{prop:cycles_boundaries} shows that $X$ satisfies Condition~$(DA_i)$ for $1\leq i<k$, so Corollary~\ref{cor:refinable-cohomology} yields $\cH^i(X,\mc{R})=\{0\}$ in the same range.
\end{proof}

\section{Non-vanishing}\label{sect:non-vanishing}

Throughout Section~\ref{sect:non-vanishing}, we again consider a relatively hyperbolic group $G$, its Morse boundary $X:=\partial_*G$ and the filtration $X=\varinjlim X_n$. Notation and assumptions are exactly the same as in the previous sections (see e.g.\ Subsection~\ref{subsect:ass_3}). In particular, the Bowditch boundary $\mc S$ of $(G, \mc P)$ is homeomorphic to a $(k+1)$--dimensional sphere $S^{k+1}$. 

The goal of this section is to prove the second part of Theorem~\ref{thm:general_intro} (which implies the second part of Theorem \ref{thm:main_intro}). To do so, we show that for any finite set of parabolic points $F$, the rank of the singular homology $H_k(\mc{S}-F,\Z)$ (as an abelian group) is a lower bound to the dimension of $\cH^k(\partial_*G,\mathbb F)$, where $\mathbb F$ is a field of characteristic 0. Recall that the rank of an abelian group $A$, which we denote by $\rk A$, is the maximal cardinality of a linearly independent subset.

The following is the crucial lemma in this section. Roughly speaking, it shows that, if there are many discrete cycles that are not boundaries, then the \v{C}ech cohomology is large. Note however that we have to be rather careful with how fine the various chains are.

\begin{lem}\label{lem:nonvanishing_criterion}
Let $\mathbb F$ be a field of characteristic $0$.
Every open cover $\mc{O}$ of $X$ admits a refinement $\mc{O}'$ such that 
\begin{align*}
    \dim\cH^k(X,\mathbb F)\geq \rk\Big(\{\mc{O'}\text{--fine $k$--cycles}\}/\{\text{$\mc{O'}$--fine boundaries of $\mc{O}$--fine $(k+1)$--chains}\}\Big).
\end{align*}

\end{lem}

\begin{proof}
Let $\mc{O}_2\ll\mc{O}_1\ll\mc{O}$ be a chain of super-refinements and let $\mc{O}'<\mc{O}_2$ be the refinement provided by Proposition~\ref{arbitrarily fine subdivisions}. The following diagram is an overview of the various covers and chain maps that we will introduce during the proof.
\[
\begin{tikzcd}[remember picture]
\mscr{C}_*(X,\mc{O}') & \mscr{C}_*(X,\mc{O}_2) \arrow[r, "g_*"] & C_*(N(\mc{O}_1)) \arrow[r, "g'_*"] & \mscr{C}_*(X,\mc{O}) \\
\mscr{C}_*(X,\mc{U}') \arrow[rr, "h'_*"]  & & C_*(N(\mc{U})) \arrow[u,"h_*"] &
\end{tikzcd}
\begin{tikzpicture}[overlay,remember picture]
\path (\tikzcdmatrixname-2-1) to node[midway,sloped]{$\sq$}
(\tikzcdmatrixname-1-1);
\path (\tikzcdmatrixname-1-1) to node[midway,sloped]{$\sq$}
(\tikzcdmatrixname-1-2);
\end{tikzpicture}
\]

Let $g \colon X\to \mc O_1$ be an $(\mc{O}_2,\mc{O}_1)$--parent map and $g'\colon\mc O_1\to X$ be a child map.  Recalling the chain map $g_*$ from discrete chains to chains in nerves provided by Lemma \ref{lem:family}, we can consider the map $\Phi \colon \{\text{$\mc O'$--fine $k$--cycles}\}\to H_k(N(\mc O_1), \Z)$ given by $c \mapsto [g_*(c)]$.

 If $[g_*(c)] = [g_*(c')]$ there exists some $d\in C_{k+1}( N(\mc O_1))$ with $\partial d = g_*(c) - g_*(c')$. Lemma \ref{lem:family} implies that $g_*'\circ g_* \colon \mscr{C}_*(X, \mc O_2)\to \mscr{C}_*(X, \mc O)$ is a chain map homotopic to the standard inclusion and that $g_*'(d)$ is $\mc O$--fine. Thus $ c - c' = \partial g'_*(d)  + \partial d'$ for some $\mc O$--fine chain $d'$. This shows that $\rm{Ker} (\Phi)\subseteq \{\text{boundaries of $\mc{O}$--fine $(k+1)$--chains}\}$. Define $A = \rm{Im} (\Phi)$.

From now on we identify $H^k(\cdot,\mathbb F)$ with $\hom(H_k(\cdot,\mathbb Z),\mathbb F)$, and denote it simply by $H^k(\cdot)$. We are only left to show that $\dim\cH^k(X,\mathbb F)\geq \rk(A)$.

\smallskip
{\bf Claim.} There exists a linear subspace $V$ of $H^k(N(\mc O_1))$ of dimension $\rk(A)$ and such that, for all $\psi\in V-\{0\}$, we have $\psi|_A\neq 0$.

\smallskip
\emph{Proof of Claim.}
Let $\{x_i\}$ be a maximal linearly independent subset of $A$, and extend it to a maximal linearly independent subset $\{x_i\}\sqcup\{y_j\}$ of $H_k(N(\mc O_1), \Z)$. We claim that we can take $V$ to be the subspace of $H^k(N(\mc O_1))$ of all elements that vanish on all $y_j$. Indeed, we now argue that any map $f\colon\{x_i\}\sqcup\{y_j\}\to \mathbb F$ extends uniquely to a homomorphism $h_f$ in $H^k(N(\mc O_1))$; this gives both the dimension bound and the non-vanishing property from the claim.

To define $h_f$, we note that all $x\in H_k(N(\mc O_1), \Z)$ have a multiple $n_xx$ that lies in $\bar L=\langle \{x_i\}\sqcup\{y_j\} \rangle$ for some integer $n_x\neq 0$, by maximality. By linear independence, $f$ extends to a homomorphism $\bar f$ defined on $\bar L$, which we can then further extend by setting $h_f(x)=\bar f(n_x x)/n_x$. It is readily checked that $h_f$ is a well-defined homomorphism, and the only one that restricts to $f$ on $\{x_i\}\sqcup\{y_j\}$.
{\hfill $\blacksquare$}

\smallskip
We now show that $V$ as in the claim embeds in $\cH^k(X,\mathbb F)$. Indeed, consider any $\alpha\in V-\{0\}$ and consider any refinement $\mc U < \mc O_1$ and any spouse map $h \colon \mc U\to \mc O_1$. We have to show $h^*(\alpha)  \neq 0$ for the induced map $h^*\colon H^k(N(\mc{O}_1))\ra H^k(N(\mc{U}))$. 

Since $\alpha\in V-\{0\}$ there exists a homology class $\tilde{c}\in A$ such that $\alpha(\tilde{c})\neq 0$. By the definition of $A$, there exists a cycle $c\in \mscr{C}_k(X, \mc O')$ such that $[g_*(c)] = \tilde{c}$.
There exists some $n\in \mathbb N$ such that $c\subseteq X_n$. Let $\mc U'$ be a super-refinement of $\mc U$ that is also a refinement of $\mc O'$ and let $h'\colon X\to \mc U$ be a $(\mc U', \mc U)$--parent map. Let $f$ be the function $f_2$ from Proposition \ref{arbitrarily fine subdivisions} applied to the open cover $\mc O_2$. 

Let $\delta$ be a Lebesgue number for the open cover $\{U'\cap X_{f(n)}\mid U'\in \mc U'\}$ of the compact space $X_{f(n)}$. By Proposition \ref{arbitrarily fine subdivisions} applied to $\mc O_2$, $\delta$ and $c$, there exists a $\mc O_2$--fine, $\delta$--fine cycle $c'\subseteq X_{f(n)}$ which is $\mc O_2$--homologous to $c$. Since $g_* \colon \mscr{C}_*(X, \mc O_2) \to C_*(N(\mc O_1))$ is a chain map, we have that $[g_*(c)] = [g_*(c')]$. Also, by the choice of $\delta$, the cycle $c'$ is $\mc U'$--fine. 

Note that both $g$ and $h\circ h'$ are $(\mc{U}',\mc{O}_1)$--parent maps, so Lemma~\ref{lem:family} shows that their induced maps $\mscr{C}_k(X,\mc{U}')\ra C_k(N(\mc{O}_1))$ are chain-homotopic. Defining $c_{\mc U} = h'_*(c')\in C_*(N(\mc U))$, we get that $[h_*(c_{\mc U})] = [g_*(c)]=\tilde c$ in $H_k(N(\mc{O}_1))$. Finally, observe that $h^*(\alpha)([c_{\mc U}]) = \alpha(\tilde c)\neq 0$, as required.
\end{proof}

We are now ready to prove that $\cH^k(X,\mathbb{F})$ is infinite dimensional (Theorem \ref{thm:general_intro}(2)).

\begin{thm}
\label{thm:non_vanishing}
    Let $G$ be a finitely generated, relatively hyperbolic group with virtually nilpotent peripherals and Bowditch boundary homeomorphic to a sphere $S^{k+1}$, for some $k\geq 1$, and with at least one peripheral subgroup. Let $\mathbb F$ be a field of characteristic $0$. Then $\cH^k(\partial_* G,\mathbb F)$ is infinite-dimensional.
\end{thm}

\begin{proof}
It suffices to show that, for every finite set of parabolic points $F\sq\mc{S}$, we have the inequality $\rk H_k(\mc{S}-F,\Z) \leq  \dim \cH^k(X,\mathbb F)$, where $X=\partial_*G$.

Let $\mc U$ be a convex cover of $\mc {S} - F$ as in Lemma \ref{lem:convex_cover}. Set $\mc{O}:=\mc{U}|_X$. To define this cover restriction, recall that there is a natural continuous and injective map $X\ra\mc{S}$ (though it is not an open map).

Let $\mc{O}'<\mc{O}$ be the refinement provided by Lemma~\ref{lem:nonvanishing_criterion}. Finally, let $N$ be the integer associated to $F$ by Proposition~\ref{prop:representing_homology}, and let $\delta$ be a Lebesgue number for the restriction $\mc{O}'|_{X_N}$.

It suffices to show that $H_k(\mc{S}-F,\Z)$ embeds into
\[\{\mc{O}'\text{--fine $k$--cycles}\}/\{\text{$\mc{O'}$--fine boundaries of $\mc{O}$--fine $(k+1)$--chains}\}\] (and hence has smaller rank) 
and apply Lemma~\ref{lem:nonvanishing_criterion}. By Proposition~\ref{prop:representing_homology}, every element of $ H_k(\mc{S}-F,\Z)$ can be represented by a straight singular cycle $c$ such that $\disc(c)$ is $\delta$--fine and contained in $X_N$, and hence $\mc {O}'$--fine. Since $c$ is a straight cycle and $\mc U$ is a convex cover, not only $\disc(c)$ but also $c$ itself is $\mc U$--fine. Lemma \ref{lem:convex_cover} concludes the proof.
\end{proof}

\bibliography{mybib}

\begin{thebibliography}{BCW14}

\bibitem[Bau71]{Baumslag}
Gilbert Baumslag.
\newblock {\em Lecture notes on nilpotent groups}.
\newblock Regional Conference Series in Mathematics, No. 2. American
  Mathematical Society, Providence, R.I., 1971.

\bibitem[BCW14]{Barcelo-Capraro-White}
H\'{e}l\`ene Barcelo, Valerio Capraro, and Jacob~A. White.
\newblock Discrete homology theory for metric spaces.
\newblock {\em Bull. Lond. Math. Soc.}, 46(5):889--905, 2014.

\bibitem[BH99]{BH}
Martin~R. Bridson and Andr\'{e} Haefliger.
\newblock {\em Metric spaces of non-positive curvature}, volume 319 of {\em
  Grundlehren der mathematischen Wissenschaften [Fundamental Principles of
  Mathematical Sciences]}.
\newblock Springer-Verlag, Berlin, 1999.

\bibitem[BM91]{BestvinaMess}
Mladen Bestvina and Geoffrey Mess.
\newblock The boundary of negatively curved groups.
\newblock {\em J. Amer. Math. Soc.}, 4(3):469--481, 1991.

\bibitem[Bow12]{Bow:rel-hyp}
B.~H. Bowditch.
\newblock Relatively hyperbolic groups.
\newblock {\em Internat. J. Algebra Comput.}, 22(3):1250016, 66, 2012.

\bibitem[BT82]{Bott-Tu}
Raoul Bott and Loring~W. Tu.
\newblock {\em Differential forms in algebraic topology}, volume~82 of {\em
  Graduate Texts in Mathematics}.
\newblock Springer-Verlag, New York-Berlin, 1982.

\bibitem[CCS19]{charney2019complete}
Ruth Charney, Matthew Cordes, and Alessandro Sisto.
\newblock Complete topological descriptions of certain morse boundaries.
\newblock {\em to appear in Groups Geom. Dyn.}, 2019.

\bibitem[CD19]{CD:stable}
Matthew Cordes and Matthew~Gentry Durham.
\newblock Boundary convex cocompactness and stability of subgroups of finitely
  generated groups.
\newblock {\em Int. Math. Res. Not. IMRN}, (6):1699--1724, 2019.

\bibitem[CH17]{CH:stable_asdim}
Matthew Cordes and David Hume.
\newblock Stability and the {M}orse boundary.
\newblock {\em J. Lond. Math. Soc. (2)}, 95(3):963--988, 2017.

\bibitem[CM19]{CM:topology}
Christopher~H. Cashen and John~M. Mackay.
\newblock A metrizable topology on the contracting boundary of a group.
\newblock {\em Trans. Amer. Math. Soc.}, 372(3):1555--1600, 2019.

\bibitem[Cor17]{C:Morse}
Matthew Cordes.
\newblock Morse boundaries of proper geodesic metric spaces.
\newblock {\em Groups Geom. Dyn.}, 11(4):1281--1306, 2017.

\bibitem[CS15]{CS:contracting}
Ruth Charney and Harold Sultan.
\newblock Contracting boundaries of {$\rm CAT(0)$} spaces.
\newblock {\em J. Topol.}, 8(1):93--117, 2015.

\bibitem[DK18]{Drutu-Kapovich}
Cornelia Dru\c{t}u and Michael Kapovich.
\newblock {\em Geometric group theory}, volume~63 of {\em American Mathematical
  Society Colloquium Publications}.
\newblock American Mathematical Society, Providence, RI, 2018.
\newblock With an appendix by Bogdan Nica.

\bibitem[Gro87]{Gromov-hyp}
Misha Gromov.
\newblock Hyperbolic groups.
\newblock In {\em Essays in group theory}, volume~8 of {\em Math. Sci. Res.
  Inst. Publ.}, pages 75--263. Springer, New York, 1987.

\bibitem[GS18]{GS:smallcanc}
Dominik Gruber and Alessandro Sisto.
\newblock Infinitely presented graphical small cancellation groups are
  acylindrically hyperbolic.
\newblock {\em Ann. Inst. Fourier (Grenoble)}, 68(6):2501--2552, 2018.

\bibitem[Isb59]{Isbell}
John~R. Isbell.
\newblock On finite-dimensional uniform spaces.
\newblock {\em Pacific J. Math.}, 9:107--121, 1959.

\bibitem[Kar94]{Karidi}
Ron Karidi.
\newblock Geometry of balls in nilpotent {L}ie groups.
\newblock {\em Duke Math. J.}, 74(2):301--317, 1994.

\bibitem[LS14]{Leininger-Schleimer}
Christopher~J. Leininger and Saul Schleimer.
\newblock Hyperbolic spaces in {T}eichm\"{u}ller spaces.
\newblock {\em J. Eur. Math. Soc. (JEMS)}, 16(12):2669--2692, 2014.

\bibitem[MS20]{Mackay-Sisto}
John~M. Mackay and Alessandro Sisto.
\newblock Quasi-hyperbolic planes in relatively hyperbolic groups.
\newblock {\em Ann. Acad. Sci. Fenn. Math.}, 45(1):139--174, 2020.

\bibitem[MS24]{MS2}
John Mackay and Alessandro Sisto.
\newblock Maps between relatively hyperbolic spaces and between their
  boundaries.
\newblock {\em Trans. Amer. Math. Soc.}, 377(2):1409--1454, 2024.

\bibitem[Sch12]{Schultz}
Reinhard Schultz.
\newblock Notes on topological dimension theory.
\newblock Available at
  https://math.ucr.edu/~res/miscpapers/top-dimension-theory.pdf, 2012.

\bibitem[Sis13]{Sisto-distformrelhyp}
Alessandro Sisto.
\newblock Projections and relative hyperbolicity.
\newblock {\em Enseign. Math. (2)}, 59(1-2):165--181, 2013.

\bibitem[Sis16]{S:hypemb}
Alessandro Sisto.
\newblock Quasi-convexity of hyperbolically embedded subgroups.
\newblock {\em Math. Z.}, 283(3-4):649--658, 2016.

\bibitem[Spa66]{Spanier}
Edwin~H. Spanier.
\newblock {\em Algebraic topology}.
\newblock McGraw-Hill Book Co., New York-Toronto, Ont.-London, 1966.

\bibitem[Zar96]{Zarichnyi}
Michael Zarichnyi.
\newblock Regular linear operators extending metrics: a short proof.
\newblock {\em Bull. Polish Acad. Sci. Math.}, 44(3):267--269, 1996.

\bibitem[Zbi22]{Zbi2}
Stefanie Zbinden.
\newblock Morse boundaries of 3-manifold groups.
\newblock {\em arXiv preprint arXiv:2212.08368}, 2022.

\bibitem[Zbi23]{Zbi1}
Stefanie Zbinden.
\newblock Morse boundaries of graphs of groups with finite edge groups.
\newblock {\em J. Group Theory}, 26(5):969--1002, 2023.

\end{thebibliography}
\bibliographystyle{alpha}

\end{document}